\documentclass[11pt, reqno]{amsart}
\usepackage[T1]{fontenc}
\usepackage[latin1]{inputenc}
\usepackage[english]{babel}
\usepackage{amsthm}
\usepackage{amssymb,amsmath,bbm,dsfont}
\usepackage{enumitem}
\usepackage{enumitem}
\usepackage[all]{xy}
\usepackage{mathrsfs}
\usepackage[left=0.0cm,right=0.0cm,top=1.7cm,bottom=1.7cm]{geometry}
\usepackage{hyperref}
\usepackage{verbatim}

\makeatletter
\newcommand{\mylabel}[2]{#2\def\@currentlabel{#2}\label{#1}}
\makeatother

\theoremstyle{plain}
\newtheorem{theorem}{Theorem}[section]

\newtheorem{lemma}[theorem]{Lemma}
\newtheorem{proposition}[theorem]{Proposition}
\newtheorem{prop}[theorem]{Proposition}
\newtheorem{corollary}[theorem]{Corollary}
\newtheorem{cor}[theorem]{Corollary}

\theoremstyle{definition}
\newtheorem{definition}[theorem]{Definition}

\newtheorem{conjecture}[theorem]{Conjecture}
\newtheorem{notation}[theorem]{Notation}

\theoremstyle{remark}
\newtheorem{remark}[theorem]{Remark}

\theoremstyle{plain}
\newtheorem{thma}{Theorem}

\newtheorem{thmb}[thma]{Theorem}
\newtheorem{thmc}[thma]{Theorem}
\newtheorem{cord}[thma]{Corollary}

\def\Z{{\bf Z}}
\def\C{{\bf C}}

\def\R{{\bf R}}

\def\H{{H}}

\def\SL{{\mathbf{SL}}}

\def\A{{\bf A}}

\def\epsilon{\varepsilon}
\def\det{\mathrm{det}}
\def\G{\mathbf{G}}
\def\H{\mathbf{H}}

\def\GSp{{\mathbf{GSp}}}
\def\Sp{{\mathbf{Sp}}}
\def\GU{{\mathbf{GU}}}

\def\SO{{\mathbf{SO}}}
\def\U{{\mathbf{U}}}
\def\GSpin{{\mathbf{GSpin}}}
\def\GO{{\mathbf{GO}}}

\def\GSO{{\mathbf{GSO}}}
\def\GL{\mathbf{GL}}

\usepackage[usenames,dvipsnames]{color}

\usepackage{tikz-cd}
\title{On Periods and $L$-functions for $\GL_4 \times \GL_2$}

\author{Antonio Cauchi}
\address{Antonio Cauchi\newline University College Dublin\\ School of Mathematics and Statistics\\ Science Centre - South, Belfield, Dublin 4, Dublin, Ireland}
\email{antonio.cauchi@ucd.ie}

\author{Armando Gutierrez Terradillos}
\address{Armando Gutierrez Terradillos\newline 
Department of Mathematics, Aarhus University, Ny Munkegade 118, DK-8000, Aarhus, Denmark}
\email{armangute@math.au.dk}

\makeatletter
\def\@tocline#1#2#3#4#5#6#7{\relax
  \ifnum #1>\c@tocdepth 
  \else
    \par \addpenalty\@secpenalty\addvspace{#2}%
    \begingroup \hyphenpenalty\@M
    \@ifempty{#4}{%
      \@tempdima\csname r@tocindent\number#1\endcsname\relax
    }{%
      \@tempdima#4\relax
    }%
    \parindent\z@ \leftskip#3\relax \advance\leftskip\@tempdima\relax
    \rightskip\@pnumwidth plus4em \parfillskip-\@pnumwidth
    #5\leavevmode\hskip-\@tempdima
      \ifcase #1
       \or\or \hskip 1em \or \hskip 2em \else \hskip 3em \fi%
      #6\nobreak\relax
    \dotfill\hbox to\@pnumwidth{\@tocpagenum{#7}}\par
    \nobreak
    \endgroup
  \fi}
\makeatother
\usepackage{hyperref}

\begin{document}

\begin{abstract}
We give a new integral representation of the $\wedge^2  \otimes {\rm std}_2$ $L$-function of generic cusp forms on $\GL_4 \times \GL_2$ and $\GU_{2,2}\times \GL_2$. In the former case, we use it to prove a relation between its central $L$-value and the generalized Shalika period. Exploiting the theta correspondence for $(\GL_4,\GL_4)$, we further establish a relation between the central value of the $L$-function attached to the strongly tempered spherical pair $(\GL_4 \times \GL_2,\GL_2 \times \GL_2)$ and its corresponding period. In the case of cusp forms on $\GL_4 \times \GL_2$ that are unramified everywhere, our formulas give new evidence towards conjectures of Wan--Zhang and of Gan--Gross--Prasad for $\GSpin_6 \times \GSpin_3$. \end{abstract}

\maketitle
\selectlanguage{english}

\setcounter{tocdepth}{1}
\tableofcontents
 
\section{Introduction}

Periods of automorphic forms are ubiquitous in Number Theory and Arithmetic Geometry. The non-triviality of automorphic periods is often expected to detect phenomena of Langlands functoriality and the non-vanishing of special values of $L$-functions. Exploring this ``paradigm''  has opened rich avenues for arithmetic applications - for instance, the results \cite{WaldsToricPeriods} or \cite{IchinoTriple} have played an important role in applications to the Birch--Swinnerton-Dyer conjecture. The first results and conjectures aiming at a systematic study of the relationship between automorphic periods and $L$-functions are \cite{GrossPrasadFirst} and \cite{GGPOriginal}. They set out a conjectural framework where diagonal periods for orthogonal and unitary groups detect the non-vanishing of central standard $L$-values. 

Recently, the works of Sakellaridis--Venkatesh \cite{SakellaridisVenkatesh} and Ben--Zvi--Sakellaridis--Venkatesh \cite{BZSV} proposed a relative Langlands duality in the context of spherical varieties and hyperspherical Hamiltonian spaces. These works led to a new fertile ground of conjectural formulae, which relate automorphic periods and $L$-functions associated to these spaces. 
To this extent, the works of Wan--Zhang \cite{Wan:Zhang:Spherical:Periods} and Mao--Wan--Zhang \cite{MWZ}, to cite a few, have shed light on various aspects of these phenomena.

We now recall the definition of a family of period integrals that naturally appear in these settings. Let $G$ be a reductive group defined over a number field $F$, and $H$ a closed connected subgroup of $G$. 
Let $A_G$ be the maximal split subtorus of the center of $G$ and $A_{G,H} = A_G \cap H$. For a cusp form $\varphi$ on $G(\A_F)$ and a (possibly trivial) character $\chi$ on $[H]$ such that the product of their central characters is trivial on $A_{G,H}$, we define the period of $\varphi$ on $(H,\chi)$ to be
\[ \mathcal{P}_{H,\chi}( \varphi) := \int_{A_{G,H}(\A_F)H(F)\setminus H(\A_F)} \varphi(h) \chi^{-1}(h)\,dh. \] If $\pi$ is a cuspidal automorphic representation of $G(\A_F)$, then the restriction of $\mathcal{P}_{H,\chi}$ to $\pi$ defines a functional \begin{align*}\mathcal{P}_{H,\chi}|_{\pi} \in {\rm Hom}_{H(\A_F)}(\pi, \chi). \end{align*}
One of the most fundamental problems in the relative Langlands program is to establish the relation between $\mathcal{P}_{H,\chi}|_{\pi}$ and special values of a certain automorphic $L$-function $L(s,\pi,\rho_{X})$ of $\pi$, attached to the $G$-variety $X=(H,\chi) \backslash G$. When $X=(H,\chi) \backslash G$ is a spherical $G$-variety (\textit{i.e.} a Borel subgroup of $G$ has a dense orbit in $X$), Sakellaridis showed that $L(s,\pi,\rho_X)$ is determined by the colors of the spherical variety $X$, see \cite{sakellaridis2013spherical}.

\subsection{The linear and Shalika periods}
Among the most classical examples are the intimately related linear and Shalika periods on $\GL_{2n}$, which play a central role in the study of exterior square $L$-functions and in establishing fertile cases of Langlands functoriality. The former corresponds to the period $\mathcal{P}_{\GL_n \times \GL_n}$, given by integrating an automorphic form of $\GL_{2n}$ over the Levi $L=\GL_n \times \GL_n$ of the Siegel parabolic $P=LN$  of $\GL_{2n}$, while the latter concerns the period $\mathcal{S}_{\psi_S}$, obtained by integrating an automorphic form of $\GL_{2n}$ over the subgroup $S=\GL_n\, N$, with $\GL_n$ embedded diagonally in $L$,  against a generic character $\psi_S$ of $[N]$.
Locally, these periods are at most of multiplicity one and non-zero for representations whose $L$-parameter is conjugate to the symplectic subgroup of $\GL_{2n}$ (see, for instance, \cite[Theorem 1.1]{Chen-Gan} and references therein). 

We now recall the global results on the non-vanishing of these periods in the case of  $\GL_{4}$, which may be viewed as motivating instances of the branching problems studied in this paper. Let $F$ be a number field and let us denote by $\A$ its adele ring. 

\begin{theorem}[{\cite{JacquetShalikaExterior}, \cite{BumpFriedberg}}]\label{JS_BF}
    Let $\pi$ be a unitary cuspidal automorphic representation of $\GL_4(\A)$. \begin{itemize}
        \item[(i)] ${\mathcal{S}_{\psi_S}}{|_\pi} \not \equiv 0$ if and only if the (partial) exterior square $L$-function $L(s,\pi, \wedge^2)$ has a pole at $s=1$. 
    \item[(ii)] $\mathcal{P}_{\GL_2 \times \GL_2}|_{\pi} \not \equiv 0$ if and only if the standard $L$-function $L(\frac{1}{2},\pi, {\rm std}_4) \ne 0$ and the (partial) exterior square $L$-function $L(s,\pi, \wedge^2)$ has a pole at $s=1$.
    \end{itemize}
\end{theorem}

\noindent Crucially, $L(\frac{1}{2},\pi, {\rm std}_4) \ne 0$ if and only if $\pi$ participates (non-trivially) in the theta correspondence for the dual pair $(\GL_4,\GL_4)$. Under this assumption, the two periods are indeed equivalent (see, for instance, \cite[\S 7]{Chen-Gan} where the similar case for unitary groups is treated).

Consider the embedding
\begin{align*}
     \GL_2\times \GL_2&\to \GL_4\times \GL_2 ,\\
    (g_1, g_2) &\mapsto \left( \mathrm{diag}(g_1,g_2),g_2\right),
\end{align*}
where $\mathrm{diag}(g_1,g_2) = \left(\begin{smallmatrix}
        g_1& \\ &g_2\end{smallmatrix}\right)$. We would like to emphasize that Theorem \ref{JS_BF} is a consequence of the study of period integrals for representations of $\GL_4(\A) \times \GL_2(\A)$ restricted to $\GL_2(\A)\times \GL_2(\A)$. Indeed, both \emph{(i)} and \emph{(ii)}, proved respectively in \cite{JacquetShalikaExterior} and \cite{BumpFriedberg}, are obtained from the study of the (generalized) Shalika period ${\mathcal{S}_{\psi_S}}{|_{\pi \otimes \sigma}}$ and the (generalized) linear period $\mathcal{P}_{\GL_2 \times \GL_2}|_{\pi \otimes \sigma}$, where $\sigma$ is a mirabolic Eisenstein series on $\GL_2$. The former calculates $L(s,\pi, \wedge^2)$, while the latter represents $L(\frac{1}{2},\pi, {\rm std}_4)L(s,\pi, \wedge^2)$, where $s$ denotes the complex parameter of the Eisenstein series $\sigma$. Then Theorem \ref{JS_BF} is essentially obtained by taking the residue at $s=1$ of both period integrals.
\\

It is therefore natural to consider a variant of these branching problems in which the Eisenstein series $\sigma$ is replaced by an arbitrary cuspidal automorphic representation of $\GL_2$. Speculations in this direction already appear in \cite[\S 8]{Chen-Gan}, and are further supported by the Gan--Gross--Prasad conjecture and the broader perspective of the relative Langlands program. On the one hand, the quotient
\[\mathbf{X}_{\rm lin} := (\GL_2\times \GL_2)\setminus \GL_4 \times \GL_2\]
defines a strongly tempered spherical variety (see \cite[Table 1]{Wan:Zhang:Spherical:Periods}). The spherical variety $\mathbf{X}_{\rm lin}$ determines the automorphic period $\mathcal{P}_{\GL_2\times \GL_2}$ and the representation $\rho_{\mathbf{X}_{\rm lin}}$ of $^L(\GL_4 \times \GL_2)$, given by $$\rho_{\mathbf{X}_{\rm lin}} := \wedge^2 \otimes {\rm std}_2 \oplus {\rm std}_4 \oplus {\rm std}_4^\vee,$$
where ${\rm std}_n$  and ${\rm std}_n^\vee$ denote the standard representation of $\GL_n(\C)$ and its dual, respectively. Building on \cite{SakellaridisVenkatesh}, the works \cite{Wan:Zhang:Spherical:Periods} and \cite{Pollack:Wan:Zydor} computed the local relative characters for the pair $(\GL_4(F_v) \times \GL_2(F_v),\GL_2(F_v) \times \GL_2(F_v))$ at unramified places $v$ (with $v$ a place of $F$) and determined the local multiplicities of the local avatars of the automorphic period $\mathcal{P}_{\GL_2\times \GL_2}$.  In particular, they showed that the summation of the multiplicities is always equal to 1 on every local tempered Vogan $L$-packet.

On the other hand, the quotient $$\mathbf{X}_{\rm Sha} := (S,\psi_S)\setminus \GL_4 \times \GL_2$$
defines a spherical variety, with associated period ${\mathcal{S}_{\psi_S}}$ and representation $\rho_{\mathbf{X}_{\rm Sha}}$ of $^L(\GL_4 \times \GL_2)$, given by $$\rho_{\mathbf{X}_{\rm Sha}} := \wedge^2 \otimes {\rm std}_2.$$
We observe that, using the exceptional isomorphism $\GSpin_3 \simeq \GL_2$ and double covering map $\GSpin_6 \to \GL_4$, the generalized Shalika period ${\mathcal{S}_{\psi_S}}$ realizes the Gross--Prasad (Bessel) model on $\GSpin_6 \times \GSpin_3$ for representations which descend to $\GL_4 \times \GL_2$. 

We now record the expectations regarding the non-vanishing of both automorphic periods for $\mathbf{X}_{\rm Sha}$ and $\mathbf{X}_{\rm lin}$, which can be seen as the natural generalization of Theorem \ref{JS_BF}.

\begin{conjecture}[{\cite{GGPOriginal}, \cite[\S 17]{SakellaridisVenkatesh}, \cite[Conjecture 1.9]{Wan:Zhang:Spherical:Periods}}]\label{conj:Wan:Zhang:+:Bessel:GGP}
Let $\pi$ and $\sigma$ be cuspidal automorphic representations of $\GL_4$ and $\GL_2$, respectively, with trivial product of central characters.
\begin{enumerate}
\item $L(\frac{1}{2},\pi \otimes \sigma,\rho_{\mathbf{X}_{\rm Sha}})\neq 0$ if and only if there exists a quaternion algebra $D/F$ such that the generalized Shalika period $\mathcal{S}_{D,\psi_S}|_{\pi_D \otimes \sigma_D} \not \equiv 0$. 
\item $L(\frac{1}{2},\pi \otimes \sigma,\rho_{\mathbf{X}_{\rm lin}})\neq 0$ if and only if there exists a quaternion algebra $D/F$ such that the generalized linear period $\mathcal{P}_{D^\times \times D^\times}|_{\pi_D \otimes \sigma_D} \not \equiv 0$.
\end{enumerate}
Moreover, in each (1) and (2), if any of the conditions holds, there exists a unique $D$ for which the period is non-zero.
\end{conjecture}
\noindent Here, for a given quaternion algebra $D/F$, we have denoted $\pi_D$, resp. $\sigma_D$, the automorphic representation of $\GL_2(D)$, resp. $D^\times$, related to $\pi$, resp. $\sigma$, via the Jacquet--Langlands correspondence. 

\noindent While Conjecture \ref{conj:Wan:Zhang:+:Bessel:GGP}(1) is an instance of the global Gan--Gross--Prasad conjectures to $\GSpin$ groups, Conjecture \ref{conj:Wan:Zhang:+:Bessel:GGP}(2) was formulated as a \emph{weak global conjecture} by Wan--Zhang in \cite{Wan:Zhang:Spherical:Periods}\footnote{The authors in \cite{Wan:Zhang:Spherical:Periods} formulated weak global conjectures as well as their Ichino--Ikeda type refinements for 10 strongly tempered spherical varieties, with $\mathbf{X}_{\rm lin}$ one of them.}.

Using \cite[Theorem 5.11]{PanYan} and \cite[Corollary 1.3(6)]{Pollack:Wan:Zydor}, the latter holding under certain technical conditions, one deduces that \begin{align} \mathcal{S}_{\psi_S}|_{\pi \otimes \sigma} \not \equiv 0 \Longrightarrow L(\tfrac{1}{2},\pi \otimes \sigma,\rho_{\mathbf{X}_{\rm Sha}})\neq 0,\label{resultPanYan} \\ \mathcal{P}_{\GL_2\times \GL_2}|_{\pi \otimes \sigma} \not \equiv 0 \Longrightarrow L(\tfrac{1}{2},\pi \otimes \sigma,\rho_{\mathbf{X}_{\rm Sha}})\neq 0,\label{resultPWZ}
\end{align}
thus giving evidence towards one direction of Conjecture \ref{conj:Wan:Zhang:+:Bessel:GGP} in the case of $D$ split. We observe that to further deduce the non-vanishing of $L(\tfrac{1}{2}, \pi \otimes \sigma, \rho_{\mathbf{X}_{\rm lin}})$  from \eqref{resultPWZ}, one needs to compare the linear and Shalika periods. In particular, if $\pi$ is in the image of the theta correspondence for the dual reductive pair $(\GL_4,\GL_4)$, both periods are equivalent (see Lemma \ref{Shalika:Equals:Period}). This shows that the non-vanishing of $\mathcal{P}_{\GL_2\times \GL_2}|_{\pi \otimes \sigma}$ implies that $L(\tfrac{1}{2},\pi \otimes \sigma,\rho_{\mathbf{X}_{\rm lin}})\neq 0$.

\noindent The main focus of the present work is the considerably more delicate converse direction of \eqref{resultPanYan} and \eqref{resultPWZ}.

\subsection{Main Theorems}
 To simplify the exposition, in this subsection we let the ambient group $G$ be equal to $\GL_4 \times \GL_2$. Let $\pi \otimes \sigma$ be a cuspidal automorphic representation of $G$, with central character satisfying $\omega_\pi \, \omega_\sigma = 1$, and consider their $\wedge^2 \otimes {\rm std}_2$ $L$-function.  We introduce a new integral representation of this $L$-function, which we describe in \S \ref{Section:2:intro}. In Theorem \ref{Final:Theorem:Central:Value}\footnote{We note that Theorem \ref{Final:Theorem:Central:Value} also holds for globally generic cuspidal automorphic representations on $\GU_{2,2} \times \GL_2$.}, using this and the Siegel--Weil formula, we obtain a period relation between the central (partial) $L$-value $L^S(\tfrac{1}{2},\pi\otimes\sigma,\wedge^2\otimes\mathrm{std}_2)$ and the automorphic period $\mathcal{P}_{\GL_2 \times_{\rm det} \GL_2}|_{\pi\otimes\sigma}$, in the spirit of the relation established by Ginzburg--Rallis \cite{GinzburgRallis} for the exterior cube $L$-function on $\GL_6$. Here, $\GL_2 \times_{\rm det} \GL_2$ is embedded into $G$ via the composition of the map $$\GL_2 \times_{\rm det} \GL_2 \hookrightarrow \GSp_4  \times_{\nu,{\rm det}} \GL_2,$$ given in \eqref{Embedding:On:GSp4} together with the natural embedding of $\GSp_4$ inside $\GL_4$. We expect the period $\mathcal{P}_{\GL_2 \times_{\rm det} \GL_2}$ to have several arithmetic applications, since it has a cohomological interpretation. Moreover, despite $\GL_2 \times_{\rm det} \GL_2$ being non-spherical in $G$, it admits a central extension by the spherical subgroup $\GL_2 \times \GL_2$. In \S \ref{Shalika_period_subsec}, we exploit this fact to show that the period $\mathcal{P}_{\GL_2 \times_{\rm det} \GL_2}|_{\pi\otimes\sigma}$ unfolds to $\mathcal{S}_{\psi_S}|_{\pi \otimes \sigma}$, allowing us to relate the non-vanishing of the central $L$-value to the generalized Shalika period. 

 When the representations are unramified at every place, Theorem \ref{Final:Theorem:Central:Value} takes the following simple form.

\begin{thma}[{Theorem \ref{Final:Theorem:Shalika}}]\label{Thma}
Let $\pi$ and $\sigma$ be cuspidal automorphic representations of $\GL_4$ and $\GL_2$ respectively, which are unramified everywhere and such that the product of their central characters is trivial. Then,
        \[L(\tfrac{1}{2},\pi \otimes \sigma,\rho_{\mathbf{X}_{\rm Sha}})\neq 0 \Longleftrightarrow \mathcal{S}_{\psi_S}|_{\pi \otimes \sigma} \not \equiv 0,  \]
        where $L(s, \pi \otimes \sigma, \wedge^2 \otimes  \mathrm{std}_2)$ denotes the completed $L$-function. 
\end{thma}

The discrepancy between Theorem \ref{Thma} and Conjecture \ref{conj:Wan:Zhang:+:Bessel:GGP}(2) seems to rely on the quotient \begin{align}\label{eq_ref_discrepancy_spectral}
    \frac{L(\frac{1}{2},\pi \otimes \sigma,\rho_{\mathbf{X}_{\rm lin}})}{L(\frac{1}{2},\pi\otimes\sigma,\rho_{\mathbf{X}_{\rm Sha}})} = L(\tfrac{1}{2}, \pi , {\rm std}_4)L(\tfrac{1}{2}, \pi , {\rm std}_4^\vee).
\end{align}
This ratio turns out to be non-zero if and only if $\pi$ participates (non-trivially) in the theta correspondence for the dual pair $(\GL_4,\GL_4)$. One of the advantages of the use of the theta lifting is that it produces various period relations by using combinations of seesaw identities and Siegel--Weil formulae. We exploit this principle to prove a relation between the linear period $\mathcal{P}_{ \GL_2 \times \GL_2}$ and the partial $L$-value $L^S( \tfrac{1}{2}, \pi \otimes \sigma, \rho_{\mathbf{X}_{\rm lin}})$, provided that we have certain control on the local zeta integrals at places in $S$ (see Corollary \ref{Main:Corollary:Spherical:1} and Remark \ref{Main:remark:after:MTSP}). In the case of representations that are unramified at every place, the result takes the following elegant form. 

\begin{thmb}[{Corollary \ref{Main:Corollary:Spherical:2}}]\label{Thmb}
    Let $\pi$ and $\sigma$ be cuspidal automorphic representations of $\GL_4$ and $\GL_2$ respectively, which are unramified everywhere and such that the product of their central characters is trivial. Then,
    \[L( \tfrac{1}{2}, \pi \otimes \sigma, \rho_{\mathbf{X}_{\rm lin}}) \neq 0 \Longrightarrow {\mathcal{P}_{ \GL_2 \times \GL_2}}|_{\pi \otimes \sigma} \not \equiv 0.\] In particular, if we further assume that \begin{itemize}
        \item[(i)] $\pi$ and $\sigma$ have trivial central characters,
        \item[(ii)] $L(\frac{3}{2},\pi \otimes \sigma,\wedge^2\otimes\mathrm{std}_2)\neq 0$, which is always the case if $\pi$ and $\sigma$ are tempered,
    \end{itemize} then,
\[L( \tfrac{1}{2}, \pi \otimes \sigma, \rho_{\mathbf{X}_{\rm lin}}) \neq 0 \Longleftrightarrow {\mathcal{P}_{ \GL_2 \times \GL_2}}|_{\pi \otimes \sigma} \not \equiv 0.\]
\end{thmb}
 \noindent In the second part of the statement, we assume hypotheses \emph{(i)} and \emph{(ii)} to ensure that \eqref{resultPWZ} (\emph{i.e.} \cite[Corollary 1.3(6)]{Pollack:Wan:Zydor}) holds. 
 
 One of the crucial ingredients in the proof of Theorem \ref{Thmb} is the theta correspondence for $(\GL_4,\GL_4)$, which allows us to take into account the discrepancy \eqref{eq_ref_discrepancy_spectral}. This is because $\pi$ is in the image of this theta correspondence if and only if the central value of the completed $L$-function $L(z, \pi , {\rm std}_4) $ is non-zero. The explicit use of this theta correspondence allows us to ``upgrade'' our zeta integral to a two complex variable zeta integral in $s$ and $z$, which calculates $$ L(s,\pi\otimes\sigma,\wedge^2\otimes\mathrm{std}_2)L(z, \pi , {\rm std}_4),$$
 and let us connect the central $L$-value $L( \tfrac{1}{2}, \pi \otimes \sigma, \rho_{\mathbf{X}_{\rm lin}})$ to the spherical period  ${\mathcal{P}_{ \GL_2 \times \GL_2}}|_{\pi \otimes \sigma}$. Thus, our method can be viewed as parallel to that of \cite{BumpFriedberg} in the proof of Theorem~\ref{JS_BF}(2). We give a more detailed sketch of the proofs of Theorems \ref{Thma} and \ref{Thmb} in \S \ref{subsec_sketch_proofs} below.

 We conclude this subsection with a few remarks outlining how Theorems \ref{Thma} and \ref{Thmb} might be extended. Firstly, we expect that Theorem \ref{Final:Theorem:Central:Value} and Corollary \ref{Main:Corollary:Spherical:1} can be used to prove Conjecture \ref{conj:Wan:Zhang:+:Bessel:GGP} in full generality. In order to do that, we would require an explicit formula of the Jacquet--Langlands lift of $\pi$ to inner forms of $\GL_4$ in the style of the Shimizu--Jacquet--Langlands formula for $\GL_2$ proved by Shimizu \cite{Shimizu:JL} (see Remark \ref{remark_on_spl} for more details).

\noindent We also expect that the proof of Theorem \ref{Thmb} might be adapted to the case of $\GU_{2,2} \times \GL_2$, but with substantial technical differences due to the dissimilarity between type $\mathrm{I}$ and $\mathrm{II}$ dual reductive pairs. We hope to return to these questions in future work.

\subsection{New integral representations}\label{Section:2:intro}
Fix $E/F$ a quadratic field extension of $F$ defining the quasi-split unitary group $\GU_{2,2}$. The starting point of our treatise is a new integral representation of the $\wedge^2  \otimes {\rm std}_2$ $L$-function on generic cusp forms on $\GU_{2,2}\times \GL_2$ and $\GL_4 \times \GL_2$. 

Let $\pi$, resp. $\sigma$, be cuspidal automorphic representations of either $\GL_4$ or $\GU_{2,2}$, resp. $\GL_2$. Building on the works \cite{CauchiGutiMVI} and \cite{CauchiGuti}, where we studied the relation between automorphic periods and the value at $s=1$ of the $\wedge^2$ $L$-function of $\GL_4$ and $\GU_{2,2}$, we consider 
\[Z(s,\varphi,f_s) = \int_{\GSp_4(F) Z_{\GSp_4}(\A)\setminus \GSp_4(\A)} \!\!\!\!\!E^*_{P_{(1,2)}}(g,s, f_s)\varphi(g)dg,\]
where $\varphi$ is a cusp form in the space of $\pi$ and $E^*_{P_{(1,2)}}(g,s, f_s)$ is a normalized non-degenerate Eisenstein series constructed from $\sigma$ relative to the Klingen parabolic $P_{(1,2)}$. We can interpret $Z(s,\varphi,\phi_s)$ as a degenerate Gross--Prasad period for $(\mathbf{GSpin}_6,\mathbf{GSpin}_5)$. As shown in Proposition \ref{unfolding}, the integral unfolds to the Whittaker models of $\pi$ and $\sigma$. Hence, using the Casselman--Shalika formula, we prove the following.

\begin{thmc}[Theorem \ref{zetaintegralfinalthm}]\label{Thmc}
    Let $\varphi = \otimes_v\varphi_v$ and $f_{s} = \otimes_v f_{s,v}$ be pure tensors. Let $S$ be a finite set of places of $F$ containing the archimedean places and the ramified places for $\varphi$, $f_s$, and $E/F$ (if working with $\GU_{2,2}$). Then we have
    \[Z(s,\varphi,f_s) = L^S(s,\pi \otimes \sigma, \wedge^2 \otimes  \mathrm{std}_2)\cdot\prod_{v\in S}Z(s,W_{\pi_v},W_{\sigma_v,s}),\]
    where $Z(s,W_{\pi_v},W_{\sigma_v,s})$ is the local zeta integral introduced in Corollary \ref{Eulerexp}. 
\end{thmc}

In Theorem \ref{Theorem_Arch_computation_split}, we calculate the local zeta integral at any archimedean place $v$ of $F$ in the case of unramified representations of $\GL_4(F_v) \times \GL_2(F_v)$, relating it to the archimedean $L$-factor $L(s,\pi_v \otimes \sigma_v, \wedge^2 \otimes  \mathrm{std}_2)$. Here, we crucially employ the formulae of spherical Whittaker functions given by \cite{IshiiStade} and \cite{IshiiMultivariate}. Our result implies that $Z(s,\varphi,f_s)$ computes the completed $\wedge^2 \otimes  \mathrm{std}_2$ $L$-function of cuspidal automorphic representations of $\GL_4 \times \GL_2$, which are unramified everywhere. It is also worth mentioning here that, in Proposition \ref{Proposition:Poles:L}, we use $Z(s,\varphi,f_s)$ to deduce that $L^S(s,\pi \otimes \sigma, \wedge^2 \otimes  \mathrm{std}_2)$ can have at most a simple pole at $s=1$. Using a  Siegel--Weil formula of \cite{KudlaRallis}, we relate the existence of this pole to a certain automorphic period (Corollary \ref{cor_pole_SW1}).

We would like to remark that, on its own, the zeta integral $Z(s,\varphi,f_s)$ might be used to address motivic and Iwasawa-theoretic questions. Indeed, since $\GSp_4$ and $\GU_{2,2} \times \GL_2$ have Shimura varieties, our integral might lead to the construction of new Euler systems, built from certain cohomology classes in the Eisenstein cohomology of the Shimura variety of $\GSp_4$.  To this end, in \cite{CauchiGutiArchGU}, we compute the archimedean zeta integral for generic discrete series on $\GU_{2,2}(\R) \times \GL_2(\R)$.

For the rest of the subsection, $\pi$ is a cuspidal automorphic representation of $\GL_4$. In \cite{Watanabe:Global:Theta}, the author considers a family of automorphic lifts \begin{align*} \mathcal{A}_{\rm cusp}([\GL_4])\times \mathcal{S}(M_{4}(\A))&\to \mathcal{A}([\GL_4]),\\(\varphi, \eta)&\mapsto \varphi^{z}_{\eta},\end{align*} depending on a complex parameter $z$, which is exactly the theta lift for $(\GL_4,\GL_4)$ when $z=0$. For a fixed $z \in \C$, we let $\Theta^z(\pi)$ be the space spanned by the automorphic forms $\varphi^{z}_{\eta}$, with $\varphi$ in the space of $\pi$. Watanabe shows that  $\varphi^z_f$ is a (possibly zero) cusp form whenever $\mathrm{Re}(z)$ is big enough (see \cite[p. 707]{Watanabe:Global:Theta}). We thus define
\[J(s,z,\varphi,f_s, \eta) := \int_{{\GSp_4}(F) Z_{\GSp_4}(\A)\setminus {\GSp_4}(\A)} \!\!\!\!\!\! E^*_{P_{(1,2)}}(g,s, f_s)\varphi_\eta^z(g)| {\rm det}(g)|_\A^{-z} dg,\]
when the integral converges absolutely. We remark that, when $\pi$ is in the image of the theta correspondence from $\GL_4$, we have the equality $$J(s,0,\varphi,f_s, \eta) = Z(s,\varphi^0_\eta,f_s).$$ Combining Theorem \ref{THM_two_var_integral} and Lemma \ref{J_ram_lemma}, we prove the following.

\begin{cord}\label{Cord} Let $\varphi = \otimes_v\varphi_v$, $\eta= \otimes_v \eta_v$, and $f_{s} = \otimes_v f_{s,v}$ be pure tensors. Let $S$ be a finite set of places of $F$ containing the archimedean places and the ramified places for $\varphi$, $\eta$, and $f_s$. We have the equality 
    $$J(s,z,\varphi,f_s, \eta) = L(z+ \tfrac{1}{2},\pi,{\rm std}_4) L^S(s,\pi \otimes \sigma,\wedge^2 \otimes \mathrm{std}_2) \prod_{v \in S} J^\star(s,z,\varphi_v,f_{s,v},\eta_v),$$where $J^\star(s,z,\varphi_v,f_{s,v},\eta_v)$ denotes the local zeta integral introduced in Lemma \ref{J_ram_lemma}.
\end{cord}

The proof of Corollary \ref{Cord} follows from the local computations of Theorem \ref{Thmc} and the formulae of the Whittaker model of $\varphi^z_\eta$ given in \cite{Watanabe:Global:Theta}. 

\subsection{Sketch of the proofs of Theorems \ref{Thma} and \ref{Thmb}}\label{subsec_sketch_proofs}
These proofs extend the methods developed in \cite{WaldsToricPeriods}, \cite{KudlaHarrisJacquet}, and \cite{GinzburgRallis}. From now on, we assume that $\pi$ and $\sigma$ are cuspidal automorphic representations of $\GL_4$ and $\GL_2$, which are unramified everywhere and such that the product of their central characters is trivial. Although the arguments and results in the manuscript apply more generally, these assumptions simplify the exposition while preserving the essential ideas. 

The strategy is summarized by the following diagram. Solid arrows represent equivariant maps between non-trivial representations. Dotted arrows, instead, either represent $\C$-valued functionals or indicate that the corresponding functional unfolds to the period that the arrow points to.

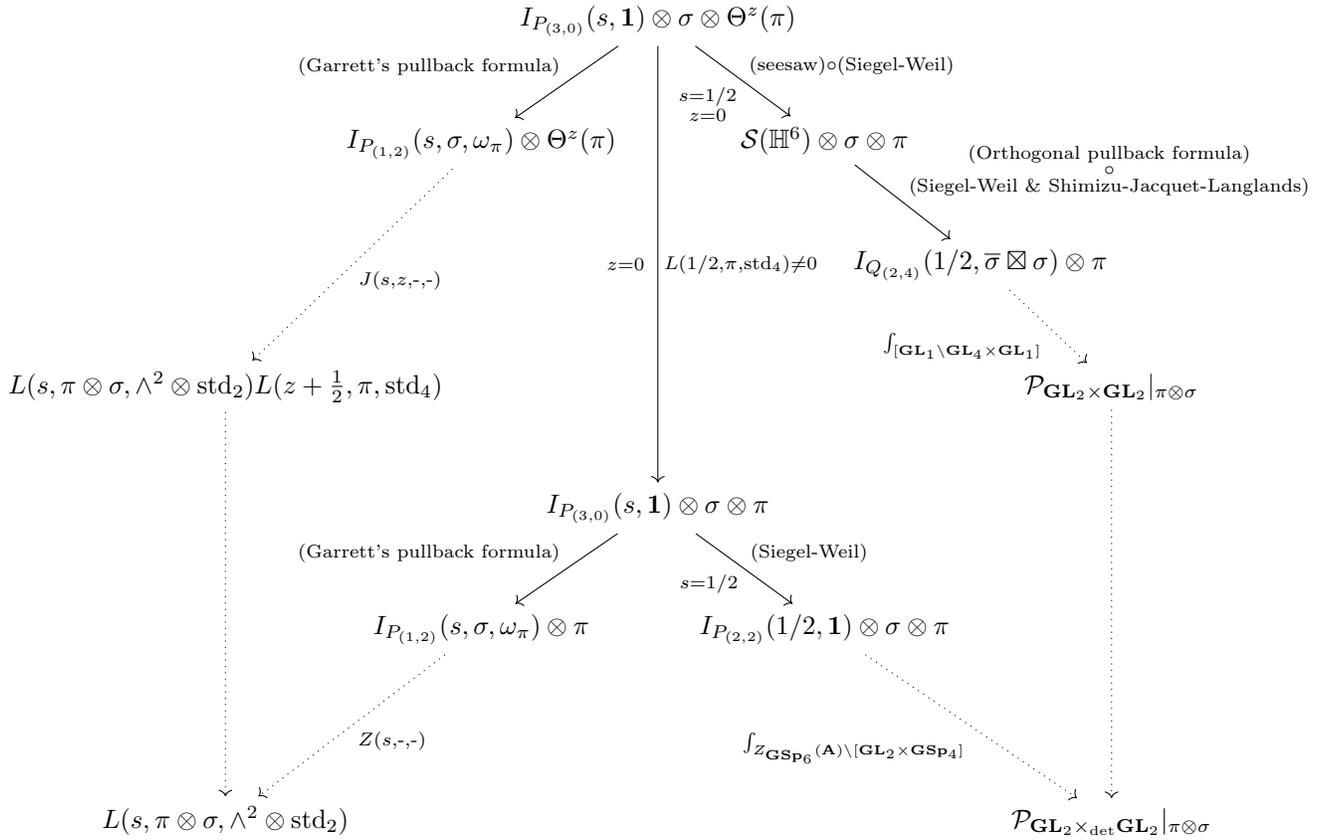
\begin{small}
\begin{equation*}
\begin{tikzcd}[row sep=large, column sep={-1.5em,-1em,0.5em,-1em,-1.5em}]
& & I_{P_{(3,0)}}(s,\mathbf{1})\otimes \sigma \otimes \Theta^z(\pi)\arrow[dddd, "L(1/2\text{,}\pi\text{,}\mathrm{std}_4) \ne 0","z = 0\text{ } "']  \arrow[ld, swap, "(\text{Garrett's pullback formula})"]  \arrow[dr, "\text{(seesaw)}\circ\text{(Siegel-Weil)}", "\substack{s = 1/2\\ z = 0}"']& & &
\\
& I_{P_{(1,2)}}(s,\sigma,\omega_{\pi})\otimes \Theta^z(\pi)\arrow[ddl, dotted, "J(s\text{,}z\text{,-,-})"] &&  \mathcal{S}(\mathbb{H}^6)\otimes \sigma \otimes \pi \arrow[dr, "\substack{\text{(Orthogonal pullback formula)}\\ \circ\\\text{(Siegel-Weil \& Shimizu-Jacquet-Langlands)}}"] & &
\\
&&& & I_{Q_{(2,4)}}(1/2,\overline{\sigma}\boxtimes \sigma)\otimes \pi \arrow[rd, dotted,swap, "\int_{[\mathbf{GL}_1\setminus \mathbf{GL}_4\times\mathbf{GL}_1]}"] &
\\
L(s,\pi\otimes\sigma,\wedge^2\otimes \mathrm{std}_2)L(z+\tfrac{1}{2},\pi,\mathrm{std}_4) \arrow[dddd, dotted]&&& &&\mathcal{P}_{\mathbf{GL}_2\times\mathbf{GL}_2}|_{\pi\otimes\sigma} \arrow[dddd, dotted]
\\
& & I_{P_{(3,0)}}(s,\mathbf{1})\otimes \sigma \otimes \pi
  \arrow[rd, "\text{(Siegel-Weil)}", "s = 1/2"']
  \arrow[ld, swap, "(\text{Garrett's pullback formula})"]
&  & &
\\
& I_{P_{(1,2)}}(s,\sigma,\omega_{\pi})\otimes \pi \arrow[ddl, dotted, "Z(s\text{,-,-})"]&
& I_{P_{(2,2)}}(1/2, \mathbf{1})\otimes \sigma \otimes \pi
  \arrow[dr, dotted,  swap, "\int_{Z_{\mathbf{GSp}_6}(\A)\setminus [\mathbf{GL}_2\times \mathbf{GSp}_4]}"] &  &
\\
 &&& &\mathcal{P}_{\GL_2 \times_{\rm det} \GL_2} |_{\pi \otimes \sigma}  \arrow[dr, dotted] &\\
 L(s,\pi\otimes\sigma,\wedge^2\otimes \mathrm{std}_2)
&&& & &\mathcal{S}_{\psi_S}|_{\pi \otimes \sigma} 
\end{tikzcd}
\end{equation*}
\end{small}

Here, the $I_{P_{(m,2(n-m))}}(-)$'s denote the induced representations of $\GSp_{2n}(\A)$ defined in \eqref{induction:Klingen}, \eqref{induction:Siegel} and \eqref{Induction:22}, and $I_{Q_{(2,4)}}(s,\overline{\sigma}\boxtimes \sigma)$ denotes the induced representation of $\GSO_{\mathbb{H}^4}(\A)$ defined in \eqref{induction:orthogonal}. 

 The proof of Theorem \ref{Thma} follows from the \emph{bottom floor} of the diagram, while the \emph{higher floor} will lead to Theorem \ref{Thmb}, as we now explain.  Start with factorizable data \[\phi_s\otimes\Psi\otimes \varphi \in  I_{P_{(3,0)}}(s,\mathbf{1})\otimes \sigma \otimes \pi,\] and assume that $\phi_{1/2}$ lies in the image of the theta correspondence of the trivial representation of $\mathbf{O}(\mathbb{H}^2)$ for the dual reductive pair $(\mathbf{O}(\mathbb{H}^2),\Sp_6)$. Associated to $\phi_s$, one can define the (normalized) degenerate Siegel Eisenstein series $E_{P_{(3,0)}}^*(g, s,\phi_s)$ on $\GSp_6$. One can then consider the integral 
\[Z'(s,\varphi, \Psi, \phi_s) := \int_{Z_{\GSp_6}(\A)\backslash [\GL_2\times_{{\rm det},\nu} \GSp_4]} E_{P_{(3,0)}}^*(\iota( g_1  , g_2), s,\phi_s) \Psi(g_1 ) \varphi(g_2)\, dg_1\,dg_2,\]
where we have fixed an embedding $\iota: \GL_2\times_{{\rm det},\nu} \GSp_4 \hookrightarrow \GSp_6$. In \S \ref{Section:Garrett:Pullback}, we extend Garrett's pullback formula to the case of similitude groups to produce a $\GSp_4(\A)$-equivariant map $$p_{\Psi}:I_{P_{(3,0)}}(s,\mathbf{1})\otimes \sigma \to   I_{P_{(1,2)}}(s,\sigma,\omega_{\pi}),$$
which let us show, in Corollary \ref{CorPullbackZeta}, that 
$Z'(s,\varphi, \Psi, \phi_s) = Z(s, \varphi, p_\Psi(\phi_s)).$
On the one hand, thanks to Theorem \ref{Thmc}, these integrals represent the $\wedge^2 \otimes {\rm std}_2$ $L$-function on $\GL_4 \times \GL_2$. On the other hand, in \S\ref{subsec:SW:Similitudes} we adapt the (intermediate) Siegel--Weil formula of \cite[Theorem 6.12]{KudlaRallis} to the case of similitude groups. Concretely, we construct a section $F_s(\phi_s)\in I_{P_{(2,2)}}(s,\mathbf{1})$ such that $$E_{P_{(3,0)}}^*(g, 1/2,\phi_{1/2}) \doteq \mathrm{Res}_{s = 1/2}E_{P_{(2,2)}}(g, s, F_s(\phi_{s})),$$ where $\doteq$ means $=$ up to an explicit non-trivial constant and where $E_{P_{(2,2)}}(g, s, F_s(\phi_{s}))$ is a degenerate Eisenstein series on $\GSp_6$ associated to the maximal parabolic $P_{(2,2)}$. Via an unfolding computation, we use this equality to prove that if $Z'(1/2,\varphi, \Psi, \phi_{1/2})$ is non-zero, the \emph{non-spherical} period $\mathcal{P}_{\GL_2 \times_{\rm det} \GL_2}|_{\pi\otimes \sigma}$ is not identically zero (see Theorem \ref{Final:Theorem:Central:Value} and Corollary \ref{Main:Corollary:Non:Spherical}). Combining these facts with the results of \S \ref{Shalika_period_subsec}, we obtain Theorem \ref{Thma}.
 \\

In order to access the linear period associated to the spherical pair $(\GL_4 \times \GL_2, \GL_2 \times \GL_2)$, we use the two-variable integral $J(s,z,\varphi,\Psi,\phi_s,\eta)$. Recall that, thanks to Corollary \ref{Cord}, it computes the product of completed $L$-functions 
$L(s,\pi\otimes\sigma,\wedge^2\otimes \mathrm{std}_2)L(z+\tfrac{1}{2},\pi,\mathrm{std}_4)$. 

In \S\ref{Two:SeeSaw:Section}, we establish a seesaw identity associated to the diagram  
\begin{align*}
\xymatrix{
    \mathbf{GSO}_{\mathbb{H}^6} \ar@{-}[dd] \ar@{-}[ddr]  & \GSp_4 \times  \GL_4  \ar@{-}[ddl] \ar@{-}[dd] \\ \\ 
   \mathbf{GSO}_{\mathbb{H}^2} \times \GL_4  &  \GSp_4.
} 
\end{align*}
In Proposition \ref{FinalPropSct83}, we use this and the regularized Siegel--Weil formula of \cite{KudlaRallis} to write $J(1/2,0,\varphi,\Psi,\phi_{1/2},\eta)$ as an integral involving the theta function for the dual reductive pair $(\mathbf{O}(\mathbb{H}^6),\Sp_4)$ and the integral studied by Shimizu. At this point, we use Shimizu--Jacquet--Langlands \cite{Shimizu:JL} and the Siegel--Weil formula proved by \cite{YamanaSingular} to express $J(1/2,0,\cdots)$ as a zeta integral, which involves the degenerate Siegel Eisenstein series on $\GSO_{\mathbb{H}^6}$. As a last technical ingredient, in \S\ref{subsec:Pullback:Orthogonal},  we prove a pullback formula for Eisenstein series on orthogonal groups, relating the degenerate Siegel Eisenstein series on $\GSO_{\mathbb{H}^6}$ to a non-degenerate Eisenstein series on $\GSO_{\mathbb{H}^4}$, constructed from $ \overline{\sigma}\boxtimes\sigma$. Via an unfolding process, we realize the spherical period $\mathcal{P}_{\GL_2 \times \GL_2}$ as an inner integral of $J(1/2,0,\cdots)$. This and Corollary \ref{Cord} let us show that 
    \[L( \tfrac{1}{2}, \pi \otimes \sigma, \rho_{\mathbf{X}_{\rm lin}}) \neq 0 \Longrightarrow {\mathcal{P}_{ \GL_2 \times \GL_2}}|_{\pi \otimes \sigma} \not \equiv 0,\] 
concluding the proof of Theorem \ref{Thmb}.

\subsection*{Acknowledgements}
The authors express their sincere gratitude to Wee Teck Gan and Paul Nelson for their guidance and insightful advice. We especially thank Wee Teck Gan for his comments on a first draft of this manuscript and for generously explaining how an hypothesis previously imposed in Theorem \ref{Thmb} could be removed. A.C.'s research in this publication was conducted with the financial support  of the JSPS Postdoctoral Fellowship for Research in Japan and of Taighde \'{E}ireann -- Research Ireland under Grant number IRCLA/2023/849 (HighCritical).  A.G.T. was supported by the Morningside Center of Mathematics (CAS) and by a research grant (VIL54509) from VILLUM FONDEN.

\section{Preliminaries}

\subsection{Notation}\label{subsec_notation} \begin{itemize}
    \item $F$ denotes a number field, with ring of integers $\mathcal{O}$, and denote by $\A$ its adeles. Let $E/F$ be a quadratic field extension, which we write $E=F(\delta)$, and let $\bar{\bullet}$ denote the non-trivial automorphism of order 2 of $E/F$. 
    \item We let $\zeta_F$ denote the Dedekind zeta function of $F$. When $v$ is a finite place of $F$, we let $F_v$ be the $v$-adic completion of $F$ with ring of integers $\mathcal{O}_v$; we let $q_v$ be the size of the quotient of $\mathcal{O}_v$ modulo its maximal ideal $\mathfrak{p}_v$ and fix a uniformizer  $\varpi_v$ of $\mathcal{O}_v$. Accordingly, we normalize the norm $|\cdot |$ of $F_v$ so that $|\varpi_v| = q_v^{-1}$. 
    \item We fix a non-trivial additive character $\psi: F \backslash \A \to \C$.
    \item We let $I_n'$  be the $n \times n$ anti-diagonal matrix with all entries 1 and denote 
$$ J_{2n}=\left( \begin{smallmatrix}  & I_n' \\ -I_n' & \end{smallmatrix} \right).$$
\item For a reductive group $G$ over $F$, let $\mathcal{A}_{\rm cusp}([G])$ denote the space of cuspidal automorphic forms on the automorphic quotient $[G] = G(F)\backslash G(\A)$.
\item The symbol ${\rm Ind}_*^\bullet(-)$ denotes normalized induction from $*$ to $\bullet$ of $-$. 
\item Given two reductive groups $G_1,G_2 \in \{\GSp_{2n},\GSO_V,\GO_V,\GU_{n,n}\}$ with similitude characters $\nu_i : G_i \to \GL_1$, we denote by $G_1 \boxtimes G_2$ the fiber product $G_1 \times_{\nu_1,\nu_2} G_2$, i.e. 
$$G_1 \boxtimes G_2 = \{(g_1,g_2) \,:\, \nu_1(g_1)=\nu_2(g_2) \}.$$
\end{itemize}

\subsection{Symplectic and Unitary Groups}\label{sec:groups} 
For every positive integer $n$, define $\GU_{n,n}$ the group scheme over $F$ given by  
\[ \GU_{n,n}(R) = \{ (g,m_g) \in \GL_{2n,F}(R \otimes_{F}  E ) \times \GL_{1,F}(R) :{}^t\bar{g} J_{2n} g = m_g J_{2n}\}, \]
where $R$ denotes any $F$-algebra.  Denote by $\nu:\GU_{n,n} \to \GL_{1,F}, \; g \mapsto m_g$ the similitude character. Inside $\GU_{n,n}$, we have the subgroup \[\GSp_{2n}(R)=\{ (g,m_g) \in \GL_{2n,F}(R) \times \GL_{1,F}(R) :{}^tg J_{2n} g = m_g J_{2n}\}, \]
which embeds into $\GU_{n,n}$ via the natural inclusion. When $n=1$, $\GSp_{2} = \GL_2$. We again denote by $\nu : \GSp_{2n} \to \GL_{1,F}, \; g \mapsto m_g$ the similitude character. The kernel of $\nu$ is denoted by $\Sp_{2n}$. 
\begin{remark}
    
These groups arise as the generic fiber of groups schemes over $\mathcal{O}$ by replacing $E$ and $F$ by $\mathcal{O}_E$ and $\mathcal{O}$ in the definition. For any such group $G$, we denote by $G(\mathcal{O})$ the group of $\mathcal{O}$-points of the resulting scheme. This choice fixes the maximal compact subgroups that will be used along the paper. 
\end{remark}

For any group $G \in \{\GL_n, \GSp_{2n}, \GU_{n,n}\}$, we let $B_G = T_G \, U_G$ be the upper triangular Borel subgroup of $G$, with $T_G$ the maximal diagonal torus in $G$.

\subsubsection{Maximal parabolic subgroups}\label{subsec:Parabolics:GSp}

We denote by $W_{2n}$ the standard representation of $\GSp_{2n}$ with basis 
  $$\{e_1,e_2,\dots ,e_n,f_n, \dots ,f_2,f_1\}.$$ 
Then the subspace $W_{2n}^+$ generated by $\{e_1,e_2,\dots ,e_n\}$ is a maximal isotropic subspace of $W_{2n}$; we also let $W_{2n}^- := \langle f_n, \dots ,f_2,f_1 \rangle$, so that $$W_{2n} = W_{2n}^+ \oplus W_{2n}^- .$$  

For any $0 \leq m \leq n $, we let  $P_{(m, 2(n-m))} = M_{(m, 2(n-m))} \, U_{(m, 2(n-m))}$ denote the maximal parabolic subgroup of $\GSp_{2n}$ that stabilizes the isotropic subspace $\langle e_1, \cdots, e_m \rangle \subseteq  W_{2n}^+ .$ The parabolic subgroup $P_{(m, 2(n-m))}$ has Levi $$ M_{(m, 2(n-m))} \simeq \GL_m \times \GSp_{2(n-m)},$$ 
where we write $\GSp_{0} = \GL_1$. In particular, $P_{(n,0)}$ is the Siegel parabolic of $\GSp_{2n}$. We then denote $P_{(m, 2(n-m))}^\circ = M_{(m, 2(n-m))}^\circ \, U_{(m, 2(n-m))}$ the intersection of $P_{(m, 2(n-m))}$ with $\Sp_{2n}$. 

\subsection{Orthogonal groups}

We denote by $\mathbb{H}$ the hyperbolic plane, \textit{i.e.} the $2$-dimensional quadratic $F$-vector space with quadratic form given by $I'_2$.
Let $V$ be any non-degenerate quadratic space over $F$ of dimension $m$ with symmetric bilinear form $B$. Then, $V$ has an orthogonal decomposition of the form $$V_0 \oplus \mathbb{H}^r,$$
with $V_0$ a non-degenerate anisotropic quadratic space. The number $r$ is called the Witt index of $V$. We refer to $V$ being split if $r =  \lfloor \tfrac{m}{2} \rfloor$.

We let $\mathbf{O}_V$, resp. $\mathbf{SO}_V$, be the isometry group of $(V,B)$, resp. the subgroup of 
$\mathbf{O}_V$ of determinant one matrices. We further define the orthogonal similitude group of $(V,B)$ by the rule 
$$\mathbf{GO}_V(R) : = \{ (g,\nu_g') \in \GL_V(R) \times \GL_1(R) \,:\, B(g x,gy) = \nu_g' B(x,y) \},$$
\noindent where $R$ denotes any $F$-algebra. When $m$ is even, we define the subgroup 
$$\mathbf{GSO}_V(R) : = \{ (g,\nu_g') \in \mathbf{GO}_V(R) \,:\, {\rm det}(g) = (\nu_g')^{m/2} \}.$$
We call $ \nu' : \mathbf{GO}_V \to \GL_1, \, g \to \nu_g'$ the similitude character of $\mathbf{GO}_V$.

\begin{remark}
   When $V$ is four dimensional, we have a short exact sequence 
\[1 \to  \mathbf{GSO}_V \to \mathbf{GO}_V \to \mu_2 \to 1,\]
with $\mu_2 = \{ +,-\}$.
\end{remark}
\subsubsection{Maximal parabolic subgroups}\label{subsec:Parabolics:GSO}
Let $V$ be any non-degenerate quadratic space over $F$ of dimension $m$ and Witt index $r$. Then $V$ has a maximal isotropic subspace $V^+$ of dimension $r$ and we can write 
$$V = V^+ \oplus V_0 \oplus V^-,$$
with $V^-$ dual to $V^+$ with respect to the quadratic form on $V$. Fix an ordered basis $\{e_1, \dots, e_r \}$ of $V^+$ and corresponding dual basis $\{f_r, \dots, f_1 \}$ of $V^-$ such that the restriction of the bilinear form $B$ to $V^+ \oplus V^-$ is given by the matrix\footnote{This choice let us identify $V^+$ and $V^-$ with $(\mathbb{H}^+)^r$ and $(\mathbb{H}^-)^r$, respectively.} $I'_{2r}$.

We then let, for any $0 \leq a \leq r$, $Q_{(a , m-2a)} = L_{(a , m-2a)}\, N_{(a , m-2a)}$ denote the maximal parabolic subgroup of $\mathbf{GSO}_V$ stabilizing the isotropic subspace $\langle e_1, \dots, e_a \rangle  \subseteq V^+$. Its Levi is $$L_{(a , m-2a)} \simeq \GL_a \times \mathbf{GSO}_{V'},$$ with $V' = V_0 \oplus \mathbb{H}^{r-a}$. Here, we use the convention $\GSO_{\emptyset} := \GL_1$. When $m = 2r$ and $a=r$, we refer to $Q_{(r , 0)}$ as the Siegel parabolic of $\mathbf{GSO}_{V}$.  Finally, we denote by $Q^\circ_{(a,m-a)} = L^\circ_{(a , m-2a)}\, N_{(a , m-2a)}$ the intersection of $Q_{(a,m-a)}$ with $\SO_{V}$.

\begin{remark}\label{remark_stupid_on_Siegel_parabolic_so22}
For instance, when $r = 2$ and $V = \mathbb{H}^2$ with quadratic form associated with $I'_4$, then $Q_{(2,0)} = L_{(2,0)}\, N_{(2,0)} $ is such that 
\[L_{(2,0)} = \left \{ \left(\begin{smallmatrix} g  & \\ &\mu I_2'\;^{t}g^{-1}I_2'\end{smallmatrix}\right), \,\, g \in \GL_2, \, \mu \in \GL_1 \right \},\]
\[N_{(2,0)} = \left \{ \left(\begin{smallmatrix} 1 & & x & \\ & 1 & & -x\\ & & 1 & \\ & & & 1 \end{smallmatrix}\right), \,\, x \in \mathbb{G}_a \right \}.\]
\end{remark}

\subsubsection{Quaternionic quadratic spaces}\label{Exceptional:Isom:Section}
Let $D$ be a quaternion algebra over $F$ with involution $\iota$ and reduced norm ${\rm Nm}_{D}$. Consider $V=D$ to be the quadratic space with non-degenerate symmetric bilinear form $B(x,y) = \frac{1}{2}{\rm Tr}_{D}(x y^{\iota})$ for $x, y \in D$. Note that $B(x,x) = {\rm Nm}_{D}(x)$ for any $x \in D$. We have an \emph{exceptional isomorphism}  
$$ 1 \to \GL_1 \xrightarrow{z} D^\times \times D^\times \xrightarrow{\rho} \mathbf{GSO}_D \to 1,$$
where $z(t) = (t,t)$, for $t \in \GL_1$ seen as an element in the center of $D^\times$, and $\rho(g,h) v = g^{-1} v h$ for $g,h \in D^\times$ and $v \in D$. Note that $$\nu' (\rho(g,h) ) =  {\rm Nm}_{D}(g)^{-1}{\rm Nm}_{D}(h).$$

We explicit the map $\rho$ in the case where $D$ is the split quaternion algebra $M_2(F)$, in which case the exceptional isomorphism reads as 
\[(\GL_2\times \GL_2)/\{(z,z),\;z\in Z_{\GL_2}\} \simeq \GSO_{\mathbb{H}^2}. \]
The isomorphism is realized by identifying $\mathbb{H}^2\simeq M_2(F)$, with quadratic form $B(x,x)=\det (x)$. Let $e_{ij} \in M_2(F)$ be the matrix with entry 1 at position $(i,j)$ and $0$ elsewhere. Fix the orthogonal basis $\{ e_{11}, e_{12}, -e_{21}, e_{22}\}$, so that we identify  $M_2(F) \simeq F^4$, with bilinear form given by the matrix  $I_4'$.
Here, we think of $F^4$ as row vectors and the $\mathbf{GSO}_{\mathbb{H}^2}$-action is the usual right action. With these choices, $\rho : \GL_2 \times \GL_2 \to \mathbf{GSO}_{\mathbb{H}^2}$ reads as \begin{align}\label{eq:exceptionalisomorphismsplit} 
    \rho:\left( \left(\begin{smallmatrix}
    a & b \\ c & d
\end{smallmatrix} \right), \left(\begin{smallmatrix}
    p & q \\ r & s
\end{smallmatrix} \right) \right) \mapsto
 \tfrac{1}{ad-bc}\left(
\begin{smallmatrix}
d p & d q & c p & -c q\\
d r & d s & c r & -c s\\
b p & b q & a p & - a q\\
-b r & -b s & - a r & a s
\end{smallmatrix}
\right).
\end{align}
\begin{remark}\label{Remark:Siegel:Parabolic:Exceptional}
    Note that $\rho^{-1}Q_{(2,0)}$, with $Q_{(2,0)}$  the Siegel parabolic of $\mathbf{GSO}_{\mathbb{H}^2}$, corresponds to the parabolic subgroup $(B_{\GL_2}^{\rm opp} \times \GL_2) /\{ (z,z)\}$ of $(\GL_2 \times \GL_2)/\{ (z,z)\}$. Indeed, with the identifications of Remark \ref{remark_stupid_on_Siegel_parabolic_so22}, we have
\[\rho \left( \left(\begin{smallmatrix}
    1 &  \\  & {\mu}^{-1} \cdot {\rm det}(g) 
\end{smallmatrix} \right), g \right) =  \left(\begin{smallmatrix} g  & \\ &\mu I_2'\;^{t}g^{-1}I_2'\end{smallmatrix}\right),\, \,  \rho \left( \left(\begin{smallmatrix}
    1 &  \\ x & 1 
\end{smallmatrix} \right), I_2 \right) = \left(\begin{smallmatrix} 1 & & x & \\ & 1 & & -x\\ & & 1 & \\ & & & 1 \end{smallmatrix}\right). \]
\end{remark}

 \subsection{The Weil representation for similitude groups}\label{subsec_weil_rep}

Consider the dual pair $(\mathbf{O}_V , \Sp_{2n})$, with $V$ a non-degenerate quadratic space of even dimension $m$. The vector space $\mathbb{W} = V\otimes_F W_{2n} $ has a natural non-degenerate symplectic form and we fix the embedding $\mathbf{O}_V \times \Sp_{2n} \to \Sp_{\mathbb{W}} \simeq \Sp_{2nm}$, defined by 
$$ (v \otimes w) (h,g) = h^{-1}v \otimes  w g,\, \text{for } v \in V, w \in W_{2n}, h \in \mathbf{O}_V, g \in \Sp_{2n}.$$
Fix a Witt decomposition $\mathbb{W} = \mathbb{W}^+ \oplus \mathbb{W}^-$ and let $\omega = \omega_\psi$ be the $\mathrm{Schr\ddot{o}dinger}$ model of the Weil representation of $\Sp_{\mathbb{W}}$ corresponding to this polarization, which we realize as acting on $\mathcal{S}( \mathbb{W}^+(\A))$. We then let $\omega$ denote also its restriction to $\mathbf{O}_V \times \Sp_{2n}$. For further details, see, for instance, \cite{Prasad:Theta:Survey}. 
We now let
\[ R = \{ (h,g)\in \mathbf{GO}_V \times \GSp_{2n} \,:\, \nu'(h) = \nu(g)\}.\]
Contrary to $\nu$, the orthogonal similitude $\nu'$ is not always surjective. Moreover, there is an isomorphism $R \to \GO_V \ltimes \Sp_{2n}, (h,g) \mapsto (h, \tilde{g}) = (h, d(\nu(g)^{-1}) g) ,$ where 
$$d(x) = \left(\begin{smallmatrix}I_n & \\ &xI_n\end{smallmatrix}\right) \in \GSp_{2n}.$$

The action of $R$ on $V\otimes_F W_{2n}$ induces an an embedding $\iota:R\to \Sp_{\mathbb{W}}$. Thus, the Weil representation extends naturally to $R$ and hence to $\GO_V \ltimes \Sp_{2n}$, as 
\[\omega(h,g)\varphi(x) := |\nu'(h)|^{-\frac{n\mathrm{dim}(V)}{4}}\omega(\iota(h,g))\varphi(x).\]
See \cite[\S 3]{RobertsTheta} for further details. For instance, if we fix the polarization associated with $\mathbb{W}^+ = V \otimes W_{2n}^+$, the extension of the Weil representation to $R$ is explicitly given as follows.  We let the group $\GO_V(\A)$ act on $\mathcal{S}((V \otimes W_{2n}^+)(\A))$ by 
\[ L(h) \phi (x \otimes y) = |\nu'(h)|^{-\frac{n \cdot \dim(V)}{4}} \cdot \phi(h^{-1} x  \otimes y).\]
Since, for $g \in \Sp_{2n}(\A)$, by \cite[Lemma 3.2]{RobertsTheta} one has $$L(h) \omega(g) L(h^{-1}) = \omega\left(d(\nu'(h)) g d(\nu'(h)^{-1}) \right),$$ we obtain a representation of $R(\A)$ on $\mathcal{S}((V \otimes W_{2n}^+)(\A))$ given by 
\begin{align}\label{eq_normalization_Weil_similitudes}
    \omega(h,g) \phi(x \otimes y) =  (L(h) \omega( \tilde{g}) \phi)(x \otimes y) =  |\nu'(h)|^{-\frac{n \cdot \dim(V)}{4}} \cdot (\omega( \tilde{g}) \phi)(h^{-1} x \otimes y).
\end{align}

\subsection{Dual groups and the $\wedge^2  \otimes {\rm std}_2$ $L$-function}\label{subsec_dualgroups}

Recall that the dual group of  $\GL_n$ is ${}^L\GL_n = \GL_n(\C)$. We then let 
${\rm std}_n$  and ${\rm std}_n^\vee$ denote the standard representation of $\GL_n(\C)$ and its dual, respectively.

\subsubsection{Exterior square representations}

The representation ${\rm std}_4$ of $\GL_4(\C)$ on $\C^4$, with basis $\{e_1,e_2,e_3,e_4\}$, induces the representation $\wedge^2$ on $\wedge^2 \C^4 \simeq \C^6$ by the rule $(\wedge^2 g) \cdot ( e_i \wedge e_j) = g \cdot e_i \wedge g \cdot e_j$. This extends to a representation of the dual group of $\GU_{2,2}$, as follows. 
Recall that is \[ {}^L\GU_{2,2} = ( \GL_4(\C) \times \GL_1(\C)) \rtimes {\rm Gal}(E/F),  \]
with the action of the non-trivial element $\tau \in {\rm Gal}(E/F)$ given by (cf. \cite[\S 1.8(c)]{BlasiusRogawski})
\[ (g,\lambda) \mapsto (\Phi_4 {}^tg^{-1} \Phi_4 , \lambda {\rm det}(g)), \text{ where }  \Phi_4 = \left( \begin{smallmatrix}  & & &1\\&&-1& \\ & 1&&\\ -1& & & \end{smallmatrix} \right). \] 
Consider the six dimensional representation \[ \GL_4(\C) \times \GL_1(\C) \to \GL_6(\C),\] given by $ (g,\lambda) \mapsto  (\wedge^2 g) \lambda$. By \cite[Lemma 2.1]{KimK}, $\wedge^2$ extends to a representation of ${}^L\GU_{2,2}$. Indeed, notice that both representations $\wedge^2$ and $\wedge^2 \circ \tau$, with $1 \ne \tau \in {\rm Gal}(E/F)$, have same highest weight. Hence they are isomorphic, \textit{i.e.}, there exists $A \in \GL_6(\C)$ such that  $(\wedge^2 \circ \tau) (g,\lambda) = A^{-1} ( \wedge^2  (g,\lambda) ) A$. Choose $A$ with ${\rm Tr}(A) >0$ and extend $\wedge^2$ to ${}^L\GU_{2,2}$ by sending \[ (g,\lambda,1) \mapsto (\wedge^2g)\lambda,\, (1,1,\tau) \mapsto A.\]
The resulting representation, which we still denote by $\wedge^2: {}^L\GU_{2,2} \to \GL_6(\C)$, is called the (twisted) exterior square representation of ${}^L\GU_{2,2}$.  
\subsubsection{$L$-functions}
Let $\G$ be either $\GL_4$ or $\GU_{2,2}$. Consider the tensor product 
    \[\wedge^2\otimes\mathrm{std}_2:{}^{L}\G \;\times {}^{L}\GL_2\to \GL_{12}(\C).\]
    Let $\pi = \otimes_v' \pi_v$ and $\sigma = \otimes_v' \sigma_v$ be cuspidal automorphic representations of $\G$ and $\GL_2$, respectively. For any (possibly empty) finite subset $S$ of places of $F$, we let $$L^S(s, \pi \otimes \sigma, \wedge^2\otimes\mathrm{std}_2) := \prod_{v \not \in S} L(s,\pi_v \otimes \sigma_v, \wedge^2\otimes\mathrm{std}_2).$$ In particular, when $S =\emptyset$,  $L(s, \pi \otimes \sigma, \wedge^2\otimes\mathrm{std}_2)$ denotes the completed $L$-function. 
    
    If $v$ is a finite place of $F$, which is unramified for $\pi_v$, $\sigma_v$, and $E/F$ (if $\G =\GU_{2,2}$), then the local $L$-factor is given by 
    $$ L(s,\pi_v\otimes\sigma_v,\wedge^2\otimes\mathrm{std}_2) = \mathrm{det}\left(1-(\wedge^2\otimes \mathrm{std}_2)(s_{\pi_v}\times s_{\sigma_v})\cdot q^{-s}\right)^{-1}\!\!\!\!\!\!\!,$$
    where $s_{\pi_v}$ and $s_{\sigma_v}$ denote representatives of the Frobenius conjugacy classes of $\pi_v$ and $\sigma_v$, respectively. Finally, we refer to \S \ref{subsec:ArchimedeanLfactors}, for a description of the local archimedean $L$-factors of the $\wedge^2 \otimes {\rm std}_2$ $L$-function of $\GL_4 \times \GL_2$.

\section{The zeta integral} \label{sec:The_zeta_integral}

In this section, we let $\H = \GSp_4$ and let $\G$ be either $\GL_4$ or $\GU_{2,2}$. For $G \in \{\H,\G\}$, also let $B_G = T_G \,U_G$ denote the upper triangular Borel subgroup of $G$, with $T_G$ the maximal diagonal torus in $G$, and $U_G$ the unipotent radical of $B_G$.

\subsection{The Klingen Eisenstein series}\label{Section:Klingen:Eisenstein}
Fix a cuspidal automorphic representation $\sigma \subset \mathcal{A}_{\rm cusp}([\GL_2])$ with central character $\omega_\sigma$ and a Hecke character $\chi$ of $F^\times \backslash \A^\times$. For $s \in \C$, we let
\begin{equation}\label{induction:Klingen}
I_{P_{(1,2)}}(s,\sigma, \chi) := {\rm Ind}_{P_{(1,2)}(\A)}^{\H(\A)}\left( ((\chi\omega_{\sigma})^{-1}\otimes{\sigma}) \delta_{P_{(1,2)}}^{\tfrac{1}{2}(s-1/2) } \right)\end{equation}be the space of smooth functions $f_s:\GSp_4(\A)\to \C$ such that, for every $ n \in U_{(1,2)}(\A)$ and $  m = \left(\begin{smallmatrix} a &  &  \\ & r & \\  & & d \end{smallmatrix} \right) \in M_{(1,2)}(\A)$, we have
\[f_{s}(nm) = (\chi\omega_{\sigma})^{-1}(d)\delta_{P_{(1,2)}}^{\tfrac{1}{2}(s + 1/2)}(m)\Psi(r),\]
where $\Psi$ is a cusp form in $\sigma$. Recall that the modulus character of $P_{(1,2)}$ is given by \begin{align*}
\delta_{P_{(1,2)}}&:\left(\begin{smallmatrix} a &  &  \\ & r & \\  & & d \end{smallmatrix} \right)  \mapsto \left|\tfrac{a}{d}\right|^2.
\end{align*}
This representation is a (possibly reducible) automorphic representation with central character $\chi^{-1}$. 
Given a smooth section $ f_s \in I_{P_{(1,2)}}(s,\sigma,\chi)$, define \begin{align}\label{WhittakerFoRKES}
    W_{\sigma,f_s}(g) : = \int_{F \backslash \A} f_s \left(\left(\begin{smallmatrix}1& & & \\ &1&x& \\ & &1& \\ & & &1\end{smallmatrix}\right)g\right)\psi(-x)dx,
\end{align}  
which is an element of $I_{P_{(1,2)}}(s,\mathcal{W}_{\sigma},\chi)$, with $\mathcal{W}_\sigma$ the Whittaker model of $\sigma$.  When the context is clear, we suppress $f_s$ from the notation and denote this simply by $W_{\sigma,s}$. We have the Fourier expansion (\textit{cf.} \cite[(10.4)]{BumpBook})
\begin{equation}\label{FourExpf}f_s(g) = \sum_{\alpha\in F^{\times}}W_{\sigma,s}\left(\left(\begin{smallmatrix}\alpha& & & \\ &\alpha& & \\ & &1& \\ & & &1\end{smallmatrix}\right)g\right).\end{equation}
Given a $K$-finite smooth section $ f_s \in I_{P_{(1,2)}}(s,\sigma,\chi)$, let $S$ be a finite set of places containing the archimedean places of $F$ and the places where $f_s$ is ramified. Define the normalized Klingen Eisenstein series (cf. \cite[p.390]{BumpBook}) \begin{align}\label{KES}
    E^*_{P_{(1,2)}}(g,s, f_s) :=   L^{S}(2s,\sigma,{\rm Sym}^2 \otimes \chi) \cdot \!\!\!\!\!\!\!\!\!\!\!\! {\sum_{\gamma \in {P_{(1,2)}}(F) \backslash \H(F)}} f_s(\gamma g),
\end{align}  
where $L^{S}(2s,\sigma,{\rm Sym}^2 \otimes \chi)$ is the partial symmetric square $L$-function\footnote{The calculation of the normalization factor follows from checking how the Satake parameters of $I_{P_{(1,2)}}(s,\sigma,\chi)$ at unramified primes act on $^L\mathfrak{n}_{{(1,2)}}$ via the natural action of the Levi of $^LP_{(1,2)}$ on it.} of $\sigma$ twisted by $\chi$.  Recall that, if $\sigma_v$ is unramified with Frobenius conjugacy class of the form $\mathrm{diag}(\xi_1(\varpi_{v}),\xi_2(\varpi_v))$, then 
\begin{equation}\label{LfunctionSigma}L(2s,\sigma_v,{\rm Sym}^2 \otimes \chi_v) = (1-(\xi^{2}_1\chi_v)(\varpi_v)q_v^{-2s})(1-(\omega_{\sigma_v}\chi_v)(\varpi_v)q_v^{-2s})(1-(\xi^{2}_2\chi_v)(\varpi_v)q_v^{-2s}).\end{equation}
If $\mathrm{Re}(s)$ is sufficiently large, the Eisenstein series defined by \eqref{KES} converges absolutely and it extends meromorphically to a function on the whole complex plane. Note that the poles of the Eisenstein series $E^*_{P_{(1,2)}}(g,s, f_s)$ are given by the poles of the $L$-function $L^{S}(2s-1,\sigma ,{\rm Sym}^2 \otimes \chi)$. According to \cite[Theorem 9.3]{GelbartJacquet} the twisted symmetric square $L$-function of $\sigma$ can have poles only when $\sigma\otimes\eta \simeq \sigma$ for some non-trivial quadratic character $\eta$. A detailed analysis of the possible isobaric factorization of the functorial lift ${\rm Sym}^2(\sigma) \otimes \chi$ in this case reveals that a necessary condition for the pole is $\chi\omega_{\sigma}\neq 1$ and $(\chi\omega_{\sigma})^2 = 1$.  

\subsubsection{Spherical vectors} \label{subsec_Weird_local_whitt}

We assume that $f_s = \otimes_v f_{s,v}$ is a pure tensor in $I_{P_{(1,2)}}(s,\sigma,\chi)$. As in \cite[\S 3.10]{BumpBook}, to $f_s$ we associate  a function $W_{\sigma,s}$, which factors as a restricted product of elements \[W_{\sigma_v,s} \in {\rm Ind}_{P_{(1,2)}(F_v)}^{\H(F_v)}\left(((\chi_v\omega_{\sigma_v})^{-1}\otimes\mathcal{W}_{\sigma_v})\delta_{P_{(1,2)}}^{\tfrac{1}{2}(s-1/2)}\right),\] with $\mathcal{W}_{\sigma_v}$ the Whittaker model of $\sigma_v$. In particular, for $n \in U_{{(1,2)}}(F_v)$, $m = \left(\begin{smallmatrix} a &  &  \\ & r & \\  & & d \end{smallmatrix} \right)  \in M_{(1,2)}(F_v)$, and $g \in \H(F_v)$, we have \[W_{\sigma_v,s}(n m g) = \delta_{P_{(1,2)}}^{\tfrac{1}{2}(s+1/2)}(m) (\chi_v\omega_{\sigma_v})^{-1}(d) W_{g,v}(r),\]
with $W_{g,v} \in \mathcal{W}_{\sigma_v}$ depending on $g$. When $v \not \in S$, we have a unique spherical vector $W_{s,v}^{\rm o}$, normalized so that its restriction to $\H(\mathcal{O}_v)$ is one. Using the Iwasawa decomposition, write $g \in \H(F_v)$ as $n m k$, with $n m \in P_{(1,2)}(F_v)$ and $k \in \H(\mathcal{O}_v)$; then, if $m = \left(\begin{smallmatrix} a &  &  \\ & r & \\  & & d \end{smallmatrix} \right) $, we have
\begin{align}\label{SphericalWhittakerforKES}
    W_{s,v}^{\rm o}(n m k) = \left|\tfrac{a}{d}\right|^{s+1/2} (\chi_v\omega_{\sigma_v})^{-1}(d)W_{\sigma_v}(r), 
\end{align}
with $W_{\sigma_v}$ the spherical vector of $\mathcal{W}_{\sigma_v}$, normalized so that $W_{\sigma_v}$ is equal to $1$ on $\GL_2(\mathcal{O}_v)$ (\textit{cf.} \cite[p. 393]{BumpBook}).

\subsection{The zeta integral and its unfolding}

Consider $\pi \subset \mathcal{A}_{\rm cusp}([\G])$,  $ \sigma \subset \mathcal{A}_{\rm cusp}([\GL_2])$ cuspidal automorphic representations with central characters $\omega_{\pi}$ and $\omega_{\sigma}$, respectively. Given a cusp form $\varphi \in \pi$ and a factorizable section $f_s \in I_{P_{(1,2)}}(s,\sigma,\omega_{\pi})$\footnote{ If $\G = \GU_{2,2}$, $\omega_{\pi}$ denotes its restriction ${\omega_{\pi| \A^\times}}$ to $F\backslash \A^\times$}, we define
\[Z(s,\varphi,f_s) := \int_{\H(F) Z_\H(\A)\setminus \H(\A)} \!\!\!\!\!E^*_{P_{(1,2)}}(g,s, f_s)\varphi(g)dg.\]
Since $\varphi|_{\H(\A)}$ is rapidly decreasing and the Klingen Eisenstein series is of moderate growth and meromorphic, the integral converges absolutely and defines a meromorphic function on $s$. Let $\psi$ be the non-trivial additive character of $F \backslash \A$ fixed in \S \ref{subsec_notation}.

\begin{definition} \label{def:principal:character} \leavevmode \begin{itemize}
    \item If $\G = \GU_{2,2}$, let $\chi :U_{\G}(F) \backslash U_{\G}(\A) \to\C^{\times}$  be the principal character given $$\chi\left(\left(\begin{smallmatrix} 1 & x & y & a\\ & 1 & b & \bar{ y} \\ & & 1 & - \bar{x} \\ & & & 1\end{smallmatrix} \right)\right) = \psi(\mathrm{Tr}_{E/F}(\delta x) - b).$$
    \item If $\G = \GL_4$, let $\chi :U_{\G}(F) \backslash U_{\G}(\A) \to\C^{\times}$  be the principal character given $$\chi\left(\left(\begin{smallmatrix} 1 & a & b & c\\ & 1 & d & e \\ & & 1 & f \\ & & & 1\end{smallmatrix} \right)\right) = \psi(a - d + f).$$
    \end{itemize} 
\end{definition} 
We consider the Whittaker model of a cusp form $\varphi$ of $\G$ with respect to $\chi$, \textit{i.e.}
\[W_{\varphi}(g) := \int_{[U_{\G}]}\varphi(ng)\chi^{-1}(n)dn.\]

\begin{prop}\label{unfolding}
The zeta integral unfolds to 
\[Z(s,\varphi,f_s) = L^S(2s, \sigma, {\rm Sym}^2 \otimes \omega_\pi)\int_{U_\H(\A)Z_\H(\A)\setminus \H(\A)} \!\!\!\!\!\!W_{\sigma,s}\left(g\right)W_{\varphi}(g)dg.\]
\end{prop}
\begin{proof}
We prove the statement only in the case of $\G = \GU_{2,2}$, as the computation when $\G = \GL_4$ is analogous.  When ${\rm Re}(s)$ is big enough, we can unfold the Eisenstein series and obtain
\[Z(s,\varphi,f_s) =L^S(2s, \sigma, {\rm Sym}^2 \otimes \omega_\pi)  \int_{P_{(1,2)}(F)Z_{\H}(\A)\setminus \H(\A)} \!\!\!\! f_s(g)\varphi(g)dg.\]
Using the invariance of $f_s(g)$ over $U_{(1,2)}(\A)$, we collapse the sum over $U_{(1,2)}(\A)$ to write the integral as
\[\int_{M_{(1,2)}(F)U_{(1,2)}(\A)Z_\H(\A)\setminus \H(\A)} \!\!\!\! f_s(g)\int_{[ U_{(1,2)}]}\varphi(ng)dndg,\]
 where $dn$ is the Haar measure on $U_{(1,2)}(\A)$. Now, let $Q_\G = L_\G  N_\G$ be the standard Klingen parabolic of $\G$, with $L_\G \simeq \GL_1 \times \GU_{1,1}$. The function $\varphi'(g) := \int_{[U_{(1,2)}]}\varphi(ng)dn$ is an automorphic form on $\G(\A)$ invariant under $U_{(1,2)}$, hence we can Fourier expand it over $U_{(1,2)}\setminus N_{{\G}}$:
\[\varphi'(g) = \sum_{\chi:[U_{(1,2)}\setminus N_{{\G}}]\to \C^{\times}}\varphi_{\chi}(g),\]
with $\varphi_{\chi}(g) = \int_{[N_{{\G}}]}\varphi(n'g)\chi^{-1}(n')dn'$. Any character of $[U_{(1,2)}\setminus N_{{\G}}]$ is of the form  \begin{align*}
     \chi_{\alpha,\beta} : [U_{(1,2)}\setminus N_{{\G}}] &\to \C^\times,\\ n' &\mapsto  \psi( {\rm Tr}_{E/F} (\alpha x + \beta y)),
\end{align*} with $\alpha, \beta \in E$ such that $\bar{\alpha} = - \alpha$ and  $\bar{\beta} = - \beta$ (\textit{cf.} \cite[Proposition 3.1]{CauchiGutiMVI}). As in \emph{loc.cit.}, we note that $M_{(1,2)}(F)$ acts on the space of characters of $[U_{(1,2)}\setminus N_{{\G}}]$  with two orbits, the trivial one and an open one. A representative of the open orbit is given by the character $\chi_{\delta,0}$. Since $\varphi$ is a cusp form, the constant term $\varphi_{N_{{\G}}}(g) =0$, hence the integral can be written as
\begin{equation}\label{auxFourierexp}\int_{M_{\delta}(F)U_{(1,2)}(\A)Z_\H(\A)\setminus \H(\A)} f_s(g)\varphi_{\chi_{\delta,0}}(g)dg,\end{equation}
with \[M_{\delta}(F) =L_{\delta}(F)N_{\delta}(F) = \left \{ m(a,d)n(b) = \left(\begin{smallmatrix} a &   &   &  \\ & a &   &   \\ & & d &   \\ & & & d\end{smallmatrix} \right)\left(\begin{smallmatrix} 1 &   &   &  \\ & 1 &  b &   \\ & & 1 &   \\ & & & 1\end{smallmatrix} \right)\;a,d\in F^\times,\, b \in F \right \}.\]
Collapse \eqref{auxFourierexp} over $[N_{\delta}]$ to obtain 
\[\int_{L_{\delta}(F)N_{\delta}(\A)U_{(1,2)}(\A)Z_\H(\A)\setminus \H(\A)}\int_{[N_{\delta}]} f_s(n(b)g)\varphi_{\chi_{\delta,0}}(n(b)g)dbdg.\]
By \eqref{FourExpf}, we have
\begin{equation*}f_s(g) = \sum_{t \in F^{\times}}W_{\sigma,s}\left(\left(\begin{smallmatrix}t & & & \\ & t & & \\ & &1& \\ & & &1\end{smallmatrix}\right)g\right),\end{equation*}
where $W_{\sigma,s}(g)$ is the element of  $I_{P_{(1,2)}}\left (s,\mathcal{W}_{\sigma},\omega_{\pi} \right)$ defined in \eqref{WhittakerFoRKES}, with $\mathcal{W}_\sigma$ the Whittaker model of $\sigma$. Therefore the integral is equal to 
\begin{equation}\label{Aux2}\int_{L_{\delta}(F)U_\H(\A)Z_\H(\A)\setminus \H(\A)}\int_{[N_{\delta}]} \sum_{t \in F^{\times}}W_{\sigma,s}\left(\left(\begin{smallmatrix} t & & & \\ & t & & \\ & &1& \\ & & &1\end{smallmatrix}\right)n(b)g\right)\varphi_{\chi_{\delta,0}}(n(b)g)dbdg.\end{equation}
Here, we used that $N_{\delta}(\A)U_{(1,2)}(\A) = U_{\H}(\A)$, with $U_\H$ the unipotent radical of the upper triangular Borel of $\H$. Since $[N_{\delta}]$ is compact and the sum \eqref{FourExpf} converges absolutely on compact sets, using Fubini's theorem, the integral \eqref{Aux2} is equal to 
\begin{align*}\int_{L_{\delta}(F)U_\H(\A)Z_\H(\A)\setminus \H(\A)}&\sum_{t \in F^{\times}}\int_{[N_{\delta}]} W_{\sigma,s}\left(\left(\begin{smallmatrix} t & & & \\ &t & & \\ & &1& \\ & & &1\end{smallmatrix}\right)n(b)g\right)\varphi_{\chi_{\delta,0}}(n(b)g)dbdg \\ &= \int_{U_\H(\A)Z_\H(\A)\setminus \H(\A)}\int_{[N_{\delta}]} W_{\sigma,s}\left(n(b)g\right)\varphi_{\chi_{\delta,0}}(n(b)g)dbdg,\end{align*}
where we have used the equality $L_{\delta}(F) =Z_\H(F)\cdot \left \{\left(\begin{smallmatrix}t& & & \\ & t & & \\ & &1& \\ & & &1\end{smallmatrix}\right),\;t\in F^\times \right \}$.
Finally, using the transformation property $W_{\sigma,s}(n(b)g) = \psi(b) W_{\sigma,s}(g)$, the above integral is equal to 
\[\int_{U_\H(\A)Z_\H(\A)\setminus \H(\A)} W_{\sigma,s}\left(g\right)\int_{[N_{\delta}]}\varphi_{\chi_{\delta,0}}(n(b)g)\psi(b)dbdg.\]
The character on $ [U_\G /[U_\G, U_\G]]$, defined by taking $b \mapsto \psi(b)$ on $N_\delta$ and by taking $\chi_{\delta,0}$ on the one parameter unipotent subgroup associated to the other simple root, is the principal character $\chi$ of Definition \ref{def:principal:character}. The result thus follows.
\end{proof}

Using the uniqueness of Whittaker models, Proposition \ref{unfolding} has an immediate Corollary.

\begin{cor}\label{Eulerexp}
Let $\varphi = \otimes_v\varphi_v$ and $f_{s} = \otimes_v f_{s,v}$ be factorizable elements in the space of $\pi$ and $I_{P_{(1,2)}}(s,\sigma,\omega_{\pi})$ respectively. The zeta integral $Z(s,\varphi,f_s)$ admits the following Euler product expansion:
\[Z(s,\varphi,f_s) = L^S(2s, \sigma, {\rm Sym}^2 \otimes \omega_\pi) \prod_{v} Z(s,W_{\pi_v},W_{\sigma_v,s}),\]
where \[Z(s,W_{\pi_v},W_{\sigma_v,s}) := \int_{U_\H(F_v)Z_\H(F_v)\setminus \H(F_v)} W_{\sigma_v,s}\left(g_v\right)W_{\pi_v}(g_v)dg_v.\]
\end{cor}

\subsection{Preparation to the unramified calculations}

Let $v$ be a finite place of $F$ which is unramified for $\pi_v$, $\sigma_v$, and $E/F$. Recall that we have denoted by $q_v$ the size of the quotient of $\mathcal{O}_v$ modulo its maximal ideal $\mathfrak{p}_v$ and fixed a uniformizer  $\varpi_v$ of $\mathcal{O}_v$ such that the norm $|\cdot |$ of $F_v$ is normalized by setting $|\varpi_v| = q_v^{-1}$. Let $W_{\pi_v}$ be the spherical vector in the Whittaker model of $\pi_v$, normalised so that it is equal to $1$ on $\G(\mathcal{O}_v)$. We also let  $W_{\sigma_v}$  be the spherical vector of $\mathcal{W}_{\sigma_v}$ and let $W_{s,v}^{\rm o}$ be the spherical vector of ${\rm Ind}_{P_{(1,2)}(F_v)}^{\H(F_v)}\left(((\omega_{\pi_v}\omega_{\sigma_v})^{-1}\otimes\mathcal{W}_{\sigma_v})\delta_{P_{(1,2)}}^{\tfrac{1}{2}(s-1/2)}\right)$, normalized as in \eqref{SphericalWhittakerforKES}. Using the Iwasawa decomposition $\H(F_v) = B_\H(F_v) \H(\mathcal{O}_v)$, the local integral  at $v$ of Corollary \ref{Eulerexp} reads as
\[Z(s,W_{\pi_v},W_{s,v}^{\rm o}) = L(2s, \sigma_v, {\rm Sym}^2\otimes \omega_{\pi_v} ) \int_{Z_\H(F_v)\setminus T_{\H}(F_v)}\delta_{B_\H}^{-1}(t) W_{s,v}^{\rm o}(t)W_{\pi_v}(t)dt.\]
From now on, we identify $Z_\H(F_v) \setminus T_{\H}(F_v)$ with $T_{1}(F_v) := \left\{t(ab,b;b) = \left(\begin{smallmatrix}ab& & & \\ &b& & \\ & &1& \\ & & &a^{-1}\end{smallmatrix}\right),\;a,b\in F_v^{\times}\right\}$. 
\begin{lemma}\label{stupidbounds}
If $W_{s,v}^{\rm o}(t(ab,b;b))W_{\pi_v}(t(ab,b;b))\neq 0$, then $a,b\in\mathcal{O}_v$.
\end{lemma}
\begin{proof} 
We show the Lemma only when $\G = \GU_{2,2}$, as the proof for the other case is almost identical. Write $u = \left(\begin{smallmatrix}1&x &y &z \\ &1 & & \overline{y}\\ & &1&-\overline{x} \\ & & & 1\end{smallmatrix}\right) \in U_{{\G}}(\mathcal{O}_v)$ and $u' = \left(\begin{smallmatrix}1& & & \\ &1 &y' & \\ & &1& \\ & & & 1\end{smallmatrix}\right)\in U_{\H}(\mathcal{O}_v)$. By \eqref{SphericalWhittakerforKES}, we have that 
\begin{align*}
W_{s,v}^{\rm o}(t(ab,b;b)) &= W_{s,v}^{\rm o}(t(ab,b;b)u') \\ 
&= W_{s,v}^{\rm o}(t(ab,b;b)u't(ab,b;b)^{-1} t(ab,b;b)) \\ 
&= \psi_v(b y')W_{s,v}^{\rm o}(t(ab,b;b)).
\end{align*}  
If $W_{s,v}^{\rm o}(t(ab,b;b)) \ne 0$, then $\psi_v(b y')=1$ for every $y' \in \mathcal{O}_v$. This implies that $b \in \mathcal{O}_v$. Similarly, since
\begin{align*}
W_{\pi_v}(t(ab,b;b)) &= W_{\pi_v}(t(ab,b;b)u) = \psi_v({\rm Tr}_{(E \otimes_F F_v)/F_v}(\delta_v a x))W_{\pi_v}(t(ab,b;b)), 
\end{align*}  
if $W_{\pi_v}(t(ab,b;b)) \ne 0$, then $a \in \mathcal{O}_v$. These two facts imply the desired result.
\end{proof}
\noindent Lemma \ref{stupidbounds} let us write $Z(s,W_{\pi_v},W_{s,v}^{\rm o})$ as
\begin{align*} L(2s, \sigma_v, {\rm Sym}^2\otimes \omega_{\pi_v} )  \sum_{n,m \geq 0} \delta_{B_\H}^{-1}(t(\varpi_v^{n+m},\varpi_v^m;\varpi_v^m))W_{s,v}^{\rm o}\left(t(\varpi_v^{n+m},\varpi_v^m;\varpi_v^m)\right)W_{\pi_v}(t(\varpi_v^{n+m},\varpi_v^m;\varpi_v^m)).\end{align*}
By \eqref{SphericalWhittakerforKES}, 
\[W_{s,v}^{\rm o}\left(t(\varpi_v^{n+m},\varpi_v^m;\varpi_v^m)\right) = q_v^{-(2n+m)(s+1/2)} (\omega_{\pi_v}\omega_{\sigma_v})(\varpi_v^n) W_{\sigma_v}(\hat{t}(\varpi_v^m,1)),\]
where $\hat{t}(\alpha,\beta) := \left(\begin{smallmatrix} \alpha& \\ & \beta \end{smallmatrix}\right)\in \GL_2(F_v)$. Using the action on the models by the center of $\H(F_v)$ and the equality $\delta_{B_\H}^{-1}(t(\varpi_v^{n+m},\varpi_v^m;\varpi_v^m)) = q_v^{4n+3m}$,   $Z(s,W_{\pi_v},W_{s,v}^{\rm o})$ is seen to be  
\begin{align}\label{WhittakerLocalExpr} L(2s, \sigma_v, {\rm Sym}^2\otimes \omega_{\pi_v} )   \sum_{n,m \geq 0} q_v^{4n+3m-(2n+m)(s+1/2)} W_{\sigma_v}(\hat{t}(\varpi_v^{n+m},\varpi_v^{n}))W_{\pi_v}(t(\varpi_v^{2n+m},\varpi_v^{n+m};\varpi_v^{2n+m})).\end{align}

\subsection{Symmetric algebra factorizations and local unramified zeta integrals}\label{subsec_localzeta1}

We keep the notation of the previous section. In order to evaluate the local zeta integrals explicitly, we crucially use two key results appearing in the Appendix of \cite{ExplicitGSR}. With these in hand and some simple manipulation, we prove the following result.

\begin{theorem}\label{zetaintegralfinalthm}
    Let $\varphi = \otimes_v\varphi_v$ and $f_{s} = \otimes_v f_{s,v}$ be pure tensors in $\pi$ and $I_{P_{(1,2)}}(s,\sigma, \omega_\pi)$, respectively. Let $S$ be a finite set of places of $F$ containing the archimedean places and the ramified places for $\varphi$, $f_s$, and $E/F$ (if $\G = \GU_{2,2}$). Then we have
    \[Z(s,\varphi,f_s) = L^S(s,\pi \otimes \sigma, \wedge^2 \otimes  \mathrm{std}_2)\prod_{v\in S}Z(s,W_{\pi_v},W_{\sigma_v,s}).\]
\end{theorem}

\begin{proof}
  This is the combination of Corollary \ref{Eulerexp} with Corollaries \ref{ZtoLtensorNonSplit} and \ref{Lrequaltowedge}, which are proved below. 
\end{proof}

\noindent The unramified computation for $\GL_4$ will follow from the one of $\GU_{2,2}$ at places which split in $E$. Therefore, in what follows we can assume that $\G = \GU_{2,2}$. Let $v \not \in S$. We distinguish the cases of $v$ inert or split in $E$.

\subsubsection{The inert case}

Let $\pi_v$ and $\sigma_v$ be unramified smooth representations of $\G(F_v)$ and $\GL_2(F_v)$. We have that $\pi_v$ is the unique unramified subquotient of $\mathrm{Ind}_{B_{\G}(F_v)}^{\G(F_v)}\chi$ with $\chi$ an unramified character of $T_{\G}(F_v)$. Denote by $\pi_v'$ the unique unramified subquotient of $\mathrm{Ind}_{B_{\H}(F_v)}^{\H(F_v)}\chi_{|_{T_{\H}}}$. Recall that the Casselman--Shalika formula for the Whittaker coefficient of $\pi_v$ involves the Satake parameters $s_{\pi_v'}$ of $\pi_v'$  (\textit{cf}. \cite[Proposition 5.2.1]{GanHundley}). It is thus natural to try to relate the integral \eqref{WhittakerLocalExpr} firstly to an $L$-factor for $\pi_v' \otimes \sigma_v$. Recall that the dual group of $\H$ is isomorphic to itself. Let \[\Lambda^+:= \{ \mathbf{k} = (k_1,k_2,k_3) \in \Z^3,\, k_1 \geq k_2 \geq 0 \}\] be the set of dominant weights for $\H(\C)$ and let $\varpi_v^\mathbf{k} = t(\varpi_v^{k_1+k_2+k_3},\varpi_v^{k_1+k_3};\varpi_v^{k_1+k_2+2k_3})$. Then, the Casselman--Shalika formula (\textit{cf}. \cite[(3.9)]{FurusawaMorimoto}) gives that \begin{align}\label{eq_ref_CS_GU_inert}
     W_{\pi_v}(\varpi_v^\mathbf{k}) \cdot \delta_{B_\G}^{-1/2}(\varpi_v^\mathbf{k}) = W_{\pi_v'}(\varpi_v^\mathbf{k}) \cdot \delta_{B_\H}^{-1/2}(\varpi_v^\mathbf{k}) = {\rm Tr}(s_{\pi_v'}| V_\mathbf{k}),
\end{align}  
where $V_\mathbf{k}$ is the irreducible representation of $\H(\C)$ of highest weight $\mathbf{k} \in \Lambda^+$. Applying this formula to the infinite series of \eqref{WhittakerLocalExpr} we get back the local integral of \cite{Novodvorsky}, which calculates $L(2s,\omega_{\pi_v'}\omega_{\sigma_v})L(s,\pi_v' \otimes \sigma_v, r),$  as the following Proposition shows. Here we have denoted $r: \H(\C) \times \GL_2(\C) \to \GL_8(\C)$ the tensor product of the two standard representations.

\begin{prop}\label{theononsplit1}
 We have that \[Z(s,W_{\pi_v},W_{s,v}^{\rm o}) = L(2s,\omega_{\pi_v'}\omega_{\sigma_v})^{-1} L(2s,\sigma_v,{\rm Sym}^2 \otimes \omega_{\pi_v}) L(s,\pi_v' \otimes \sigma_v, r). \]   
\end{prop}

\begin{proof}
The Casselman--Shalika formula for $\GL_2$ says that \begin{align*}
    W_{\sigma_v}(\hat{t}(\varpi_v^{m+n},\varpi_v^n)) &=q_v^{-m/2}\cdot  \mathrm{Tr}(s_{\sigma_v}|W_{(m+n,n)}),
\end{align*} 
where $W_{(m+n,n)}$ is the irreducible representation of $\GL_2(\C)$ of highest weight $(m+n,n)$. On the other hand, \eqref{eq_ref_CS_GU_inert} applied to $\mathbf{k}=(m+n,n,0)$ gives that \begin{align*}
     W_{\pi_v}(t(\varpi_v^{2n+m},\varpi_v^{m+n};\varpi_v^{2n+m})) &= q_v^{-3n-2m}\cdot \mathrm{Tr}(s_{\pi_v'}|V_{(m+n,n,0)}).
\end{align*} 
Plugging in these formulas, \eqref{WhittakerLocalExpr} is equal to
\begin{align*} L(2s, \sigma_v, {\rm Sym}^2\otimes \omega_{\pi_v} )   \sum_{n,m \geq 0} q_v^{-s(2n+m)} \mathrm{Tr}(s_{\pi_v'} \otimes s_{\sigma_v}|V_{(m+n,n,0)} \otimes W_{(m+n,n)}).
\end{align*}

The result then follows from Lemma \ref{SymmetricAlgFact} below. 
\end{proof}

\begin{lemma}\label{SymmetricAlgFact}
We have 
\[L(s,\pi_v' \otimes \sigma_v, r) = L(2s,\omega_{\pi_v'}\omega_{\sigma_v}) \sum_{n,m \geq 0} q_v^{-s(2n+m)} \mathrm{Tr}(s_{\pi_v'} \otimes s_{\sigma_v}|V_{(m+n,n,0)} \otimes W_{(m+n,n)}),\]
where $V_{(m+n,n,0)} $, resp.  $W_{(m+n,n)}$, is the irreducible representation of $\GSp_4(\C)$, resp. $ \GL_2(\C)$, of highest weight $(m+n,n,0)$, resp. $(m+n,n)$. 
\end{lemma}

\begin{proof}
This follows from the calculation of the right hand side of \cite[(A.1.1)]{ExplicitGSR} in \cite[Proof of Proposition A.1]{ExplicitGSR} at pages 139-140 and we only sketch it. By the Poincar\'e polynomial identity \cite[(A.1.2)]{ExplicitGSR}, we have 
\[L(s,\pi_v' \otimes \sigma_v, r) = \sum_{\ell = 0}^{\infty}\mathrm{Tr}(\mathbf{Sym}^\ell(s_{\pi_v'} \otimes s_{\sigma_v}))q_v^{-s\ell},\]
where $\mathbf{Sym}^\ell$ is the $\ell$-th symmetric power of $r$, one is left to describe the irreducible components of each representation $\mathbf{Sym}^\ell$. At page 140 of \emph{loc.cit.} it is shown that\footnote{The formula at page 140 of \cite{ExplicitGSR} is given for the standard representation of $\Sp_{2n} \times \GL_n$, however it extends to the one of $\GSp_{2n} \times \GL_n$ after taking into account that $\GSp_{2n}$ acts on the algebra $I(E^*)$ of \textit{loc.cit.} by the symplectic multiplier.}
\[ \mathrm{Tr}(\mathbf{Sym}^\ell(s_{\pi_v'} \otimes s_{\sigma_v})) = \sum_{2 i + j = \ell}  \mathrm{Tr}( \mathbf{Sym}^i( \nu(s_{\pi_v'})\Lambda^2( s_{\sigma_v}))) \cdot \sum_{\substack{m_1\geq m_2 \geq 0 \\  m_1 + m_2 = j }} \mathrm{Tr}(s_{\pi_v'}\otimes s_{\sigma_v}|V_{(m_1,m_2,0)} \otimes W_{(m_1,m_2)}), \]
where recall that $\nu:\H(\C) \to \C^\times$ denotes the symplectic multiplier. 
This and the Poincar\'e polynomial identity show that 
 \[L(s,\pi_v'\otimes \sigma_v, r) = L(2s,\omega_{\pi_v'}\omega_{\sigma_v}) \sum_{m_1\geq m_2\geq 0}\mathrm{Tr}(s_{\pi_v'}\otimes s_{\sigma_v}|V_{(m_1,m_2,0)} \otimes W_{(m_1,m_2)})q_v^{-s (m_1+m_2)},\]
as desired.
\end{proof}

\noindent We now relate the $L$-functions of $\pi_v' \otimes \sigma_v$ and $\pi_v \otimes \sigma_v$ and obtain the following.

\begin{cor}\label{ZtoLtensorNonSplit}
We have the equality 
\[Z(s,W_{\pi_v},W_{s,v}^{\rm o}) = L(s,\pi_v \otimes \sigma_v, \wedge^2 \otimes  \mathrm{std}_2). \]   
\end{cor}

\begin{proof}
The semi-simple Frobenius conjugacy class of $\wedge^2(s_{\pi_v})$ in $\GL_6(\C)$ can be represented by \[ \chi_0(\varpi_v) {\rm diag}(1,\chi_1(\varpi_v),\chi_2(\varpi_v),\chi_1(\varpi_v)\chi_2(\varpi_v), x , y ),\] with $x$ and $y$ the roots of $X^2 - \chi_1(\varpi_v)\chi_2(\varpi_v)$ (cf. \cite[p. 12]{KimK}), while the one of $\pi_v'$ can be chosen to be \[ \chi_0(\varpi_v) {\rm diag}(1,\chi_1(\varpi_v),\chi_2(\varpi_v),\chi_1(\varpi_v)\chi_2(\varpi_v)).\] Writing the semi-simple Frobenius conjugacy class of $s_{\sigma_v}$ as ${\rm diag}(\xi_1(\varpi_v),\xi_2(\varpi_v)) \in \GL_2(\C)$, we then have that \begin{align*}
     \frac{L(s,\pi_v' \otimes \sigma_v, r)}{L(s,\pi_v \otimes \sigma_v, \wedge^2 \otimes  \mathrm{std}_2)} &= \prod_{i=1}^2(1 - x\chi_0(\varpi_v)\xi_i(\varpi_v)q_v^{-s})(1 - y\chi_0(\varpi_v)\xi_i(\varpi_v) q_v^{-s}) \\ &= \prod_{i=1}^2(1 -(\xi_i^{2}\omega_{\pi_v})(\varpi_v) q_v^{-2s}) \\ &= L(2s,\omega_{\sigma_v} \omega_{\pi_v})L(2s,\sigma_v,{\rm Sym}^2\otimes \omega_{\pi_v})^{-1}  ,
 \end{align*} 
where we used that $\omega_{\pi_v}(\varpi_v) = \omega_{\pi_v'}(\varpi_v) =  (\chi_0^2\chi_1\chi_2)(\varpi_v)$ and the explicit formula \eqref{LfunctionSigma}. This and Proposition \ref{theononsplit1} give the result. 
\end{proof}

\subsubsection{The split case}
 
Let $v$ be a split prime in $E$ and let $\pi_v$ and $\sigma_v$ be unramified representations of $\G(F_v)$ and $\GL_2(F_v)$ respectively. The isomorphism $E \otimes_F F_v \simeq F_v \times F_v$ induces an isomorphism $\GL_4(E \otimes_F F_v  ) \simeq \GL_4(F_v) \times \GL_4( F_v)$. If $g$ maps to $(g_1,g_2)$, then $\bar{g}$ is sent to $(g_2,g_1)$. Hence, we have \begin{align*}
         \G(F_v)  \simeq  \{ (g_1 , g_2, m) \in \GL_4(F_v) \times \GL_4( F_v) \times \GL_1(F_v)\,:\, g_2 = m J {}^tg_1^{-1} J^{-1}  \},  
    \end{align*} and then, projection to the first factor induces an isomorphism  $\G(F_v)  \simeq  \GL_4( F_v ) \times \GL_1(F_v), \, (g_1, g_2, m) \mapsto (g_1,m)$. This allows us to regard $\pi_v$ as a representation $\pi_v' \otimes \chi_v$ of $\GL_4(F_v) \times \GL_1(F_v)$. Furthermore, the restriction to $F_v^\times$ of the central character $\omega_{\pi_v}$ is identified with $\omega_{\pi_v'}\chi_v^2$.

Let $\{ e_1, e_2, e_3, e_4\}$ denote a basis of the standard representation of $\GL_4(\C)$ and define the symmetric form $Q$ on $\wedge^2 \C^4$ by setting  $$e_i \wedge e_j \wedge e_k \wedge e_\ell = Q(e_i \wedge e_j , e_k \wedge e_\ell) e_1 \wedge e_2 \wedge e_3 \wedge e_4$$ and then extending it by linearity. With respect to the basis  $$\{ e_1 \wedge e_2, e_1 \wedge e_3, e_1 \wedge e_4, e_2 \wedge e_3, -e_2 \wedge e_4, e_3 \wedge e_4\},$$ the symmetric form $Q$ is given by the matrix $I'_6$. Hence, there are isomorphisms $\GO_{\mathbb{H}^3}(\C) \simeq \GO_{\wedge^2 \C^4}(\C)$ as well as  $\GSO_{\mathbb{H}^3}(\C) \simeq \GSO_{\wedge^2 \C^4}(\C)$. Then, the exterior square representation defines a morphism \[ \wedge^2: \GL_4(\C) \times \GL_1(\C) \to \GO_{\mathbb{H}^3}(\C),\, (g,\lambda) \mapsto \lambda \cdot \wedge^2\!g,\] as $^t(\wedge^2\!g)I_6'(\wedge^2\!g)= {\rm det}(g)I_6'$. Moreover, ${\rm det}(\lambda \cdot \wedge^2\!g)= \lambda^6\cdot {\rm det}(g)^3$ implies that the map $\wedge^2$ factors through $\GSO_{\mathbb{H}^3}(\C)$. Hence, if we let $s_{\pi_v}$ be the Frobenius conjugacy class of $\pi_v$ (or rather $\pi'_v \otimes \chi_v$), we regard $\wedge^2(s_{\pi_v}) \in \GSO_{\mathbb{H}^3}(\C)$.

\begin{lemma}\label{LfunctionFactSplit}
We have
\[L(s,\pi_v \otimes \sigma_v, \wedge^2 \otimes \mathrm{std}_2) = L(2s,\sigma_v,\mathrm{Sym}^2\otimes\omega_{\pi_v}) \sum_{m_1\geq m_2\geq 0}\mathrm{Tr}\left(\wedge^2(s_{\pi_v})\otimes s_{\sigma_v}|V_{(m_1,m_2)}\right)q_v^{-s (m_1+m_2)},\]
where $V_{(m_1,m_2)} := V_{(m_1,m_2,0,0)}\otimes W_{(m_1,m_2)}$ is the representation of $\GSO_{\mathbb{H}^3}(\C)\times \GL_2(\C)$ obtained as the tensor product of the two irreducible algebraic representations of highest weight $(m_1,m_2,0,0)$ and $(m_1,m_2)$ respectively.
\end{lemma}
\begin{proof}
    Since $\wedge^2$ factors through $\GSO_{\mathbb{H}^3}(\C)$, this follows from a formula for the standard representation of $\GSO_{\mathbb{H}^3}(\C) \times \GL_2(\C)$ similar to \cite[(A.2.1)]{ExplicitGSR}. The proof is identical to the one of \emph{loc. cit.} and we only sketch it. Let $E = \C^{6 \times 2}$ and let $S(E^*)$ be the symmetric algebra of complex valued polynomial functions on $E$ equipped with the natural $\GSO_{\mathbb{H}^3}(\C)\times \GL_2(\C) $-action. Note that the factorization of $S(E^*)$ into irreducible $\GSO_{\mathbb{H}^3}(\C)\times \GL_2(\C) $-representations is equivalent to the one of $\mathbf{Sym}^\ell(\mathrm{std}_{\GSO_{\mathbb{H}^3}\times\GL_2})$ for all $\ell \geq 0$. Let $J(E^*)$ denote the subring of $S(E^*)$ consisting of $\SO_6(\C)$-invariant polynomial functions on $E$. Equivalently, $J(E^*)$ is the algebra of polynomials on the space of matrices of the form $Y={}^tX I_6' X$, with $X \in E$. Through this identification, a polynomial (in the variable $Y={}^tX I_6' X$) of degree $\ell$ corresponds to one of degree $2 \ell$ in the variable $X \in E$. Notice that $\GSO_{\mathbb{H}^3}(\C)$ acts on $J(E^*)$ via the multiplier. Indeed, if $g \in \GSO_{\mathbb{H}^3}(\C)$ then $g \cdot Y = {}^t(gX) \,I_6'\, gX=  {}^tX{}^tg\,I_6'\, gX= \nu'(g) Y$. The action of an element $g \in \GL_2(\C)$ on the subspace of homogeneous polynomials (in the variable $Y$) of degree $\ell$ is instead equivalent to ${\rm Sym}^\ell({\rm Sym}^2(g))$. If we let $H(E^*)$ denote the subspace of $S(E^*)$ of $\GSO_{\mathbb{H}^3}(\C)$-harmonic polynomials, by \cite[Theorem 2.5]{Ton-That} we have a ``separation of variables'' \[S(E^*) \simeq J(E^*) \otimes  H(E^*),\]
    where the isomorphism is as $ \GSO_{\mathbb{H}^3}(\C) \times \GL_2(\C)$-modules. The result then follows from decomposing $H(E^*)$ into its irreducible $ \GSO_{\mathbb{H}^3}(\C) \times \GL_2(\C)$-components, which goes as follows. Let $H(E^*)^{(m_1,m_2)}$ denote the subspace of $H(E^*)$ consisting of the polynomials which transform under $\GL_2(\C)$ according to the representation $W_{(m_1,m_2)}$ of highest weight $(m_1,m_2)$; then as a $\GSO_{\mathbb{H}^3}(\C)$-module $H(E^*)^{(m_1,m_2)}$ is equivalent to the irreducible representation $V_{(m_1,m_2,0,0)}$ by \cite[Theorem 3.1]{Ton-That}. This implies that \[ H(E^*) \simeq \bigoplus_{m_1 \geq m_2 \geq 0} V_{(m_1,m_2,0,0)} \otimes W_{(m_1,m_2)}. \] 
    Hence we have that $\mathrm{Tr}\left(\mathbf{Sym}^\ell \left(\wedge^2(s_{\pi_v})\otimes s_{\sigma_v}\right)\right)$ equals to
\[\sum_{2 i + j = \ell}  \mathrm{Tr}( \mathbf{Sym}^i( \nu'(\wedge^2 s_{\pi_v}){\rm Sym}^2( s_{\sigma_v})))  \sum_{\substack{m_1\geq m_2 \geq 0 \\  m_1 + m_2 = j }} \mathrm{Tr}( \wedge^2 (s_{\pi_v})\otimes s_{\sigma_v}|V_{(m_1,m_2,0,0)} \otimes W_{(m_1,m_2)}). \]
Since the multiplier $\nu'(\wedge^2 s_{\pi_v}) = \chi_v(\varpi_v)^2 {\rm det}(s_{\pi_v'}) = \chi_v(\varpi_v)^2 \omega_{\pi_v'}(\varpi_v) = \omega_{\pi_v}(\varpi_v)$, the result follows.
\end{proof}

We can now prove the following.

\begin{corollary}\label{Lrequaltowedge}
We have
\[Z(s,W_{\pi_v},W_{\sigma_v}) =L(s,\pi_v\otimes \sigma_v,\wedge^2 \otimes \mathrm{std}_2).\]
\end{corollary}
\begin{proof}

By the isomorphism $\G(F_v)\simeq \GL_4(F_v)\times \GL_1(F_v)$ described above, we get
\[W_{\pi_v}(t(\varpi_v^{2n+m},\varpi_v^{n+m};\varpi_v^{2n+m})) = \chi_v(\varpi_v)^{2n+m}W_{\pi_v'}(t(\varpi^{2n+m}_v,\varpi^{n+m}_v,\varpi^n_v,1)),\]
where we used the notation $t(a,b,c,d) := \mathrm{diag}(a,b,c,d)\in \GL_4(F_v)$. By the Casselman--Shalika formula we relate the values of this Whittaker functional to traces of representations of $\GSO_{\mathbb{H}^3}(\C)$ as follows. 
From now on we use the basis of $X^{\bullet}(T_{\GSO_{\mathbb{H}^3}})$ consisting of $\{\alpha_i\}_{i = 1}^4$, where when $1\leq i\leq 3$, $\alpha_i$ is the projection $T_{\GSO_{\mathbb{H}^3}} \to \GL_1$ of the $i$-th diagonal entry, and $\alpha_4$ is the similitude character. First of all, the isomorphism $X_{\bullet}(T_{\GL_4})\simeq X^{\bullet}(T_{\GL_4})$ identifies the character 
\begin{align*}\chi_{m,n}:T_{\GL_4}\to \GL_1,\; t(a,b,c,d)\mapsto a^{2n+m}b^{n+m}c^{n}\end{align*}
with the cocharacter \[\GL_1\to T_{\GL_4},\;a\mapsto t(a^{2n+m},a^{n+m},a^{n},1).\]
Under the exterior square map $\wedge^2:(\GL_4(\C) \times \GL_1(\C))/\langle (z I_4, z^{-2}),\, z \in \GL_1(\C) \rangle  \simeq \GSO_{\mathbb{H}^3}(\C)$, the character $\chi_{m,n}$ induces the character $\tilde{\chi}_{m,n}(t) := \alpha_1(t)^{m+n}\alpha_2(t)^{n}$ of $\GSO_{\mathbb{H}^3}(\C)$.
Then the Casselman--Shalika formula \cite{CasselmanShalikaII} implies that 
\begin{align*} W_{\pi_v'}(t(\varpi^{2n+m}_v,\varpi^{n+m}_v,\varpi^n_v,1)) = q_v^{-3n-2m}\, \mathrm{Tr}(\wedge^2(s_{\pi_v'})|V_{(m+n,n,0,0)}),\end{align*}
where $V_{(m+n,n,0,0)}$ is the representation of $\GSO_{\mathbb{H}^3}(\C)$ with highest weight $\tilde{\chi}_{m,n}$, and hence 
\begin{align}\label{SplitWhittGL4}W_{\pi_v}(t(\varpi_v^{2n+m},\varpi_v^{n+m};\varpi_v^{2n+m})) = q_v^{-3n-2m} \,\mathrm{Tr}(\wedge^2(s_{\pi_v})|V_{(m+n,n,0,0)}),\end{align}
since $\mathrm{Tr}(\wedge^2(s_{\pi_v})|V_{(m+n,n,0,0)}) = \chi_v(\varpi_v)^{2n+m} \mathrm{Tr}(\wedge^2(s_{\pi_v'})|V_{(m+n,n,0,0)})$.
On the other hand we have 
\begin{align}\label{SplitWhittGL2}
    W_{\sigma_v}(\hat{t}(\varpi_v^{n+m},\varpi_v^n)) = q_v^{-m/2}\, \mathrm{Tr}(s_{\sigma_v}|W_{(n+m,n)}).
\end{align} 
Using \eqref{WhittakerLocalExpr}, \eqref{SplitWhittGL2}, and \eqref{SplitWhittGL4} we see that $Z(s,W_{\pi_v},W_{\sigma_v})$ is equal to 
\begin{equation*}L(2s,\sigma_v,\mathrm{Sym}^2\otimes\omega_{\pi_v})   \sum_{n,m \geq 0} q_v^{-s(2n+m)}\, \mathrm{Tr}(\wedge^2(s_{\pi_v})|V_{(m+n,n,0,0)})\mathrm{Tr}(s_{\sigma_v}|W_{(m+n,n)}).\end{equation*}
The result then follows from Lemma \ref{LfunctionFactSplit}.
\end{proof}

\subsection{The ramified integrals}\label{ramified_integrals}

Recall that, for any place $v$ of $F$, we have set
\[Z(s,W_{\pi_v},W_{\sigma_v,s}) = \int_{U_\H(F_v)Z_\H(F_v)\setminus \H(F_v)} W_{\sigma_v,s}\left(g_v\right)W_{\pi_v}(g_v)dg_v,\]
where $W_{\sigma_v,s}$ was introduced in \S \ref{subsec_Weird_local_whitt}. Let $\Sigma_F$ denote the set of archimedean places of $F$. We assume that the product $\omega_{\pi_v} \cdot \omega_{\sigma_v}$ is unitary.

\begin{prop}\label{non:vanishing:and:abs:conv:zeta:bad}
\item[(i)]  There exists a constant $\varepsilon > 0$ so that $Z(s,W_{\pi_v},W_{\sigma_v,s})$ converges absolutely for $\mathrm{Re}(s)> 1/2-\varepsilon$.
\item[(ii)] For any finite place $v$ of $F$, there exist $W_{\pi_v}$ and $W_{\sigma_v,s}$ such that $Z(s,W_{\pi_v},W_{\sigma_v,s})$ is a non-zero constant independent of $s$. 
\item[(iii)] For any place $v \in \Sigma_F$, given $s_0 \in \C$ one can choose $W_{\pi_v}$ and $W_{\sigma_v,s}$ such that $Z(s,W_{\pi_v},W_{\sigma_v,s})$ is non-zero at $s = s_0$.
\end{prop}
\begin{proof} The proof of (i) is rather standard (see for instance \cite[\S 7, Proposition 1]{JacquetShalikaExterior}) and uses the asymptotic expressions for $W_{\sigma_v,s}$ and $W_{\pi_v}$ as follows. Using the Iwasawa decomposition, it is enough to show that 
 \begin{align}\label{eq_zeta_ram_bound}\int_{(F_v^\times)^2} |a|^{2s-3}|b|^{s-5/2}(\omega_{\pi_v}\omega_{\sigma_v})^{-1}(a)W_{\sigma_v}\left(\begin{smallmatrix}
     b & \\ & 1
 \end{smallmatrix}\right)  W_{\pi_v}\left(\begin{smallmatrix}
     ab & & & \\ & b & & \\ &  & 1 & \\ & & & a^{-1} 
 \end{smallmatrix}\right)da db,\end{align}
for $W_{\sigma_v} \in \mathcal{W}_{\sigma_v}$, converges absolutely for ${\rm Re}(s) > 1/2 - \epsilon$ for some $\epsilon >0$. From the asymptotics given in \cite[Proposition 3, p. 177]{JacquetShalikaExterior} and \cite[Proposition 3.1]{FurusawaMorimoto}, we get that 
    \[W_{\pi_v}\left(\begin{smallmatrix}
     ab & & & \\ & b & & \\ &  & 1 & \\ & & & a^{-1} 
 \end{smallmatrix}\right)= |a|^3 |b|^2\sum_{\mu\in X}\phi_{\mu}(a,b)\mu(a,b),\]
    where $X$ is a finite set of finite functions, which are linear combinations of functions on $(F_v^\times)^2$ of the form 
    \[ \mu(a_1,a_2) = \prod_{i=1}^2 \mu_i(a_i)|a_i|^{u_i} (\log |a_i|)^{n_i}, \]
 with $\mu_i$ unitary characters, $u_i \in \R$, $n_i \in \Z_{\geq 0}$, and the $\phi_\mu$'s are Schwartz functions in $\mathcal{S}(F_v^2)$. Similarly,
 \[\mathcal{W}_{\sigma_v}\left(\begin{smallmatrix}
     b & \\ & 1
 \end{smallmatrix}\right)  = \sum_{\mu'\in X'}\phi_{\mu'}(b)\mu'(b),\]
where now $X'$ is a finite set of finite functions on $F_v^\times$ and the $\phi_{\mu'}$'s are Schwartz functions in $\mathcal{S}(F_v)$. Since $\omega_{\pi_v} \cdot \omega_{\sigma_v}$ is unitary, the absolute value of
\eqref{eq_zeta_ram_bound} is bounded by 
$$ \sum_{\mu\in X,\mu' \in X'} \int_{(F_v^\times)^2} |a|^{2s}|b|^{s-1/2}| \phi_\mu(a,b) \phi_{\mu'}(b) \mu(a,b)\mu'(b)|da db.$$ 
By density, we can assume that $\phi_\mu(a,b)$ factors as a product of Schwartz functions $\phi_{\mu,1}(a) $ and $ \phi_{\mu,2}(b)$ on $F_v$, so that we are left to bound 
$$ \sum_{\mu\in X,\mu' \in X'}\int_{(F_v^\times)^2} |a|^{2s}|b|^{s-1/2}| \phi_{\mu,1}(a) \phi_{\mu,2}(b)\phi_{\mu'}(b) \mu(a,b)\mu'(b)|da db.$$
We can look at each integral separately. The integral over $a$ converges if ${\rm Re}(2s) > - \nu,$ for some $\nu > 0$, while the integral over $b$ conveges if ${\rm Re}(s) > 1/2 - \epsilon,$ for some $\epsilon > 0$. This completes the proof of (i). 
    The proofs of (ii) and (iii) follow from applying the Dixmier--Malliavin theorem as in \cite[Lemma 9.1]{GanSavinSW} and \cite[Lemma 12.3]{GanGurevichCAP}. 
     
\end{proof}

\section{Archimedean computations for $\GL_4 \times \GL_2$}\label{sec:zeta:integral:arch}

Let $F = \R$ or  $\C$. In this section, we show the following. 

\begin{theorem}\label{Theorem_Arch_computation_split}
    Let $\pi_\infty$ and $\sigma_\infty$ be irreducible spherical principal series representations of $\GL_4(F)$ and $\GL_2(F)$, respectively, such that \[ \omega_{\sigma_\infty}\omega_{\pi_\infty} = 1.\] For the rapidly decreasing spherical Whittaker functions $W_{\pi_\infty}$ and $W_{\sigma_\infty}$ of $\pi_\infty$ and $\sigma_\infty$, we have 
    \[Z(s, W_{\pi_\infty}, W_{\sigma_\infty,s}) = \frac{L(s, \pi_\infty \otimes \sigma_\infty, \wedge^2 \otimes  \mathrm{std}_2)}{L(2s, \sigma_\infty, {\rm Sym}^2 \otimes \omega_{\pi_\infty})}, \]
    where $W_{\sigma_\infty,s} \in {\rm Ind}_{P_{(1,2)}(F)}^{\H(F)}\left((\mathbf{1} \otimes\mathcal{W}_{\sigma_\infty}) \delta_{P_{(1,2)}}^{\tfrac{1}{2}(s-1/2)}\right)$ is the spherical vector related to $W_{\sigma_\infty}$ by \eqref{SphericalWhittakerforKES}. 
\end{theorem}

\begin{remark} \leavevmode
\begin{enumerate}
    \item In \cite{HiranoIshiiMiyazakiGL4}, the authors give explicit formulas of Whittaker functions on $\GL_4(\R)$ for all irreducible generic representations. We expect that these formulae will allow us to extend Theorem \ref{Theorem_Arch_computation_split} to any irreducible generic representation of $\GL_4(\R) \times \GL_2(\R)$. We will return to these calculations in a separate work.
    \item In \cite{CauchiGutiArchGU}, we carry out the archimedean computation of the zeta integral for generic discrete series on $\GU_{2,2}(\R) \times \GL_2(\R)$.
\end{enumerate}
\end{remark}

\subsection{Auxiliary formulae for Barnes integrals}

Recall that the Gamma function $\Gamma(s)$ is holomorphic on $\C \setminus \{0, -1,-2,\dots  \}$ and has simple poles at $s= -m$, for $m \in \Z_{\geq 0}$. Define 
$$\Gamma_\R(s) : = \pi^{-s/2} \Gamma(s/2), \, \,\,\, \Gamma_\C(s) := 2 (2 \pi)^{-s} \Gamma(s).$$
Note that $$\Gamma_\R(s) \Gamma_\R(s+1) = \Gamma_\C(s).$$

We now recall some identities for Barnes integrals, which are heavily used in the proof of Theorem \ref{Theorem_Arch_computation_split}. We mainly follow the expositions of \cite[\S 1]{IshiiMultivariate} and \cite[\S 8.2]{HiranoIshiiMiyazakiGL3}. 

\begin{lemma}[{Barnes' first Lemma - \cite[Lemma 8.5]{HiranoIshiiMiyazakiGL3}}]\label{BarnesFirst}
    For $a_1,a_2,b_1,b_2 \in \C$ such that ${\rm Re}(a_i + b_j) >0 $ for $1\leq i,j \leq 2$, we have that 
    \begin{align*}
        \frac{1}{4 \pi i} \int_z \Gamma_F(z + a_1) \Gamma_F(z + a_2) \Gamma_F(-z + b_1) \Gamma_F(-z + b_2) dz = \frac{\Gamma_F(a_1+b_1)\Gamma_F(a_1+b_2)\Gamma_F(a_2+b_1)\Gamma_F(a_2+b_2)}{\Gamma_F(a_1+a_2+ b_1 + b_2)},
    \end{align*}
    where the path of integration $\int_z$ is a vertical line from $\sigma - i \infty $ to $\sigma + i \infty $ such that $$ {\rm max} \{ - {\rm Re}(a_1), - {\rm Re}(a_2)\} <\sigma < {\rm min}\{{\rm Re}(b_1), {\rm Re}(b_2) \}.$$
\end{lemma}

\begin{lemma}[{Barnes' second Lemma - \cite[Lemma 8.7]{HiranoIshiiMiyazakiGL3}}]\label{BarnesSecond}
    For $a_1,a_2,b_1,b_2,b_3 \in \C$ such that ${\rm Re}(a_i + b_j) >0 $ for $1\leq i \leq 2, 1 \leq j \leq 3$, we have that 
    \begin{align*}
        \frac{1}{4 \pi i} &\int_z \frac{\Gamma_F(z + a_1) \Gamma_F(z + a_2) \Gamma_F(-z + b_1) \Gamma_F(-z + b_2)\Gamma_F(-z + b_3)}{\Gamma_F(-z + a_1+a_2+b_1+b_2+b_3 )} dz \\ &= \frac{\Gamma_F(a_1+b_1)\Gamma_F(a_1+b_2)\Gamma_F(a_1+b_3)\Gamma_F(a_2+b_1)\Gamma_F(a_2+b_2)\Gamma_F(a_2+b_3)}{\Gamma_F(a_1+a_2+ b_1 + b_2)\Gamma_F(a_1+a_2+ b_1 + b_3)\Gamma_F(a_1+a_2+ b_2 + b_3)},
    \end{align*}
    where the path of integration $\int_z$ is a vertical line from $\sigma - i \infty $ to $\sigma + i \infty $ such that $$ {\rm max} \{ - {\rm Re}(a_1), - {\rm Re}(a_2)\} < \sigma < {\rm min}\{{\rm Re}(b_1), {\rm Re}(b_2),{\rm Re}(b_3) \}.$$
\end{lemma}

We conclude with the description of a formula due to Stade (see also \cite[Lemma 1.5]{IshiiMultivariate}).

\begin{lemma}[{ \cite[Lemma 2.1]{Stade}}]\label{StadeLemma}
    For $a_1,a_2,a_3,a_4,b_1,b_2,b_3,b_4 \in \C$, let $I(a_1,a_2,a_3,a_4,b_1,b_2,b_3,b_4)$ be the integral 
    
    \begin{align*}
        \frac{1}{4 \pi i} &\int_z \frac{\Gamma_F(z + a_1) \Gamma_F(z + a_2)\Gamma_F(z + a_3) \Gamma_F(-z + b_1) \Gamma_F(-z + b_2)\Gamma_F(-z + b_3)}{\Gamma_F(z + a_4 )\Gamma_F(-z + b_4)} dz,
    \end{align*}
    where the path of integration $\int_z$ is a vertical line that separates the increasing and decreasing sequences of poles. If $\sum_{i=1}^3 (a_i + b_i) = a_4 + b_4$, we have 
    $$ I(a_1,a_2,a_3,a_4,b_1,b_2,b_3,b_4) = I(a_1,a_2,b_4 - b_1 - b_2,a_1+a_2+b_3,b_1,b_2, a_4 - a_1 - a_2,b_1 + b_2 + a_3).$$
\end{lemma}

\subsection{Spherical principal series and their Whittaker functionals}\label{sec:SphWhittakerFormulae}

Define $\epsilon =$ 1 or 2 and put $K_{\GL_n(F)}$ equal to $\mathbf{O}_n$ or $\mathbf{U}_n$ (for $n \in \{2,4\}$), depending on whether $F = \R$ or  $\C$. 

\subsubsection{Principal series on $\GL_4(F)$}

Recall the Langlands decomposition $B_{\GL_4}(F) = M_{\GL_4}(F) A_{\GL_4} U_{\GL_4}(F)$ of the upper triangular Borel subgroup of $\GL_4(F)$, with \begin{align*}
    M_{\GL_4}(F)&:= \{{\rm diag}(m_1,m_2,m_3,m_4)\in \GL_4(F)\,:\, |m_i| = 1\,\,\, \forall i \}, \\
     A_{\GL_4} &:= \{{\rm diag}(a_1,a_2,a_3,a_4)\in \GL_4(\R)\,:\, a_i> 0 \,\,\, \forall i \}.
\end{align*} 
For $\mu=(\mu_1,\mu_2,\mu_3,\mu_4) \in \C^4,$ define the character $\chi_\mu$ on $M_{\GL_4}(F)A_{\GL_4}$ by $\chi_\mu(m a) := \prod_{i=1}^4 a_i^{\epsilon \mu_i}$. Let $\pi_\infty$ be an irreducible spherical principal series of $\GL_4(F)$ given by the (normalized) induction ${\rm Ind}_{B_{\GL_4}(F)}^{\GL_4(F)}(\chi_{\mu})$. The representation $\pi_\infty$ admits a unique rapidly decreasing right $K_{\GL_4(F)}$-invariant Whittaker functional up to scalars, which we denote by $W_{4,\mu}^F$. In view of the Iwasawa decomposition of $\GL_4(F)$, $W_{4,\mu}^F$ is determined by its restriction to $A_{\GL_4}$. We refer to \cite{IshiiStade} and \cite{IshiiMultivariate} (and the references therein) for more details on the following formula.

\begin{proposition}{\cite[Proposition 3.4]{IshiiMultivariate}}\label{WhittakerSphericalForGL4} Let $| \mu| : = \sum_i \mu_i$. Up to constant multiple, the $A_{\GL_4}$-radial part of $W_{4,\mu}^F$ is of the following form:
\begin{align*}
    W_{4,\mu}^F({\rm diag}(a_1,a_2,a_3,a_4)) &= \frac{2^{ \epsilon} a_1^{3 \epsilon/2} a_2^{\epsilon/2} a_3^{- \epsilon/2} a_4^{-3 \epsilon/2 + \epsilon|\mu|} }{(4 \pi i)^3}  \\ &\cdot \int_{p_1,p_2,p_3} V_\mu^F(p_1,p_2,p_3) \left(\frac{a_1}{a_2} \right)^{-\epsilon p_1}\left(\frac{a_2}{a_3} \right)^{-\epsilon p_2}\left(\frac{a_3}{a_4} \right)^{-\epsilon p_3} dp_1 dp_2 dp_3,
\end{align*}
    where 
    \begin{align*}
   V_\mu^F(p_1,p_2,p_3) &:= \frac{1}{(4 \pi i)^2} \int_{q_1,q_2} U_\mu^F(q_1,q_2) \Gamma_F(p_1 + \mu_1)\Gamma_F(p_1 - q_1)\Gamma_F(p_2 -q_1 + \mu_1)  \\ & \cdot \Gamma_F(p_2 -q_2 - \mu_1)\Gamma_F(p_3 - q_2)\Gamma_F(p_3 + |\mu| - \mu_1) dq_1 dq_2,
\end{align*}
with $$U_\mu^F(q_1,q_2) := \frac{\prod_{i=2}^4 \Gamma_F(q_1 +\mu_i)\Gamma_F(q_2+|\mu|-\mu_i)}{\Gamma_F(q_1 + q_2 +|\mu|)}.$$
\end{proposition}

\subsubsection{Principal series on $\GL_2(F)$}\label{subsec:Principal_GL2}
Here we closely follow \cite[\S 2]{IshiiMultivariate} and hence adopt their normalizations. Recall the Langlands decomposition $B_{\GL_2}(F) = M_{\GL_2}(F) A_{\GL_2} U_{\GL_2}(F)$ of the upper triangular Borel subgroup of $\GL_2(F)$, where we choose \begin{align*}
    M_{\GL_2}(F) &:= \{{\rm diag}(m_1,m_2)\in \GL_2(F)\,:\, |m_1|=|m_2| = 1\}, \\
     A_{\GL_2} &:= \{{\rm diag}(a_0a_1,a_1^{-1})\in \GL_2(\R)\,:\, a_0,a_1> 0 \}.
\end{align*} 
For $\nu=(\nu_0,\nu_1) \in \C^2$, we let $\chi_\nu$ be the character on $M_{\GL_2}(F) A_{\GL_2}$ given by $\chi_\nu(m a) := \prod_{i=0}^1 a_i^{\epsilon \nu_i}$. Let $\sigma_\infty$ be an irreducible spherical principal series of $\GL_2(F)$ of the form ${\rm Ind}_{B_{\GL_2}(F)}^{\GL_2(F)}(\chi_{\nu})$. It has a unique rapidly decreasing right $K_{\GL_2(F)}$-invariant Whittaker functional up to constant, which we denote by $W_{2,\nu}^F$. Note that $W_{2,\nu}^F$  is determined by its restriction to $A_{\GL_2}$. 

\begin{proposition}{\cite[Proposition 2.5]{IshiiMultivariate}} \label{WhittakerSphericalForGL2} Up to constant multiple, the $A_{\GL_2}$-radial part of $W_{2,\nu}^F$ is of the following form:
\begin{align*}
    W_{2,\nu}^F({\rm diag}(a_0a_1,a_1^{-1})) &= \frac{2^{ \epsilon} a_0^{\epsilon(\nu_0 - \nu_1/2) + \epsilon/2} a_1^{\epsilon} }{4 \pi i} \int_{p_0} \Gamma_F(p_0 + \tfrac{\nu_1}{2})\Gamma_F(p_0 - \tfrac{\nu_1}{2}) \left(a_0 a_1^2 \right)^{-\epsilon p_0} dp_0.
\end{align*}
\end{proposition}

\subsection{Archimedean $L$-factors for the $\wedge^2 \otimes {\rm std}$ $L$-functions}\label{subsec:ArchimedeanLfactors}

The main references for this section are \cite{KnappLLCA} and \cite[\S 6.1]{HiranoIshiiMiyazakiGL4}. 
Recall that the Weil group $W_\R$ of $\R$ is the non-split extension of $\C^\times$ by $\Z/2\Z$ given by $$ W_\R = \C^\times \cup \C^\times  {\rm j},$$
where $ {\rm j}^2 = -1$ and $ {\rm j} z  {\rm j}^{-1} = \overline{z}$ for all $z \in \C$ (with $\overline{\bullet}$ denoting complex conjugation). The Weil group of $\C$ is given by $W_\C = \C^\times$. 

We recall the theory of finite dimensional semisimple representation of $W_F$.

\subsubsection{Representations of $W_\R$}

Firstly, recall the irreducible representations of $W_\R$ and their $L$-factors. 

\begin{enumerate}
    \item Characters: Given $\nu \in \C$ and $\delta \in \{ 0,1\}$, let $\phi_\nu^\delta$ be the character of $W_\R$ given by $$ \phi_\nu^\delta(z) = |z|^{2\nu}\,\,\,(z \in \C), \,\,\,\,\,\,\, \phi_\nu^\delta({\rm j})=(-1)^\delta.$$
    To $\phi_\nu^\delta$, we attach an $L$-factor given by $$ L(s,\phi_\nu^\delta) := \Gamma_\R(s+ \nu + \delta).$$
    \item Two dimensional representations: Given $\nu \in \C, \kappa \in \Z_{\geq 1}$, let $\phi_{\nu,\kappa}: W_\R \to \GL_2(\C)$ defined by 
    \begin{align*}
        \phi_{\nu,\kappa}( r e^{i \theta}) = \left( \begin{smallmatrix}
            r^{2 \nu} e^{i \kappa \theta} & 0 \\ 0 & r^{2 \nu} e^{-i \kappa \theta}
        \end{smallmatrix} \right)\,\,\,(r \in \R_{>0},\,\theta \in \R),\,\,\, \phi_{\nu,\kappa}( {\rm j}) = \left( \begin{smallmatrix}
            0 & (-1)^\kappa \\ 1 & 0
        \end{smallmatrix} \right).
    \end{align*}
    To $\phi_{\nu,\kappa}$, we attach an $L$-factor given by $$ L(s,\phi_{\nu,\kappa}) := \Gamma_\C(s+ \nu + \tfrac{\kappa}{2}).$$
\end{enumerate}
The set of equivalence classes of irreducible representations of $W_\R$ is exhausted by this list. Moreover, by the Lemma in page 403 of \cite{KnappLLCA}, every finite dimensional semisimple complex representation of $W_\R$ is fully reducible, hence a direct sum of $\phi_\nu^\delta$'s or $\phi_{\nu,\kappa}$'s. For a finite dimensional semisimple complex representation $\phi$ of $W_\R$, we define $$L(s,\phi) = \prod_{i=1}^m L(s, \phi_i),$$ if $\phi \simeq \oplus_{i=1}^m \phi_i$ is the decomposition of $\phi$ into irreducibles. 

\subsubsection{Representations of $W_\C$}

For $z \in \C$, let $\boldsymbol{[} z \boldsymbol{]} = z /|z|$. Given $\ell \in \Z$ and $t \in \C$, let $\varphi_{\ell,t}$ be the character on $W_\C$ defined by $$\varphi_{\ell,t} \,:\, z \mapsto  \boldsymbol{[} z \boldsymbol{]}^\ell |z|^{2t}.$$
All irreducible representations of $W_\C$ are of this form. Moreover, since $W_\C$ is abelian, every  $n$-dimensional semisimple complex representation $\varphi$ of $W_\C$ can be written as the direct sum $\oplus_{j=1}^n\varphi_{\ell_j,t_j}$ with $(\ell_j,t_j) \in \Z \times \C$ for all $j \in \{1, \dots, n\}$.

The $L$-factor of $\varphi_{\ell,t}$ is given by $$L(s,\varphi_{\ell,t}) := \Gamma_\C(s+ t + \tfrac{\ell}{2}).$$ 
Thus, if $\varphi = \oplus_{j=1}^n\varphi_{\ell_j,t_j}$, define 
$$L(s,\varphi) := \prod_{j=1}^nL(s,\varphi_{\ell_j,t_j}).$$

\subsubsection{Local Langlands Correspondence and $L$-factors}

The local Langlands correspondence for $\GL_n(F)$ is a bijection between the set of infinitesimal equivalence classes of irreducible admissible representations of $\GL_n(F)$ and the set of equivalence classes of $n$-dimensional semisimple complex representations of $W_F$ (see \cite[Theorems 2 \& 5]{KnappLLCA}). For an irreducible admissible representation $\pi_\infty$ of $\GL_n(F)$, the corresponding representation  $\phi[\pi_\infty]$ of $W_F$ is called Langlands parameter of $\pi_\infty$. Define the $L$-factor for $\pi_\infty$ as

$$ L(s,\pi_\infty) := L(s, \phi[\pi_\infty]) =  \prod_{i=1}^m L(s, \phi_i) \text{ if } \phi[\pi_\infty] \simeq \oplus_{i=1}^m \phi_i.$$

Recall that $\wedge^2:\GL_4(\C) \to \GL_6(\C)$ denotes the exterior square representations and that we have a 12 dimensional complex representation on $\GL_4(\C)\times \GL_2(\C)$ given by the tensor product $\wedge^2 \otimes {\rm std}_2$. If $\pi_\infty$ and $\sigma_\infty$ are irreducible admissible representations on $\GL_4(F)$ and $\GL_2(F)$, respectively, we let $$L(s, \pi_\infty \otimes \sigma_\infty, \wedge^2 \otimes {\rm std}_2) = L(s, \wedge^2(\phi[\pi_\infty]) \otimes \phi[\sigma_\infty]).$$

\noindent We calculate this explicitly when $\pi_\infty = {\rm Ind}_{B_{\GL_4}(F)}^{\GL_4(F)}(\chi_{\mu})$ and $\sigma_\infty = {\rm Ind}_{B_{\GL_2}(F)}^{\GL_2(F)}(\chi_{\nu})$ are the irreducible spherical principal series defined in \S \ref{sec:SphWhittakerFormulae}. 

\begin{lemma}\label{L_factor_for_spherical_arch}
   Let $\pi_\infty = {\rm Ind}_{B_{\GL_4}(F)}^{\GL_4(F)}(\chi_{\mu})$ and $\sigma_\infty = {\rm Ind}_{B_{\GL_2}(F)}^{\GL_2(F)}(\chi_{\nu})$ be irreducible spherical principal series with $\mu = (\mu_1,\mu_2,\mu_3,\mu_4) \in \C^4$ and $\nu=(\nu_0,\nu_1) \in \C^2$. Then we have that
   \begin{enumerate}
       \item \begin{align*}
L(s, \pi_\infty \otimes \sigma_\infty, \wedge^2 \otimes {\rm std}_2) &=  \prod_{1\leq i< j \leq 4} \Gamma_F(s +\mu_i+\mu_j + \nu_0) \Gamma_F(s+\mu_i+\mu_j + \nu_0-\nu_1),
\end{align*}
\item  \begin{align*} L(s, {\rm Sym}^2(\sigma_\infty) \otimes \omega_{\pi_\infty} ) = \Gamma_F(s+ \nu_0 + |\mu|) \Gamma_F(s+ 2\nu_0 - 2\nu_1 + |\mu|)\Gamma_F(s+ 2\nu_0 - \nu_1 + |\mu|).
\end{align*}
   \end{enumerate}

\end{lemma}
\begin{proof}
\begin{enumerate}
    \item Suppose that $F=\R$. Then, we have 
$$\phi[\pi_\infty] = \oplus_{i=1}^4 \phi_{\mu_i}^0, \, \, \,\;\;\;\;\;\; \phi[\sigma_\infty] = \phi^0_{\nu_0} \oplus \phi^0_{\nu_0-\nu_1}.$$
Hence, \begin{align*}
\wedge^2(\phi[\pi_\infty]) \otimes \phi[\sigma_\infty] &= \left(\bigoplus_{1\leq i< j \leq 4} \phi_{\mu_i+\mu_j}^0 \right) \otimes \left (   \phi^0_{\nu_0} \oplus \phi^0_{\nu_0-\nu_1} \right) \\ 
&= \bigoplus_{1\leq i< j \leq 4}\left( \phi_{\mu_i+\mu_j + \nu_0}^0 \oplus \phi_{\mu_i+\mu_j + \nu_0-\nu_1}^0 \right),
\end{align*}
and 

\begin{align*}
L(s, \pi_\infty \otimes \sigma_\infty, \wedge^2 \otimes {\rm std}_2) &=  \prod_{1\leq i< j \leq 4} L(s,\phi_{\mu_i+\mu_j + \nu_0}^0) \cdot L(s, \phi_{\mu_i+\mu_j + \nu_0-\nu_1}^0) \\ 
&= \prod_{1\leq i< j \leq 4} \Gamma_\R(s +\mu_i+\mu_j + \nu_0) \Gamma_\R(s+\mu_i+\mu_j + \nu_0-\nu_1).
\end{align*}

Similarly, when $F= \C$, we have that $$\phi[\pi_\infty] = \oplus_{i=1}^4 \varphi_{0,\mu_i}, \, \, \,\;\;\;\;\;\; \phi[\sigma_\infty] = \varphi_{0,\nu_0} \oplus \varphi_{0,\nu_0-\nu_1},$$
which shows that
\begin{align*}
L(s, \pi_\infty \otimes \sigma_\infty, \wedge^2 \otimes {\rm std}_2) &=  \prod_{1\leq i< j \leq 4} \Gamma_\C(s +\mu_i+\mu_j + \nu_0) \Gamma_\C(s+\mu_i+\mu_j + \nu_0-\nu_1).
\end{align*}
\item  We show the formula when $F = \R$, with the other case being identical.  Note that $$ {\rm Sym}^2(\phi[\sigma_\infty] ) = \phi^0_{2\nu_0} \oplus \phi^0_{2\nu_0-\nu_1} \oplus \phi^0_{2\nu_0-2\nu_1}.$$
Hence, twisting by $\omega_{\pi_\infty}$ gives the desired formula. 
\end{enumerate}
\end{proof}

\subsection{Proof of Theorem \ref{Theorem_Arch_computation_split}}

Let $\pi_\infty = {\rm Ind}_{B_{\GL_4}(F)}^{\GL_4(F)}(\chi_{\mu})$ and $\sigma_\infty = {\rm Ind}_{B_{\GL_2}(F)}^{\GL_2(F)}(\chi_{\nu})$ be irreducible spherical principal series whose product of central characters is trivial. This condition translates into saying that \begin{align}\label{eq:CC}
    2 \nu_0 - \nu_1 = -| \mu|.
\end{align}

Denote $\H = \GSp_4$. We will evaluate the archimedean local zeta integral at the rapidly decreasing spherical Whittaker functionals $W_{2,\nu}^F$ and $W_{4,\mu}^F$:

\[Z(s, W_{4,\mu}^F, W_{2,\nu,s}^F) =\int_{U_\H(F)Z_\H(F)\setminus \H(F)} W_{2,\nu,s}^F(g)W_{4,\mu}^F(g)dg,\]
where \[W_{2,\nu,s}^F \in {\rm Ind}_{P_{(1,2)}(F)}^{\H(F)}\left((\mathbf{1} \otimes\mathcal{W}_{\sigma_\infty}) \delta_{P_{(1,2)}}^{\tfrac{1}{2}(s-1/2)}\right)\]
is the spherical vector related to $W_{2,\nu}^F$ by \eqref{SphericalWhittakerforKES}. We claim that \[Z(s, W_{4,\mu}^F, W_{2,\nu,s}^F) = \frac{L(s, \pi_\infty \otimes \sigma_\infty, \wedge^2 \otimes  \mathrm{std}_2)}{L(2s, \sigma_\infty, {\rm Sym}^2 \otimes \omega_{\pi_\infty})}. \]
The proof of the claim consists mainly of two steps. In the first one, we write $Z(s, W_{4,\mu}^F, W_{2,\nu,s}^F)$ as an integral of Mellin--Barnes type by means of the explicit formulae for the class one Whittaker functionals  as given in Propositions \ref{WhittakerSphericalForGL4} and \ref{WhittakerSphericalForGL2}. The second step involves the repeated application of Barnes' first and second Lemmas (Lemmas \ref{BarnesFirst} and \ref{BarnesSecond}), together with an auxiliary result due to Stade (Lemma \ref{StadeLemma}). 

\subsubsection{Step 1: Reduction to a Mellin--Barnes integral} 
The Iwasawa decomposition for $\H(F)$, together with the fact that $\pi_\infty$ and $\sigma_\infty$ are spherical, let us write 
\[ Z(s, W_{4,\mu}^F, W_{2,\nu,s}^F) 
 = \int_{\R_{>0}\setminus A^\circ_{\H,\R}}\delta_{B_\H}^{-1}(t) W_{2,\nu,s}^F(t)W_{4,\mu}^F(t)dt,\]
 where $A^\circ_{\H,\R} =\{ {\rm diag}(a_1,a_2,a_0a_2^{-1},a_0a_1^{-1}) \in \H(\R)\,:\, a_0,a_1,a_2 >0  \}$. Here, we are normalizing the Haar measure $dg$ to give volume 1 on the maximal compact subgroup of $\H(F)$.  The quotient $\R_{>0}\setminus A^\circ_{\H,\R}$ can be written as $$\left\{a(y_1 y_2^{1/2},y_2^{1/2}) := {\rm diag}\left ( y_1 y_2^{1/2}, y_2^{1/2},y_2^{-1/2}, y_1^{-1} y_2^{-1/2}\right)\,:\, y_1,y_2 >0 \right\},$$ in which case $$ \delta_{B_\H}^{-1}(a(y_1 y_2^{1/2},y_2^{1/2})) = y_1^{-4\epsilon}y_2^{-3\epsilon}.$$
Thus, \[ Z(s, W_{4,\mu}^F, W_{2,\nu,s}^F) 
 = \int_{\R_{>0}^2}y_1^{-4\epsilon}y_2^{-3\epsilon} W_{2,\nu,s}^F(a(y_1 y_2^{1/2},y_2^{1/2}))W_{4,\mu}^F(a(y_1 y_2^{1/2},y_2^{1/2})) \frac{d y_1}{y_1} \frac{d y_2}{y_2}.\]
By \eqref{SphericalWhittakerforKES}, we have 
\begin{align}
     W_{2,\nu,s}^F(a(y_1 y_2^{1/2},y_2^{1/2})) &= \delta_{P_{(1,2)}}^{\tfrac{1}{2}(s+1/2)}(a(y_1 y_2^{1/2},y_2^{1/2})) \cdot W_{2,\nu}^F({\rm diag}(y_2^{1/2},y_2^{-1/2}))\nonumber \\
     &= \left(y_1 y_2^{1/2}\right)^{(2s+1)\epsilon}W_{2,\nu}^F({\rm diag}(y_2^{1/2},y_2^{-1/2})). \label{eq:formulasphericalarchimedeanwhittinductionQ}
\end{align}
Plugging  \eqref{eq:formulasphericalarchimedeanwhittinductionQ}, we write 
\[ Z(s, W_{4,\mu}^F, W_{2,\nu,s}^F) 
 = \int_{\R_{>0}^2}y_1^{(2s-3)\epsilon}y_2^{(s-5/2)\epsilon} W_{2,\nu}^F({\rm diag}(y_2^{1/2},y_2^{-1/2}))W_{4,\mu}^F(a(y_1 y_2^{1/2},y_2^{1/2})) \frac{d y_1}{y_1} \frac{d y_2}{y_2}.\]
We now use the explicit formulae of $W_{2,\nu}^F$ and $W_{4,\mu}^F$ given in Propositions \ref{WhittakerSphericalForGL4} and \ref{WhittakerSphericalForGL2}: 

\begin{align*}
    Z(s, W_{4,\mu}^F, W_{2,\nu,s}^F) 
 &= \frac{2^\epsilon}{4\pi i}\int_{y_1,y_2} \int_{p_0} y_1^{\epsilon(2s-3)}y_2^{\epsilon(s-2-p_0)} \Gamma_F(p_0 + \tfrac{\nu_1}{2})\Gamma_F(p_0 - \tfrac{\nu_1}{2})W_{4,\mu}^F(a(y_1 y_2^{1/2},y_2^{1/2})) d p_0\, \frac{d y_1}{y_1} \frac{d y_2}{y_2} \\
 &=\frac{2^{2\epsilon}}{(4\pi i)^6}\int_{y_1,y_2}\int_{p_0,p_1,p_2,p_3,q_1,q_2} \!\!\!\! \!\!\!\!y_1^{\epsilon(2s - |\mu|-p_1-p_3)} y_2^{\epsilon(s -\tfrac{|\mu|}{2} -p_0 - p_2)} U_\mu^F(q_1,q_2)\cdot \\
 &\cdot \Gamma_F(p_0 + \tfrac{\nu_1}{2})\Gamma_F(p_0 - \tfrac{\nu_1}{2}) \Gamma_F(p_1 + \mu_1)\Gamma_F(p_1 - q_1)\Gamma_F(p_2 -q_1 + \mu_1) \cdot \\ & \cdot \Gamma_F(p_2 -q_2 - \mu_1)\Gamma_F(p_3 - q_2)\Gamma_F(p_3 + |\mu| - \mu_1) d p_0 d p_1 d p_2 d p_3 dq_1 dq_2 \frac{d y_1}{y_1} \frac{d y_2}{y_2}.
\end{align*}
Firstly, we isolate the integrals over $y_2,p_0$ and $y_1,p_1$. By Mellin inversion Theorem, we have 

\begin{align*}
    \int_{y_2} \int_{p_0}y_2^{\epsilon(s -\tfrac{|\mu|}{2} -p_0 - p_2)}  \Gamma_F(p_0 + \tfrac{\nu_1}{2})\Gamma_F(p_0 - \tfrac{\nu_1}{2}) dp_0 \frac{d y_2}{y_2}  &=  2^{2-\epsilon} (\pi i) \Gamma_F(s -\tfrac{|\mu|}{2} - p_2 + \tfrac{\nu_1}{2})\Gamma_F(s -\tfrac{|\mu|}{2} - p_2 - \tfrac{\nu_1}{2}), \\ 
    \int_{y_1} \int_{p_1} y_1^{\epsilon(2s - |\mu|-p_1-p_3)}  \Gamma_F(p_1 + \mu_1)\Gamma_F(p_1 - q_1) dp_1 \frac{d y_1}{y_1}  &=  2^{2-\epsilon} (\pi i) \Gamma_F(2s -|\mu| - p_3 + \mu_1)\Gamma_F(2s -|\mu| - p_3-q_1).
\end{align*} 
Hence $Z(s, W_{4,\mu}^F, W_{2,\nu,s}^F) $ becomes

\begin{align}\label{eq:final_step_1}
   &\frac{1}{(4\pi i)^4}\int_{p_2,p_3,q_1,q_2} \!\!\!\! \!\!\!\! U_\mu^F(q_1,q_2) \Gamma_F(s -\tfrac{|\mu|}{2} - p_2 + \tfrac{\nu_1}{2})\Gamma_F(s -\tfrac{|\mu|}{2} - p_2 - \tfrac{\nu_1}{2})\Gamma_F(2s -|\mu| - p_3 + \mu_1)\cdot \nonumber \\
 &\cdot \Gamma_F(2s -|\mu| - p_3-q_1) \Gamma_F(p_2 -q_1 + \mu_1)  \Gamma_F(p_2 -q_2 - \mu_1)\Gamma_F(p_3 - q_2)\Gamma_F(p_3 + |\mu| - \mu_1) d p_2 d p_3 dq_1 dq_2.
\end{align}

\subsubsection{Step 2: Connection to $L$-factors}

In the following, we solve the integral \eqref{eq:final_step_1} explicitly. Firstly, we apply Barnes' first Lemma for the integrals over $p_2$ and $p_3$. By Lemma \ref{BarnesFirst}, we have
\small
\begin{align*}
    \frac{1}{4 \pi i} \int_{p_2}\Gamma_F(p_2 -q_1 + \mu_1)  \Gamma_F(p_2 -q_2 - \mu_1)\Gamma_F(s -\tfrac{|\mu|}{2} - p_2 + \tfrac{\nu_1}{2})\Gamma_F(s -\tfrac{|\mu|}{2} - p_2 - \tfrac{\nu_1}{2}) d p_2  =\frac{A_1(s,q_1,q_2)}{\Gamma_F(2s - |\mu| - q_1 -q_2)},
\end{align*}
\normalsize
where we have denoted $$A_1(s,q_1,q_2) = \Gamma_F(s - \tfrac{|\mu|}{2} + \tfrac{\nu_1}{2} - q_1 + \mu_1)\Gamma_F(s - \tfrac{|\mu|}{2} - \tfrac{\nu_1}{2} - q_1 + \mu_1)\Gamma_F(s - \tfrac{|\mu|}{2} + \tfrac{\nu_1}{2} - q_2 - \mu_1)\Gamma_F(s - \tfrac{|\mu|}{2} - \tfrac{\nu_1}{2} - q_2 - \mu_1),$$
and 
\small
\begin{align*}
    \frac{1}{4 \pi i} \int_{p_3}&\Gamma_F(p_3 - q_2)\Gamma_F(p_3 + |\mu| - \mu_1)\Gamma_F(2s -|\mu| - p_3 + \mu_1) \Gamma_F(2s -|\mu| - p_3-q_1)d p_3 = \frac{\Gamma_F(2s)A_2(s,q_1,q_2)}{\Gamma_F(4s - |\mu| - q_1 -q_2)},
\end{align*} 
\normalsize
where $$A_2(s,q_1,q_2) = \Gamma_F(2s - |\mu| + \mu_1 - q_2)\Gamma_F(2s - |\mu| - q_1 - q_2)\Gamma_F(2s -  q_1 - \mu_1).$$

\noindent These two equalities let us write $Z(s, W_{4,\mu}^F, W_{2,\nu,s}^F) $ as 
\begin{align}\label{eq:intermediate_step_archComp_afterB1L}
   &\frac{\Gamma_F(2s)}{(4\pi i)^2}\int_{q_1,q_2}  \frac{U_\mu^F(q_1,q_2) A_1(s,q_1,q_2)A_2(s,q_1,q_2)}{\Gamma_F(2s - |\mu| - q_1 -q_2)\Gamma_F(4s - |\mu| - q_1 -q_2)}  dq_1 dq_2.
\end{align}
Recall that $$U_\mu^F(q_1,q_2) = \frac{\prod_{i=2}^4 \Gamma_F(q_1 +\mu_i)\Gamma_F(q_2+|\mu|-\mu_i)}{\Gamma_F(q_1 + q_2 +|\mu|)}.$$

In order to apply Barnes' second Lemma for the two remaining integrals, we use Lemma \ref{StadeLemma} for the integral over the variable $q_1$, with  
    \begin{align*}
        a_1 &= \mu_2,a_2 = \mu_3,a_3=\mu_4,a_4= q_2 + |\mu|, \\  b_1 &= 2s-\mu_1,b_2 = s - \tfrac{\mu}{2} + \tfrac{\nu}{2} + \mu_1 ,b_3 = s - \tfrac{\mu}{2} - \tfrac{\nu}{2} + \mu_1 ,b_4 = 4s-q_2-|\mu|,  
    \end{align*} 
to write \eqref{eq:intermediate_step_archComp_afterB1L} as 
\small
\begin{align*}
    \frac{ \Gamma_F(2s)\Gamma_F(s + \mu_1 + \mu_2 - \tfrac{|\mu|}{2}-\tfrac{\nu_1}{2})\Gamma_F(s + \mu_1 + \mu_3 - \tfrac{|\mu|}{2}-\tfrac{\nu_1}{2})\Gamma_F(2s-\mu_1 + \mu_4)\Gamma_F(s + \mu_1 + \mu_4 - \tfrac{|\mu|}{2}+\tfrac{\nu_1}{2})}{(4\pi i)^2} \cdot \\ \int_{q_1,q_2} \frac{\Gamma_F(q_2+|\mu|-\mu_4)\Gamma_F(q_2 - q_1 + \mu_1 +\mu_4) \Gamma_F(-q_2+s - \tfrac{|\mu|}{2} - \mu_1 +\tfrac{\nu_1}{2})\Gamma_F(-q_2+s - \tfrac{|\mu|}{2} - \mu_1 -\tfrac{\nu_1}{2})}{\Gamma_F(-q_2 + 3s -  \tfrac{|\mu|}{2} - \mu_1 -\tfrac{\nu_1}{2})} \cdot \\ \frac{\Gamma_F(-q_2 + q_1 + s - \tfrac{|\mu|}{2} - \tfrac{\nu_1}{2}) \Gamma_F(q_1 + \mu_2)  \Gamma_F(q_1 + \mu_3)\Gamma_F(-q_1 + 2s - \mu_1) \Gamma_F(-q_1 + s -  \tfrac{|\mu|}{2} + \mu_1 +\tfrac{\nu_1}{2}) }{\Gamma_F(q_1 + s -  \tfrac{|\mu|}{2} + \mu_1 + \mu_2 + \mu_3 -\tfrac{\nu_1}{2})\Gamma_F(-q_1+ 3s - \tfrac{|\mu|}{2} + \tfrac{\nu_1}{2} + \mu_4)} dq_1 dq_2.
\end{align*}
\normalsize

We now apply Barnes' second Lemma (Lemma \ref{BarnesSecond}) to the integral over $q_2$, where
$$a_1= |\mu|-\mu_4, a_2=- q_1 + \mu_1 +\mu_4,b_1=s - \tfrac{|\mu|}{2} - \mu_1 +\tfrac{\nu_1}{2},b_2=s - \tfrac{|\mu|}{2} - \mu_1 -\tfrac{\nu_1}{2},b_3 = q_1 + s - \tfrac{|\mu|}{2} - \tfrac{\nu_1}{2}.$$
We can apply the Lemma as their sum equals to $ 3s -  \tfrac{|\mu|}{2} - \mu_1 -\tfrac{\nu_1}{2}$. This shows that \eqref{eq:intermediate_step_archComp_afterB1L} equals to

\small 
\begin{align}\label{eq:almost_final_step_arch_spherical}
    \tfrac{\Gamma_F(2s-\mu_1 + \mu_4)}{\Gamma_F(2s - \nu_1)}  \prod_{i=2}^3 \Gamma_F(s +  \tfrac{|\mu|}{2}- \mu_i - \mu_4 - \tfrac{\nu_1}{2}) \prod_{j=1}^2 \Gamma_F(s + \tfrac{|\mu|}{2} -  \mu_j - \mu_{5-j} -\tfrac{\nu_1}{2})\Gamma_F(s +  \tfrac{|\mu|}{2} - \mu_j -\mu_{5-j}+\tfrac{\nu_1}{2}) I_{q_1}(s), 
\end{align}\normalsize
where 
\begin{align*}
    I_{q_1}(s) = \frac{1}{4 \pi i}\int_{q_1} \tfrac{\Gamma_F(q_1+ \mu_2)\Gamma_F(q_1+ \mu_3)\Gamma_F(-q_1+s - \tfrac{|\mu|}{2} + \mu_1 + \tfrac{\nu_1}{2})\Gamma_F(-q_1+s - \tfrac{|\mu|}{2} +  \mu_4 - \tfrac{\nu_1}{2})\Gamma_F(-q_1+s - \tfrac{|\mu|}{2} +  \mu_4 + \tfrac{\nu_1}{2})}{\Gamma_F(-q_1 +3s -\tfrac{|\mu|}{2} + \tfrac{\nu_1}{2} + \mu_4)}  dq_1.
\end{align*}
We can (and do) apply Barnes' second lemma (Lemma \ref{BarnesSecond}) to solve $I_{q_1}(s)$: 
\small 
\begin{align*}
    I_{q_1}(s) = \frac{\prod_{i=2}^3\Gamma_F(s + \tfrac{|\mu|}{2} - \mu_i - \mu_4 + \tfrac{\nu_1}{2}) \prod_{j=2}^3\Gamma_F(s + \tfrac{|\mu|}{2} - \mu_1 - \mu_j - \tfrac{\nu_1}{2})\Gamma_F(s + \tfrac{|\mu|}{2} - \mu_1 - \mu_j + \tfrac{\nu_1}{2}) }{\Gamma_F(2s)\Gamma_F(2s + \nu_1)\Gamma_F(2s + \mu_4 - \mu_1)}.
\end{align*}
 \normalsize

 Hence,
 \begin{align}
     \text{\eqref{eq:almost_final_step_arch_spherical}} &= \frac{\prod_{1 \leq i <j\leq 4}\Gamma_F(s + \tfrac{|\mu|}{2} - \mu_i - \mu_j + \tfrac{\nu_1}{2})\Gamma_F(s + \tfrac{|\mu|}{2} - \mu_i - \mu_j - \tfrac{\nu_1}{2})}{\Gamma_F(2s)\Gamma_F(2s + \nu_1)\Gamma_F(2s - \nu_1)} \nonumber \\ 
     &= \frac{\prod_{1 \leq i <j\leq 4}\Gamma_F(s - \tfrac{|\mu|}{2} + \mu_i + \mu_j + \tfrac{\nu_1}{2})\Gamma_F(s - \tfrac{|\mu|}{2} + \mu_i + \mu_j - \tfrac{\nu_1}{2})}{\Gamma_F(2s)\Gamma_F(2s + \nu_1)\Gamma_F(2s - \nu_1)}. \label{eq:final_before_Lvalue_arch_spher}
 \end{align}
This is the desired ratio of $L$-factors. Indeed, the central character condition \eqref{eq:CC} and Lemma \ref{L_factor_for_spherical_arch} give 

 \begin{align*}
     \text{\eqref{eq:final_before_Lvalue_arch_spher}} &= \frac{\prod_{1 \leq i <j\leq 4}\Gamma_F(s + \mu_i + \mu_j +\nu_0)\Gamma_F(s  + \mu_i + \mu_j +\nu_0 - \nu_1)}{\Gamma_F(2s)\Gamma_F(2s + \nu_1)\Gamma_F(2s - \nu_1)}  \\
     &=  \frac{L(s, \pi_\infty \otimes \sigma_\infty, \wedge^2 \otimes {\rm std}_2)}{L(2s, \sigma_\infty, {\rm Sym}^2 \otimes \omega_{\pi_\infty})}, 
 \end{align*}
concluding the proof of Theorem \ref{Theorem_Arch_computation_split}.

\section{A note on Garrett's pullback formula}

In what follows, we extend Garrett's pullback formula for $\SL_2\times \Sp_4\hookrightarrow \Sp_6$ to the corresponding similitude groups and derive an alternative expression of our zeta integral involving degenerate Siegel Eisenstein series on $\GSp_6$. This will then be used to apply a Siegel-Weil formula of Kudla--Rallis and give a period relation for the central value of $L^S(s,\pi \otimes \sigma, \wedge^2 \otimes  \mathrm{std}_2)$ as well as to study the poles of this $L$-function. 

\subsection{The pullback formula of Garrett}\label{Section:Garrett:Pullback}
Garrett's pullback formula (\textit{cf.}  \cite[Theorem, p. 255]{GarrettIntegralRepEisenstein})  connects non-degenerate Klingen-type Eisenstein series on $\Sp_{2m}$ to the integral over $\Sp_{2n}$ of a degenerate Siegel Eisenstein series on $\Sp_{2(m+n)}$ against a cusp form on $\Sp_{2n}$. In what follows, we discuss how this formula extends to a relation between Eisenstein series for the corresponding similitude groups in the case when $m=2$ and $n=1$. This will allow us to relate the normalized Eisenstein series $E^*_{P_{(1,2)}}(g, s, f_s)$, as defined in \eqref{KES}, to an integral of the normalized degenerate Siegel Eisenstein series on $\GSp_6$. 

Recall that $W_6$ denotes the standard representation of $\GSp_6$ with basis 
  $\{e_1,e_2,e_3,f_3,f_2,f_1\}$; the decomposition of $W_6 = W_2 \oplus W_4$, with $W_2:=\langle e_1,f_1\rangle$ and $W_4:=\langle e_2,e_3, f_3, f_2\rangle$ induces an embedding of $\GL_2\boxtimes \GSp_4 = \{(h_1,h_2)\in \GL_2\times \GSp_4\,:\, {\rm det}(h_1) = \nu(h_2)\}$ into $\GSp_6$: 
\begin{align*}
    \iota:\GL_2\boxtimes \GSp_4 &\hookrightarrow \GSp_6,\\
    \left(\begin{smallmatrix}a&b\\c&d\end{smallmatrix}\right)\times g&\mapsto \left(\begin{smallmatrix}
        a& &b\\ &g& \\c& &d
    \end{smallmatrix}\right).
\end{align*}
We also denote by $\iota$ the corresponding embedding of  $\SL_2\times \Sp_4$ into $\Sp_6$.
For $n \in \{ 1,2,3 \}$, we recall that $P_{(n,0)}$, resp. $P_{(n,0)}^\circ$, denote the standard Siegel parabolic subgroup of $\GSp_{2n}$, resp. $\Sp_{2n}$. The associated flag variety parametrizes maximal isotropic subspaces of $W_{2n}$. We start by recalling the computation of the double coset \[P_{(3,0)}^\circ(F) \backslash \Sp_6(F) / (\SL_2\times \Sp_4)(F).  \]
 Let \begin{equation}\label{openorbitelement}\xi := \left(\begin{smallmatrix}
      1&&&&&\\&1&&&&\\&&1 &&& \\1&&&1&&\\&0&&&1&\\&&1&&&1
\end{smallmatrix}\right).\end{equation}

 \begin{lemma}\label{lemmaondoublecosetsimple} The double quotient 
     $P_{(3,0)}^\circ(F) \backslash \Sp_6(F) / (\SL_2\times \Sp_4)(F)$ consists of two elements $\{I_6,\xi\}$, with corresponding flag and stabilizer as follows: \begin{enumerate}
         \item  $P_{(3,0)}^\circ(F) \cdot I_6$ corresponds to the flag $\langle f_1,f_2,f_3 \rangle$ and has stabilizer given by $P_{(1,0)}^\circ(F) \times P_{(2,0)}^\circ(F)$;
         \item $P_{(3,0)}^\circ(F) \cdot \xi$ corresponds to the flag $\langle e_1 + f_3, e_3 + f_1, f_2 \rangle$ and has stabilizer in  $(\SL_2\times \Sp_4)(F)$ given by \[S_\xi : = \left \{ \left(\begin{smallmatrix}a & b \\ c & d\end{smallmatrix}\right) \times \left(\begin{smallmatrix}
      \alpha & &&\\&d&c&\\&b&a & \\&&&\alpha^{-1}
\end{smallmatrix}\right)\,:\, \alpha \in \GL_1(F), \left(\begin{smallmatrix}a & b \\ c & d\end{smallmatrix}\right) \in \SL_2(F) \right \} \cdot \left(\{ I_2\} \times U_{(1,2)}(F)\right). \]
     \end{enumerate}
 \end{lemma}
\begin{proof}
Recall that any flag in $P_{(3,0)}^\circ(F) \backslash \Sp_6(F)$ is given by $0 \subset S$, with $S=\langle v_1,v_2,v_3\rangle$ a maximal isotropic subspace of $W_6$. The intersection of $S$ with the standard representation $W_2$ of $\SL_2$ has dimension either 0 or 1. We look at the two cases separately. If ${\rm dim}\, S \cap W_2 = 1$, then, as $S$ is isotropic, we have that ${\rm dim}\, S \cap W_4 = 2$ and so $S \cap W_4$ defines a maximal isotropic subspace of $W_4$. We can then use the $\SL_2 \times \Sp_4$ to send $\langle v_1,v_2,v_3\rangle$ to $\langle f_1,f_2,f_3\rangle$. This implies that the set of flags $0 \subset S$ for which ${\rm dim}\, S \cap W_2 = 1$ constitutes the identity orbit $P_{(3,0)}^\circ(F) \cdot I_6 $ with stabilizer given by $P_{(3,0)}^\circ(F) \cap (\SL_2 \times \Sp_4)(F) = P_{(1,0)}^\circ(F) \times P_{(2,0)}^\circ(F)$. 
Now, suppose that ${\rm dim}\, S \cap W_2 = 0$, which implies that ${\rm dim}\, S \cap W_4 = 1$. Up to rescaling the basis, we can assume that $v_3 \in W_4$, while, if we write $v_i^j$ for the component of $v_i$ in $W_j$, we have that $v_1^2,v_2^2$ are linearly independent. Using the $\SL_2(F)$-action, we can assume $v_1^2=e_1,v_2^2=f_1$. Since the symplectic pairing has to map $v_1^4,v_2^4$ to $-1$, using the $\Sp_4(F)$-action, we can then assume that $v_3 = f_2$ and $v_1^4=f_3$, $v_2^4=e_3$. Thus there is a single $(\SL_2 \times \Sp_4)(F)$-orbit of flags for which ${\rm dim}\, S \cap W_2 = 0$, with a generator given by the flag $0 \subset \langle e_1 + f_3, e_3 + f_1, f_2 \rangle$. This flag corresponds to the matrix $\xi$ and it's  immediate to check that its stabilizer is of the prescribed shape.
\end{proof}

For any cuspidal automorphic representation $\sigma \subset \mathcal{A}_{\rm cusp}([\GL_2])$ with central character $\omega_\sigma$, for a Hecke character $\chi$ of $F^\times \backslash \A^\times$, and $s \in \C$, recall that 
$I_{P_{(1,2)}}(s,\sigma, \chi)$ denotes the $\GSp_4(\A)$-representation introduced in \S \ref{Section:Klingen:Eisenstein}. Given $f_s \in I_{P_{(1,2)}}(s,\sigma,\chi)$, we denote by $E_{P_{(1,2)}}(g,s,f_s)$ the unnormalized Klingen Eisenstein series associated to  $f_s$, \textit{i.e.} $$  E_{P_{(1,2)}}(g,s, f_s) ={\sum_{\gamma \in {P_{(1,2)}}(F) \backslash \H(F)}} f_s(\gamma g).$$
Similarly, let  $I_{P_{(3,0)}}\left(s , \chi\omega_\sigma \right)$ be the induced representation of $\GSp_{6}(\A)$ consisting of smooth functions $\phi_{s}:\GSp_{6}(\A)\to\C$ which satisfy
\begin{equation}\label{induction:Siegel}\phi_{s}(n m(h,\mu) g) =\chi\omega_\sigma(\mathrm{det}(h))(\chi\omega_\sigma)^{-2}(\mu)\delta_{P_{(3,0)}}^{\frac{1}{2}(s+1/2)}(m(h,\mu))\phi_{s}(g),\end{equation}
for each unipotent element $n \in U_{(3,0)}(\A)$, each $m(h,\mu) \in M_{(3,0)}(\A) \simeq \GL_3(\A) \times \GL_1(\A)$, and $g \in \GSp_6(\A)$; here $\delta_{P_{(3,0)}}$ is the modulus character of $P_{(3,0)}$: 
\[\delta_{P_{(3,0)}}: m(h,\mu) = \left( \begin{smallmatrix} h & \\ & \mu \cdot I_3' {}^th^{-1} I_3'
\end{smallmatrix} \right)     \mapsto |{\rm det}(h)|^4 |\mu|^{-6}. 
 \]
For any $\phi_s\in I_{P_{(3,0)}}\left(s , \chi\omega_\sigma \right)$, we denote the corresponding Siegel Eisenstein series by
\[E_{P_{(3,0)}}\left(g,s,\phi_s\right) ={\sum_{\gamma \in {P_{(3,0)}}(F) \backslash \GSp_6(F)}} \phi_s(\gamma g).\] 
Let $w= I_2' = \left( \begin{smallmatrix}
    &1 \\ 1 & 
\end{smallmatrix}\right)$ denote the non-trivial Weyl element for $\GL_2$ and, for any cusp form $\Psi$ of $\GL_2$, we let $\Psi^w(g):= \Psi(w g w)$. Given $\phi_s \in I_{P_{(3,0)}}(s ,\chi \omega_\sigma)$ and $\Psi \in \sigma$, set
$$ p_\Psi(\phi_s)(g_2) := \int_{Z_{\SL_2}(\A)\backslash \SL_2(\A)}\phi_s(\xi\iota(h_{\nu(g_2)} g_1,g_2))\Psi^w(h_{\nu(g_2)} g_1)dg_1,$$
where for each $g_2 \in \GSp_4(\A)$ we let $h_{\nu(g_2)} \in \GL_2(\A)$ denote any element such that ${\rm det}(h_{\nu(g_2)}) = {\nu(g_2)}$. For ${\rm Re}(s)$ big enough, the integral converges absolutely. Moreover, it is a direct calculation to check that $p_\Psi(\phi_s)$ is an element of $I_{P_{(1,2)}}(s,\sigma,\chi)$ and that its definition does not depend on the choice of $h_{\nu(g_2)} \in \GL_2(\A)$. Thus, for ${\rm Re}(s)>>0$, it defines a map
\begin{align}
    p_\Psi:I_{P_{(3,0)}}(s ,\chi \omega_\sigma) & \to I_{P_{(1,2)}}(s,\sigma,\chi),\label{PullbackMapRepresentation}\\ 
    \phi_s&\mapsto p_\Psi(\phi_s)(g_2).\nonumber
\end{align}

\begin{theorem}\label{PullbackFormulaNew} For any Hecke character $\chi$ of $F^\times \backslash \A^\times$, any cuspidal automorphic representation $\sigma$ of $\GL_2$ with cusp form $\Psi$ in $\sigma$, and a holomorphic section $\phi_s = \otimes_v \phi_{v,s} \in I_{P_{(3,0)}}(s,\chi \omega_\sigma)$, we have the following. \begin{enumerate}
    \item  There exists a section $f^{\Psi}_{\phi_s}\in I_{P_{(1,2)}}(s,\sigma,\chi)$, satisfying the equality $f^{\Psi}_{\phi_s} = p_{\Psi}(\phi_s)$ for ${\rm Re}(s)$ large enough, such that 
    \begin{align}\label{eq:basic_pullback}
        \int_{[Z_{\SL_2}\backslash \SL_2]} E_{P_{(3,0)}}(\iota( h_{\nu(g_2)} g_1  , g_2), s,\phi_s) \Psi^w( h_{\nu(g_2)} g_1  ) d g_1 = E_{P_{(1,2)}}(g_2,s,f^{\Psi}_{\phi_s}).
    \end{align}  
    \item  Let $S$ be a sufficiently large set of places containing the archimedean places of $F$ as well as the ramified places for $\Psi$ and $\phi_s$.  Suppose that $\phi_s$ is normalized so that $\phi_{v,s}(1)=1$ when $v \not \in S$. For ${\rm Re}(s)$ large enough, we have 
    \[ p_\Psi(\phi_s)(1) = \frac{L^S(2s,\sigma ,\mathrm{Sym}^2\otimes \chi)}{L^S(2s+1,\chi\omega_{\sigma})L^S(4s,\chi^2\omega_{\sigma}^2)} \int_{Z_{\SL_2}(\A_S)\backslash \SL_2(\A_S)}\phi_{S,s}(\xi\iota( g_1,1))\Psi^w(g_1)dg_1, \]
where $\A_S = \prod_{v \in S} F_v$ and $\phi_{S,s} = \otimes_{v \in S} \phi_{v,s}$. 
\end{enumerate}
\end{theorem}

\begin{proof} Let us sketch the proof of the statement, which closely follows \cite{GarrettIntegralRepEisenstein} (see also \cite{PitaleSchmidt}). Fix $g_2 \in \GSp_4(\A)$ and let $h_{\nu(g_2)} \in \GL_2(\A)$ be an element such that  ${\rm det}(h_{\nu(g_2)}) = {\nu(g_2)}$. For ${\rm Re}(s)$ large enough, after unfolding the Siegel Eisenstein series, Lemma \ref{lemmaondoublecosetsimple} let us write \begin{align*}
    \int_{[Z_{\SL_2}\backslash \SL_2]} E_{P_{(3,0)}}(\iota( h_{\nu(g_2)} &g_1  , g_2), s,\phi_s) \Psi^w( h_{\nu(g_2)} g_1  ) d g_1 \\ 
    &= \int_{[Z_{\SL_2}\backslash \SL_2]} {\sum_{\gamma \in (P_{(1,0)}^\circ(F) \times P_{(2,0)}^\circ(F))\backslash (\SL_2\times \Sp_4)(F)}} \phi_s(\gamma \iota( h_{\nu(g_2)} g_1  , g_2))\Psi^w( h_{\nu(g_2)} g_1  ) d g_1 \\ &+ \int_{[Z_{\SL_2}\backslash \SL_2]} {\sum_{\gamma \in S_\xi \backslash (\SL_2\times  \Sp_4)(F)}} \phi_s(\xi \gamma \iota( h_{\nu(g_2)} g_1  , g_2))\Psi^w( h_{\nu(g_2)} g_1  ) d g_1.
\end{align*}
We analyze each integral separately. The first one equals to 
\[ \int_{Z_{\SL_2}(\A)P_{(1,0)}^\circ(F) \backslash \SL_2(\A)} {\sum_{\gamma \in P_{(2,0)}^\circ(F)\backslash \Sp_4(F)}} \phi_s(\iota( h_{\nu(g_2)} g_1  , \gamma g_2))\Psi^w( h_{\nu(g_2)} g_1  ) d g_1.  \]
By collapsing the integral over the unipotent part of $P_{(1,0)}^{\circ}$, one sees that this integral vanishes because of the cuspidality of $\Psi$. We now look at the second integral. Firstly, notice that any element of $S_\xi$ can be written as  \begin{equation}\label{Form:Stab} \left(g, \left(\begin{smallmatrix} a &  &  \\ & g^w & \\  & & a^{-1} \end{smallmatrix} \right)\right) \cdot \left( I_2 , n \right), \end{equation}
with $g \in \SL_2(F)$, $g^w = w g w $, $a \in F^\times$, and $n \in U_{(1,2)}(F)$. Hence, the quotient $S_\xi\backslash (\SL_2 \times \Sp_4)(F)$ is isomorphic to $\SL_2(F) \times \left (P_{(1,2)}^\circ (F) \backslash \Sp_4(F) \right)$. This let us write the integral as
\begin{align*}
     \int_{[Z_{\SL_2}\backslash \SL_2]} {\sum_{\gamma \in S_\xi \backslash (\SL_2\times  \Sp_4)(F)}} &\phi_s(\xi \gamma \iota( h_{\nu(g_2)} g_1  , g_2))\Psi^w( h_{\nu(g_2)} g_1  ) d g_1 \\ &= \int_{Z_{\SL_2}(\A)\backslash \SL_2(\A)} {\sum_{\gamma \in P_{(1,2)}^\circ (F) \backslash \Sp_4(F)}} \phi_s(\xi  \iota( h_{\nu(g_2)} g_1  , \gamma g_2))\Psi^w( h_{\nu(g_2)} g_1  ) d g_1 \\ &= {\sum_{\gamma \in P_{(1,2)}^\circ (F) \backslash \Sp_4(F)}} \int_{Z_{\SL_2}(\A)\backslash \SL_2(\A)}  \phi_s(\xi  \iota( h_{\nu(g_2)} g_1  , \gamma g_2))\Psi^w( h_{\nu(g_2)} g_1  ) d g_1 \\ &= {\sum_{\gamma \in P_{(1,2)}^\circ (F) \backslash \Sp_4(F)}} p_\Psi(\phi_s)(\gamma g_2) \\ 
     &= E_{P_{(1,2)}}(g_2,s,p_\Psi(\phi_s)).
\end{align*}
The left hand side of \eqref{eq:basic_pullback} converges absolutely for $s \in \C$ away from the poles of $E_{P_{(3,0)}}(\cdot, s,\phi_s)$, thus we conclude the proof of (1) by analytic continuation. For (2), let $S$ be a finite set of places of $F$ containing the archimedean places of $F$ and the ramified places for $\sigma$ and $\chi$. Consider the local integral at $v \not \in S$ that defines $p_\Psi(\phi_s)(1)$:
\begin{equation}\label{Local:Zeta:Pullback:Aux} z(s, \phi_{v,s}, v_0):= \int_{Z_{\SL_2}(F_v)\backslash \SL_2(F_v)}  \phi_{v,s}(\xi  \iota( g_1  ,1)) \sigma_v(g_1^w) \cdot v_0\, d g_1, \end{equation}
with $v_0$ a spherical vector of $\sigma_v$. The integral $z(s, \phi_{v,s}, v_0)$ can be calculated following the proof of \cite[Proposition 4.1]{PitaleSchmidt}. Notice that if $k_1,k_2 \in K := \SL_2(\mathcal{O}_v)$, then, using \eqref{Form:Stab}, one can show that \[\phi_{v,s}( \xi \iota (k_1 g k_2, 1)) =  \phi_{v,s}( \xi \iota (g, 1)  \iota ( k_2, k_1^{\sharp,-1})) =   \phi_{v,s}( \xi \iota (g, 1)), \]
where $k_1^\sharp = \left(\begin{smallmatrix} 1 &  &  \\ & k_1^w & \\  & & 1 \end{smallmatrix} \right)$. Using this and the Cartan decomposition \[ \SL_2(F_v) = \sqcup_{n \geq 0} K{\rm diag}(\varpi_v^{n},\varpi_v^{-n})K, \]
we have 
\begin{align*}
    z(s, \phi_{v,s}, v_0) = \sum_{n \geq 0}\phi_{v,s}(\xi  \iota( {\rm diag}(\varpi_v^{n},\varpi_v^{-n})  ,1)) \int_{K} \int_{K} \sigma_v((k_1{\rm diag}(\varpi_v^{n},\varpi_v^{-n})k_2)^w) \cdot v_0\, d k_1\, d k_2.
\end{align*}
Using another matrix manipulation and again the invariance of $\phi_{v,s}$ by right multiplication of $\Sp_6(\mathcal{O}_v)$, one can show that \[\phi_{v,s}(\xi  \iota( {\rm diag}(\varpi_v^{n},\varpi_v^{-n})  ,1)) = (\chi_v \omega_{\sigma_v})(\varpi_v^{n})|\varpi_v|^{n(2s+1)}\phi_{v,s}(1) =(\chi_v \omega_{\sigma_v})(\varpi_v^{n})|\varpi_v|^{n(2s+1)}.\]
We can then use B\"ocherer's identity (\textit{cf}. \cite[(23)]{PitaleSchmidt} with $Y= (\chi_v \omega_{\sigma_v})(\varpi_v)|\varpi_v|^{2s+1}$) to deduce that \begin{equation}\label{Bocherer:Identity}z(s, \phi_{v,s}, v_0) = \frac{L(2s,\sigma_v,{\rm Sym}^2 \otimes \chi_v)}{L(2s+1, \chi_v\omega_{\sigma_v})L(4s, \chi_v^2\omega_{\sigma_v}^2) } \cdot v_0. \end{equation}
The proof of (2) then follows by induction on the set of places not in $S$ as in the case of zeta integrals
unfolding to non-unique models.
\end{proof}

Let $S$ be a finite set of places for $F$ including the archimedean places such that outside of $S$, all data is unramified; define the normalized Siegel Eisenstein series by \[E_{P_{(3,0)}}^*\left(g,s,\phi_s\right):= L^S(2s+1, \chi\omega_{\sigma})L^S(4s, \chi^2\omega_{\sigma}^2) E_{P_{(3,0)}}\left(g,s,\phi_s\right). \]
This normalization, together with Theorem \ref{PullbackFormulaNew}, let us give an alternative expression of the zeta integral studied earlier in the manuscript, which is better suited for the study of its central value and poles.
\begin{corollary}\label{CorPullbackZeta}
Keeping the notation of Theorem \ref{PullbackFormulaNew} and writing $\Psi = \otimes_v \Psi_v$, we have the identity  
\begin{align*} \int_{[Z_{\SL_2}\backslash \SL_2]} E_{P_{(3,0)}}^*(\iota( h_{\nu(g_2)} g_1  , g_2), s,\phi_s) \Psi^w( h_{\nu(g_2)} g_1  ) d g_1 = E_{P_{(1,2)}}^*(g_2,s,f^{\Psi}_{\phi_s}),\end{align*}
where  $f^\Psi_{\phi_s}$ is the section obtained in Theorem \ref{PullbackFormulaNew} (1).
In particular, if $\pi$ is a cuspidal automorphic representation of $\GL_4(\A)$ (resp.  $\GU_{2,2}(\A)$), $\varphi$ a cusp form in the space of $\pi$, and $\chi $ equals $ \omega_{\pi}$ (resp. $ (\omega_{\pi})_{|\A^\times}$), we have \begin{align}\label{eq_Cor_Pullback}
   Z(s,\varphi, f^\Psi_{\phi_s}) = \int_{Z_{\GSp_6}(\A)\backslash [\GL_2\boxtimes \GSp_4]} E_{P_{(3,0)}}^*(\iota( g_1  , g_2), s,\phi_s) \Psi^w( g_1  ) \varphi(g_2)\, dg_1\,dg_2.
\end{align}
\end{corollary}
\section{The degree $12$ central $L$-value and the generalized Shalika period}\label{Section:Period:Central:Value}
We let $\pi$ be a cuspidal automorphic representation of either $\GL_4(\A)$ or $\GU_{2,2}(\A)$ and let $\sigma$ be a cuspidal automorphic representation of $\GL_2(\A)$. We suppose that they are unramified outside a finite set $S$ of places of $F$. In what follows, we give a period condition for the vanishing of the central value (at $s=1/2$) of the partial $L$-function $L^S(s,\pi \otimes \sigma, \wedge^2 \otimes  \mathrm{std}_2)$. Throughout this section, we assume that as a Hecke character on $F^\times \backslash \A^\times$ \begin{equation}\label{TrivialCentralCharacters}
    \omega_\pi \omega_\sigma = 1.
\end{equation}
In Corollary \ref{CorPullbackZeta}, we have shown that 
\[Z(s,\varphi, f^\Psi_{\phi_s}) = \int_{Z_{\GSp_6}(\A)\backslash [\GL_2\boxtimes \GSp_4]} E_{P_{(3,0)}}^*(\iota( g_1  , g_2), s,\phi_s) \Psi^w( g_1  ) \varphi(g_2)\, dg_1\,dg_2,\]
where $\varphi$ is in the space of $\pi$, $\Psi$ in the space of $\sigma$ and, under the assumption \eqref{TrivialCentralCharacters}, $\phi_s$ is a function in $I_{P_{(3,0)}}(s, \mathbf{1})$. We now recall a Siegel--Weil formula which relates the value at $s=1/2$ of $E_{P_{(3,0)}}^*(\iota( g_1  , g_2), s,\phi_s)$
to a theta integral for dual reductive pairs of the form $(\mathbf{O}(V),\Sp_6)$, where $V$ is a quaternionic quadratic space of dimension $4$. We note that, according to our definition, $E_{P_{(3,0)}}^*(g, s,\phi_s)$ corresponds to the normalized Siegel Eisenstein series of Kudla--Rallis in \cite{KudlaRallis} up to a twist on $s$ by $1/2$; in particular, the value at $s=1/2$ of (the restriction to $\Sp_6$ of) $E_{P_{(3,0)}}^*(g, s,\phi_s)$ corresponds to the value at $s=0$ of the normalized Siegel Eisenstein series defined in \cite[p. 4]{KudlaRallis}. The extension of the Siegel--Weil formula of \cite{KudlaRallis} to the case of similitude groups is considered in  \cite[Theorem 4.1]{KudlaHarrisJacquet} and we now recall it.

\subsection{Siegel-Weil formula for similitudes}\label{subsec:SW:Similitudes}

Let $V_{1,v}$ denote the totally split quadratic space of dimension $4$ over $F_v$ and let $V_{2,v}$ be the anisotropic quadratic space of dimension $4$ over $F_v$, defined by the definite quaternion algebra over $F_v$ endowed with its norm as in \S \ref{Exceptional:Isom:Section}. For $i = 1,2$, the local theta correspondence of the identity for the dual reductive pairs $(\mathbf{O}(V_{i,v}),\Sp_6(F_v))$ is realized by the maps
\begin{align*}\lambda_{V_{i,v}}:\mathcal{S}(V_{i,v}^3)&\to \tilde{I}_{P_{(3,0)}}\left(1/2,\mathbf{1}\right)_v,\, \, f_v \mapsto \omega(g)f_v(0),\end{align*}
where $\tilde{I}_{P_{(3,0)}}\left(1/2,\mathbf{1}\right)_v$ is the representation consisting of the sections of $I_{P_{(3,0)}}\left(1/2,\mathbf{1}\right)_v$ restricted to $\Sp_6(F_v)$  and $\omega$ denotes the action via the Weil representation of $(1, g) \in  \mathbf{O}(V_{i,v})\times\Sp_6(F_v)$ on $\mathcal{S}(V_{i,v}^3)$ . We denote by $\tilde{\Pi}(V_{i,v})$ the image of $\lambda_{V_{i,v}}$. \cite[Theorem 2.2 (iii)]{KudlaRallis} shows that, for any non-archimedean place $v$, we have \[\tilde{I}_{P_{(3,0)}}\left(1/2,\mathbf{1}\right)_v \simeq \tilde{\Pi}(V_{1,v})\oplus\tilde{\Pi}(V_{2,v}).\]
This generalizes globally as follows. Given a family $V=\{V_v\}_v$ of quadratic spaces, with $V_v$ either equal to $V_{1,v}$ or $V_{2,v}$, let $V(\A) := \otimes_v V_v$ and consider the map $\lambda_{V} := \otimes_v\lambda_{V_v}:\mathcal{S}(V(\A)^3)\to \tilde{I}_{P_{(3,0)}}\left(1/2,\mathbf{1}\right)$, whose image is denoted by $\tilde{\Pi}(V)$. One has a decomposition of $\tilde{I}_{P_{(3,0)}}\left(1/2,\mathbf{1}\right)$ into irreducible representations of $\Sp_6(\A)$ as follows: 
\begin{align}\label{thetadecSp6}
\tilde{I}_{P_{(3,0)}}\left(1/2,\mathbf{1}\right) \simeq \left(\bigoplus_{V\text{ coherent}} \tilde{\Pi}(V)\right)\oplus \left(\bigoplus_{\mathcal{V}\text{ incoherent}} \tilde{\Pi}(\mathcal{V})\right).
\end{align}
Here, the first sum runs over the isomorphism classes of global quadratic spaces $V$ of dimension $4$ over $F$ such that $V \otimes_F F_v \in \{ V_{1,v},V_{2,v}\}$ for each place $v$ of $F$. We refer to such quadratic spaces as \textit{quaternionic}. The second sum is over the isomorphism classes of incoherent quadratic spaces $\mathcal{V} = \otimes_v \mathcal{V}_v$, \textit{i.e.} quadratic spaces $\otimes_v \mathcal{V}_v$, with $\mathcal{V}_v \in \{ V_{1,v},V_{2,v}\}$, which do not arise as the localizations of a global quadratic space $V$ over $F$. 
For instance, if $\phi \in \tilde{I}_{P_{(3,0)}}\left(1/2,\mathbf{1}\right)$ is such that $\phi_v \in \tilde{\Pi}(V_{1,v})$ for all places $v$, then $\phi$ is in the image of $\lambda_{\mathbb{H}^2}$, where $\mathbb{H}^2 = \otimes_v V_{1,v}$ denotes the totally split quadratic space of dimension $4$ defined over $F$.

\begin{remark}
According to \cite[p. 8 (ii)]{KudlaRallis}, sections associated to incoherent collections produce trivial zeroes at the value $s=1/2$ of the corresponding Eisenstein series. 
\end{remark}

As explained in \cite[p. 9]{KudlaHarrisJacquet}, the restriction map $I_{P_{(3,0)}}\left(1/2,\mathbf{1}\right) \to \tilde{I}_{P_{(3,0)}}\left(1/2,\mathbf{1}\right)$ is an isomorphism and has an inverse of the form $\phi \mapsto \phi^{\sim}$, where \[\phi^{\sim}(g) = |\nu(g)|^{-3}\phi(\tilde{g}),\] with $\tilde{g} = \left(\begin{smallmatrix}I_3 & \\ &\nu(g)^{-1}I_3\end{smallmatrix}\right)g\in \Sp_6(\A)$. Thus, the decomposition \eqref{thetadecSp6} gives
\[I_{P_{(3,0)}}\left(1/2,\mathbf{1}\right) \simeq \left(\bigoplus_{V\text{ coherent}} \Pi(V)\right)\oplus \left(\bigoplus_{\mathcal{V}\text{ incoherent}} \Pi(\mathcal{V})\right),\]
where $\Pi(V)$, resp. $\Pi(\mathcal{V})$, denote the inverse image of $\tilde{\Pi}(V)$, resp. $\tilde{\Pi}(\mathcal{V})$, under the restriction map.  

Given a quadratic space $V$ over $F$ of dimension $4$ defined as above, let 
\begin{align}\label{GSp6+def}\GSp_6^+(\A) &:= \{g\in \GSp_6(\A),\;:\;\nu(g)\in \nu' \left(\GO_V(\A)\right)\} \\&= \{g\in \GSp_6(\A),\;s.t.\;\nu(g_v)>0 \text{ when }v \text{ is real and }V_v\text{ is anisotropic}\}.\nonumber\end{align}

If $V$ is anisotropic over $F$, then, for $g \in \GSp_6^+(\A)$ and $f \in \mathcal{S}(V(\A)^3)$, we can consider the theta integral \[ I(g, f) = \int_{[\mathbf{O}_V]} \theta(h_1 h, g ; f)\, d h_1,\]
with $h \in \mathbf{GO}(V)(\A)$ such that $\nu'(h) = \nu(g)$ and $\theta$ is the theta kernel associated to the dual reductive pair $(\mathbf{O}_V,\Sp_6)$ (\emph{cf.} \cite[(3.7)]{KudlaRallis}). If $V$ is totally split, then the theta integral, which we still denote by $I(g, f)$, is defined by regularization using the operators of \cite[\S 5.1]{KudlaRallis} or \cite[\S 1]{IchinoSW}. The definition of $I(g, f)$ does not depend on the choice of $h \in \mathbf{GO}_V(\A)$. Moreover, 
$I(g, f)$ is invariant under $\GSp_6^+(\A) \cap \GSp_6(F)$ and has trivial central character. Finally, notice that, as $\GSp_6(\A) =\GSp_6(F) \GSp_6^+(\A)$, $I(g, f)$ extends uniquely to a left $\GSp_6(F)$-invariant function on $\GSp_6(\A)$. 

The description of $\GSp_6^+(\A)$ above, together with
\begin{equation*} \GSp_6(\A) = \GSp_6(F)Z_{\GSp_6}(\A)\Sp_6(\A)\GSp_6(\A_f),\end{equation*}
serve as crucial elements in the proof of \cite{KudlaHarrisJacquet} of the extension of the Siegel--Weil formula to similitude groups.
\begin{theorem}[{\cite[Theorem 4.2]{KudlaHarrisJacquet}}]\label{Siegel-Weilnew}\leavevmode 
\begin{enumerate}
    \item  Let $\phi_s \in I_{P_{(3,0)}}\left(s,\mathbf{1}\right)$ be a section such that $ \phi = \phi_{1/2} \in \Pi(V)$ for a quaternionic quadratic space $V$ over $F$, and let $f \in \mathcal{S}(V(\A)^3)$ be such that $\phi  = \lambda_V(f)^{\sim}$.  We have 
     \[E_{P_{(3,0)}}(g, 1/2,\phi) = 2 \cdot I(g, f).\]
     \item If $\phi_s \in I_{P_{(3,0)}}\left(s,\mathbf{1}\right)$ is a section for which $\phi_{1/2} \in \Pi(\mathcal{V})$, for any incoherent collection $\mathcal{V}$, then 
     \[E_{P_{(3,0)}}(g, 1/2,\phi_{1/2}) = 0.\]
\end{enumerate}
\end{theorem}

The proof of the following Lemma, which will be used in \S \ref{sec:SVvsLcentral}, is an adaptation of the one of \cite[Theorem 4.1(ii)]{Kudla:Integral:Borcherds}. 

\begin{lemma}\label{Reduction:To:SO}
  We have that
     $I(g,f) = \frac{1}{2}\int_{[\mathbf{SO}_V]}\theta(h_1h,g;f)dh_1$.
\end{lemma}
\begin{proof}
  We first note that the group $$\mu_2(\A) := \mathbf{O}_V(\A)/\mathbf{SO}_V(\A)$$ acts trivially on the space of $\mathbf{SO}_V(\A)$-invariant linear functionals on $\mathcal{S}(V(\A)^3)$. This can be verified as follows. For any place $v$, $\mu_2(F_v) = \mathbf{O}_V(F_v)/\mathbf{SO}_V(F_v)$ acts on the space 
  $$\mathrm{Hom}_{\SO_V(F_v)}(\mathcal{S}(V(F_v)^3),\C)$$
  either trivially or via the sign character ${\rm sgn}_v$ of $\mathbf{O}_V(F_v)$. If it acts via ${\rm sgn}_v$, then ${\rm sgn}_v$ would occur in the local theta correspondence for the dual reductive pair $(\mathbf{O}_V(F_v),\Sp_6(F_v))$. However, the conservation relation stated in \cite[p. 11]{Conservation:Relation:Theta} shows that the sign character  does not occur in the theta correspondence for $(\mathbf{O}_V(F_v),\Sp_6(F_v))$. This implies that $\mu_2(F_v)$ has to act trivially on $\mathrm{Hom}_{\SO_V(F_v)}(\mathcal{S}(V(F_v)^3),\C)$, as claimed. 
   
   Since $$\left(f\mapsto \int_{[\SO_V]}\theta(h_1h,g;f) dh\right)\in \mathrm{Hom}_{\SO_V(\A)}(\mathcal{S}(V(\A)^3),\C),$$ the group $\mu_2(\A)$ acts trivially on it and we have
    \begin{align*}
    I(g,f) &= \int_{[\mu_2]}\int_{[\SO_V]}\theta(h_1 c h,g;f)dh dc\\
    &= \int_{[\mu_2]} d c \int_{[\SO_V]}\theta(h_1 h,g;f)dh\\
    &= \frac{1}{2}\int_{[\SO_V]}\theta(h_1 h,g;f)dh.
    \end{align*}
\end{proof}

We now discuss the Siegel--Weil formula in the special case where $V$ equals the four dimensional split quadratic space $\mathbb{H}^2= \otimes_v V_{1,v}$. Suppose that $\phi_{1/2} \in \Pi(\mathbb{H}^2)$, then the Siegel--Weil formula can be understood as an identity between the value $E_{P_{(3,0)}}^*(g, 1/2,\phi_{1/2})$ and a term in the Laurent series expansion of an Eisenstein series of $\GSp_6$, associated with the parabolic subgroup $P_{(2, 2)}$. For a detailed discussion on this, see \cite[\S 5.5 \& Theorem 6.1]{KudlaRallis}. We now outline how the formula of Kudla--Rallis extends to the similitude groups.  Firstly, we recall the following facts on $P_{(2, 2)}$.

\begin{remark}
 $P_{(2, 2)} = M_{(2, 2)}\, U_{(2, 2)}$ is the parabolic subgroup of $\GSp_6$ with Levi isomorphic to $\GL_2\times\GL_2$ and unipotent subgroup 
\[U_{(2, 2)} = \left\{ \left(\begin{smallmatrix} I_2 &  \star & \star \\ & I_2 & \star \\ & & I_2 \end{smallmatrix} \right)\in \GSp_6 \right\} \supseteq [U_{(2, 2)} , U_{(2, 2)} ] = \left\{ \left(\begin{smallmatrix} I_2 &  & \star \\ & I_2 & \\ & & I_2 \end{smallmatrix} \right)\right\}. \]
Note the following: \begin{itemize}
    \item It has modulus character $\delta_{P_{(2, 2)} } : P_{(2, 2)} (\A) \to \C$ given by \[ M_{(2, 2)}(\A) \ni  m(g_1,g_2) := \left(\begin{smallmatrix} g_1 &   &  \\ & g_2 &  \\ & & {\rm det}(g_2) I_2'{}^t g_1^{-1} I_2' \end{smallmatrix} \right) \mapsto \left| \frac{{\rm det}(g_1)}{{\rm det}(g_2)} \right|^5. \]
    \item The flag variety $P_{(2, 2)}  \backslash \GSp_6$ parametrizes isotropic planes in the standard representation $W_6$ of $\GSp_6$, with $P_{(2,2)}$ being the stabilizer of the isotropic plane $\langle f_2, f_1\rangle \subset W_6$. 
\end{itemize}

\end{remark}

Let $I_{P_{(2,2)}}(s, \mathbf{1})$ be the induced representation of $\GSp_{6}$ consisting of smooth functions $F_{s}:\GSp_{6}(\A)\to\C$ that satisfy
\begin{equation}\label{Induction:22}F_{s}(n m(g_1,g_2) -) =\delta_{P_{(2,2)}}^{\frac{s}{5} + \frac{1}{2} }(m(g_1,g_2))F_{s}(-),\end{equation}
for all $n \in U_{(2,2)}(\A)$ and $m(g_1,g_2) \in M_{{(2,2)}}(\A)$.  
 For any $F_s\in I_{P_{(2,2)}}(s, \mathbf{1})$, we define the Eisenstein series \[E_{P_{(2,2)}}(g,s,F_s) := \sum_{\gamma\in P_{(2,2)}(F)\setminus \GSp_6(F)}F_s(\gamma g).\] 

Given $\varphi\in \mathcal{S}(\mathbb{H}^2(\A)^3)$, we define the Fourier transform $\hat{\varphi} \in \mathcal{S}(W_6(\A)^2)$ by (see equation \cite[(5.3.3)]{KudlaRallis}) 
\[\hat{\varphi}(u,v) := \int_{M_{2\times 3}(\A)}\varphi\left(\begin{smallmatrix}x\\u\end{smallmatrix}\right)\psi(\mathrm{Tr}({}^tv x))dx,\]
where we have written any element in $W_6(\A)^2 \simeq M_{2\times 6}(\A)$ as $(u,v)$, with $u,v \in M_{2\times 3}(\A)$. The map sending $\varphi\mapsto \hat{\varphi}$ gives an intertwining map $\mathcal{S}(\mathbb{H}^2(\A)^3)\to \mathcal{S}(W_6(\A)^2)$.
\begin{lemma}\label{lemma_constructing_sectionsforR}
    Let $\phi \in I_{P_{(3,0)}}\left(1/2,\mathbf{1}\right)$ be a section such that $ \phi \in \Pi(\mathbb{H}^2)$ and let $\varphi \in \mathcal{S}(\mathbb{H}^2(\A)^3)$ be such that $\phi  = \lambda_V(\varphi)^{\sim}$, \textit{i.e.}  $\phi(g) = \omega(h,g)\varphi(0)$ for some $h\in \GO_{\mathbb{H}^2}(\A)$ with $\nu(h) = \nu(g)$. Set 
     \[F(\phi,g,s):= | \nu(g)|^{-s+1/2}\int_{\GL_2(\A)}\omega(h,g)\hat{\varphi}
    ({}^{t}aw_0)|\det(a)|^{s+5/2} d a,\]
    where $w_0 = (0,0,I_2)\in (F^{6})^2$. For ${\rm Re}(s)>-1/2$,  $F(\phi,g,s)$ converges absolutely and belongs to $I_{P_{(2,2)}}(s, \mathbf{1})$. 
\end{lemma}
 
\begin{proof}
    According to \cite[Lemma 5.5.2 (i)]{KudlaRallis}, $F(\phi,g,s)$ converges absolutely for ${\rm Re}(s)>-1/2$.  Moreover, by \cite[(5.5.15)]{KudlaRallis}, we have that 
    \[F(\phi,g,s)|_{\Sp_6(\A)}\in \mathrm{Ind}_{P_{(2,2)}^\circ(\A)}^{\Sp_6(\A)}(|\cdot|^s\circ \mathrm{det}),\]
    where $|\cdot|^s\circ \mathrm{det}$ denotes the character of the Levi $\GL_2 \times \SL_2$ of $P_{(2,2)}^\circ$ which sends $$  \left(\begin{smallmatrix} g_1 &   &  \\ & g_2 &  \\ & &  I_2'{}^t g_1^{-1} I_2' \end{smallmatrix} \right) \mapsto | {\rm det}(g_1)|^s.$$
    Then, the Proposition follows from showing that $$F(\phi,m(I_2,g),s) = | {\rm det}(g)|^{-s -5/2 } F(\phi,I_6,s), \, \forall g \in \GL_2(\A).$$
 Let $g \in \GL_2(\A)$ and write $\eta : = {\rm det}(g)$ and $h := \left(\begin{smallmatrix} 1 & & & \\ & 1 & & \\ & & \eta & \\ & & & \eta \end{smallmatrix}\right) \in \GO_{\mathbb{H}^2}(\A)$. Then, using the explicit action of $\omega(h,m(I_2,g))$ on $\varphi$ (using \eqref{eq_normalization_Weil_similitudes}), we have

\begin{align*}
        F(\phi,m(I_2,g),s) &= | \eta |^{-s+1/2}  \int_{\GL_2(\A)}\omega(h,m(I_2,g))\hat{\varphi}
    (^{t}a w_0)|\det(a)|^{s+5/2}da \\ &= | \eta |^{-s+1/2} \int_{\GL_2(\A)}\int_{M_{2\times 3}(\A)}\omega(h,m(I_2,g))\varphi\left(\begin{smallmatrix}x_1&x_2&x_3\\ x_4&x_5&x_6\\ 0& 0& 0\\ 0& 0& 0\end{smallmatrix}\right)\psi\left(\mathrm{Tr}\left(\left(\begin{smallmatrix}0&0\\a_1&a_2\\ a_3& a_4\end{smallmatrix}\right)x\right)\right)|\det(a)|^{s+5/2}dx_i \,da \\ &= | \eta |^{-s-5/2} \int_{\GL_2(\A)}\int_{M_{2\times 3}(\A)}\omega(\widetilde{m}(I_2,g))\varphi\left(\begin{smallmatrix}x_1&x_2&x_3\\ x_4&x_5&x_6\\ 0& 0& 0\\ 0& 0 & 0\end{smallmatrix}\right)\psi\left(\mathrm{Tr}\left(\left(\begin{smallmatrix}0&0\\a_1&a_2\\ a_3& a_4\end{smallmatrix}\right)x\right)\right)|\det(a)|^{s+5/2}dx_i \,da \\ 
    &= | \eta |^{-s-5/2}\int_{\GL_2(\A)} \omega(\widetilde{m}(I_2,g))\hat{\varphi}
        \left({}^{t}aw_0\right)|\det(a)|^{s+5/2}da,
    \end{align*} 
    where we have denoted any $a = \left(\begin{smallmatrix}a_1&a_2\\a_3&a_4\end{smallmatrix}\right)\in \GL_2(\A)$ and $\widetilde{m}(I_2,g) := \left(\begin{smallmatrix}I_3& \\ &\eta^{-1}I_3\end{smallmatrix}\right) m(I_2,g)\in \Sp_6(\A)$. Finally, by \cite[(5.5.15)]{KudlaRallis}, we have that \[\int_{\GL_2(\A)} \omega(\widetilde{m}(I_2,g))\hat{\varphi}
        \left( {}^{t}aw_0\right)|\det(a)|^{s+5/2}da = \int_{\GL_2(\A)} \hat{\varphi}
        \left({}^{t}aw_0\right)|\det(a)|^{s+5/2}da,\]
        hence, we get \begin{align*}
      F(\phi,m(I_2,g),s) &= |\eta|^{- s - 5/2}F(\phi,I_6,s),
    \end{align*}
    as desired.
\end{proof}

Lemma \ref{lemma_constructing_sectionsforR} allows us to define the following:
\begin{align*}
    F_s:\Pi(\mathbb{H}^2) & \to I_{P_{(2,2)}}(s, \mathbf{1}),\\ \phi(g) = \omega(h,g)\varphi(0)&\to F(\phi,g,s).
\end{align*}

Thanks to the results of \cite[\S 5.5]{KudlaRallis} (see in particular their equation (5.5.23) and Lemma 5.5.6), if $\phi \in \Pi(\mathbb{H}^2)$, the Eisenstein series $E_{P_{(2,2)}}(g, s, F_s(\phi))$ might attain a pole at $s=1/2$. By extending \cite[Theorem 6.12]{KudlaRallis} to similitude groups, the Siegel--Weil formula of Theorem \ref{Siegel-Weilnew}(i) can be reinterpreted as follows.

\begin{theorem}\label{SWcentralSplit}
    Let $\phi_s \in I_{P_{(3,0)}}\left(s,\mathbf{1}\right)$ be a section such that $ \phi = \phi_{1/2} \in \Pi(\mathbb{H}^2)$. There exists a non-zero constant $c$ such that  
    \[E_{P_{(3,0)}}\left(g,1/2,\phi\right) =c \cdot   \mathrm{Res}_{s = 1/2}E_{P_{(2,2)}}\left(g,s,F_s(\phi)\right).\]
\end{theorem}

\begin{proof}
Let us factor any element $g\in\GSp_6(\A) = \GSp_6(F)Z_{\GSp_6}(\A)\Sp_6(\A)\GSp_6(\A_f)$ as $g = g_Fg_Z\tilde{g}g_f$. By equation (5.5.12) and Theorem 6.12 of \cite{KudlaRallis}, there exists a non-zero constant $c$ such that \begin{equation}\label{equalityEisensteinClassic}E_{P_{(3,0)}}\left(\tilde{g},\frac{1}{2}, r(g_f)\phi \right) =c \cdot  \mathrm{Res}_{s = 1/2} E_{P_{(2,2)}}\left(\tilde{g},s,F_s(r(g_f)\phi)\right),\end{equation}
where $r$ denotes the action by right translation in the induced representation  $I_{P_{(3,0)}}\left(1/2,\mathbf{1}\right)$.
    Then the statement follows from showing that 
    \begin{align}\label{intermediate_08/08} r_{1/2}'(g_f)F_{1/2}(\phi) = F_{1/2}( r(g_f) \phi),
    \end{align}
    where $r_s'$ denotes the action by right translation in the induced representation $I_{P_{(2,2)}}(s, \mathbf{1})$, \emph{i.e.} from proving the $\GSp_6(\A)$-equivariance of the maps $$ I_{P_{(3,0)}}(1/2,\mathbf{1}) \xleftarrow{\lambda_{\mathbb{H}^2}(-)^\sim} \Pi(\mathbb{H}^2) \xrightarrow{F_{1/2}(-)} I_{P_{(2,2)}}(1/2, \mathbf{1}).$$
If we let $u \in\Sp_6(\A)$, we have 
$$(r(g_f) \phi)(u) = r(g_f) \lambda_{\mathbb{H}^2}(\varphi)^\sim(u) = | \nu(g_f)|^{-3} \cdot \omega(u g'_f)\varphi(0),$$
for $\varphi \in \mathcal{S}(\mathbb{H}^2(\A)^3)$ and $g'_f = \left(\begin{smallmatrix}I_3& \\ &\nu(g_f)^{-1}I_3\end{smallmatrix}\right) g_f\in \Sp_6(\A)$. Thus, 
the right hand side
of \eqref{intermediate_08/08} equals  \begin{align*} F_{1/2}( r(g_f) \phi)(u) &= F\left(r(g_f) \phi, u, 1/2\right) \\ &= | \nu(g_f)|^{-3} \int_{\GL_2(\A)}\omega(u g_f')  \hat{\varphi}
    (^{t}aw_0)|\det(a)|^{3} d a.
\end{align*}
The left hand side
of \eqref{intermediate_08/08} is 
\begin{align*} r_{1/2}'(g_f)F_{1/2}(\phi)(u) &= F\left( \phi, u g_f, 1/2\right) \\ &= \int_{\GL_2(\A)} \omega(h_f,u g_f) \hat{\varphi}
    (^{t}aw_0)|\det(a)|^{3} d a \\  &= | \nu(g_f)|^{-3}\int_{\GL_2(\A)} \omega(u g'_f)\hat{\varphi}({}^{t}aw_0)|\det(a)|^{3}da,
\end{align*}
where we chose $h_f = \left(\begin{smallmatrix} 1 & & & \\ & 1 & & \\ & & \nu(g_f) & \\ & & & \nu(g_f) \end{smallmatrix}\right) \in \GO_{\mathbb{H}^2}(\A_f)$ and, for the last equality, we argued as in Lemma \ref{lemma_constructing_sectionsforR}. Comparing both sides, we obtain the desired equality. 
\end{proof}

\subsection{On the non-vanishing of the central $L$-value}\label{subsec_non_vanishing_central_L_value_degree_12}

We now apply Theorems \ref{Siegel-Weilnew} and \ref{SWcentralSplit} to give a criterion for the non-vanishing of the value at $s=1/2$ of the zeta integral, hence of the central value of the degree $12$ $L$-function.

 By Theorem \ref{Siegel-Weilnew}, if we let $V$ range through the (finite) set of global quadratic spaces of dimension $4$ over $F$ for which the image of $\phi_{1/2}$ via the projection ${\rm pr}_{\Pi(V)}: I_{P_{(3,0)}}(1/2, \mathbf{1}) \to  \Pi(V)$ is non-zero, we have
\[E_{P_{(3,0)}}^*(g, 1/2,\phi_{1/2}) = 2 \,\zeta^S_F(2)^2  \sum_V I(g, f^V),\]
where $f^V \in \mathcal{S}(V(\A)^3)$ is such that ${\rm pr}_{\Pi(V)} \phi_{1/2}  = \lambda_V(f^V)^{\sim}$.

\begin{remark}\label{RemarkonHowManyVAppear}
The quaternion algebras associated to these $V$'s are split away from $S$ as $\phi_s$ is unramified at each $v \not \in S$.     
\end{remark}
\noindent Plugging this in the zeta integral, we have 
\begin{align*}
    Z(\tfrac{1}{2},\varphi,f^\Psi_{\phi_{1/2}}) &= 2 \,\zeta^S_F(2)^2  \sum_V \int_{Z_{\GSp_6}(\A)\backslash [\GL_2\boxtimes \GSp_4]} I(\iota( g_1  , g_2), f^V) \Psi^w( g_1 ) \varphi(g_2)\, dg_1\,dg_2 \\ &= 2 \,\zeta^S_F(2)^2  \sum_V \int_{Z_{\GSp_6}(\A)\backslash [\GL_2\boxtimes \GSp_4]} \int_{[\mathbf{O}_V]} \theta(h h_1, \iota( g_1  , g_2) ; f^V) \Psi^w( g_1 ) \varphi(g_2) \, d h \, dg_1\,dg_2,
\end{align*}
where $h_1 \in \mathbf{GO}_V(\A)$ has multiplier $\nu(h_1) = \nu (\iota( g_1  , g_2))$.

We now introduce the algebraic group $\mathbf{J}$ over $F$ given by
\[\mathbf{J}(R) := (\GL_2\boxtimes \GL_2)(R) = \{(h_1,h_2)\in \GL_2(R)\times \GL_2(R)\;:\;\mathrm{det}(h_1) = \mathrm{det}(h_2)\},\]
for any $F$-algebra $R$. We embed it into $\GSp_4$ via \begin{align*} j:\mathbf{J} &\to \GSp_4, \\ (h_1,h_2)&\mapsto  j(h_1,h_2) := \left(\begin{smallmatrix}
    d& &c\\ &h_2& \\ b& &a
\end{smallmatrix}\right),\end{align*}
where we denoted $h_1 = \left(\begin{smallmatrix}a&b\\c&d\end{smallmatrix}\right)$. We thus embed $\mathbf{J}$ into $\GL_2\boxtimes \GSp_4$ using the map
\begin{align}\label{Embedding:On:GSp4} \mathbf{J} \to \GL_2\boxtimes \GSp_4, \,\, (h_1,h_2)\mapsto  (h_1, j(h_1,h_2)).\end{align}
Let 
\begin{align}\label{def_open_orbit_forJR}
    \tau_{1} := \left(\begin{smallmatrix}0&0&1&&&\\0&1&0&&&\\1&0&0&&&\\ &&&0&0&1\\&&&0&1&0\\&&&1&0&0 \end{smallmatrix}\right),\, \tau_{2} := \left(\begin{smallmatrix}1&0&0&&&\\0&1&0&&&\\-1&0&1&&&\\ &&&1&0&0\\&&&0&1&0\\&&&1&0&1 \end{smallmatrix}\right),\,  \tau_{\rm op} := \left(\begin{smallmatrix}1&0&&&&\\0&1&&&&\\&&1&0&&\\ &&0&1&&\\1&0&&&1&0\\0&1&&&0&1 \end{smallmatrix}\right).
\end{align} 

\begin{lemma}\label{doublecosetPtilde}
    The double quotient $P_{(2,2)}(F)\setminus \GSp_6(F)/\iota\left(\GL_2\boxtimes\GSp_4\right)(F)$ consists of four elements represented by $I_6, \tau_1, \tau_2$, and $\tau_{\rm op}$, with corresponding flag and stabilizer as follows:
    \begin{enumerate}
        \item $P_{(2,2)}(F)\cdot I_6$ corresponds to the flag $0 \subset \langle f_1,f_2 \rangle $ and has stabilizer $\iota(P_{(1,0)}\boxtimes P_{(1,2)})(F)$.
        \item $P_{(2,2)}(F)\cdot \tau_1$ corresponds to the flag $0 \subset \langle f_2,f_3 \rangle $ and has stabilizer $\iota\left(\GL_2\boxtimes P_{(2,0)}\right)(F)$.
        \item $P_{(2,2)}(F)\cdot \tau_2$ corresponds to the flag $0 \subset \langle f_1+f_3,f_2 \rangle $ and has stabilizer 
        \[\left \{ \left(\begin{smallmatrix}a & b \\  & d\end{smallmatrix}\right) \times \left(\begin{smallmatrix}
      \alpha &  &&\\&a&&\\&&d & \\&&&\delta 
\end{smallmatrix}\right)\,:\, ad = \alpha \delta \right \} \cdot \left(\{ I_2\} \times U_{\GSp_4}(F)\right).\]
        \item $P_{(2,2)}(F) \cdot \tau_{\rm op}$ corresponds to $0 \subset \langle e_2+f_1, e_1+f_2 \rangle$ and has stabilizer $ \mathbf{J}(F).$
    \end{enumerate}
    
\end{lemma}
\begin{proof}
Recall that we have fixed the basis 
  $\{e_1,e_2,e_3,f_3,f_2,f_1\}$ of standard representation $W_6$ of $\GSp_6$ and that the embedding $\iota$ is induced by the decomposition of $W_6 = W_2 \oplus W_4$, with $W_2=\langle e_1,f_1\rangle$ and $W_4= \langle e_2,e_3,f_3,f_2\rangle$. 
Any flag in $P_{(2,2)}(F)\setminus  \GSp_6(F)$ is represented by $0 \subset S$, with $S = \langle v_1, v_2 \rangle$ an isotropic plane in the standard representation $W_6$ of $\GSp_6$. There are 4 orbits with respect to the action of $\iota\left(\GL_2\boxtimes\GSp_4\right)(F)$ which can be described in terms of the dimensions of the intersections $S \cap W_2$ and $S \cap W_4$, as we now explain. We start by distinguishing two cases, depending on whether $S$ is in $W_4$ or not. 
If $S \subset W_4$, then we can use the action of $\GSp_4(F)$ to move $\langle v_1, v_2 \rangle$ to $\langle f_2, f_3 \rangle$. Thus, the set of flags $0 \subset S$ such that $S \subset W_4$ forms an orbit $P_{(2,2)}(F) \cdot \tau_1$ with stabilizer $\iota\left(\GL_2\boxtimes P_{(2,0)}\right)(F)$. 

If $S \not\subset W_4$, then $\dim S \cap W_2$ is either 1 or 0. If $\dim S \cap W_2 = 1$, then we can assume that $v_1 \in W_2$. Using the action of the $\GL_2$-block we can assume that $v_1 = f_1$. Since $S$ is isotropic, we can then assume that $v_2$ belongs to $W_4$ and, using the action of $\GSp_4(F)$, we can move $v_2$ to $f_2$. Thus, the set of flags $0 \subset S$ such that $S \not\subset W_4$ and $\dim S \cap W_2 = 1$ forms an orbit represented by the identity matrix $I_6$ with stabilizer $$P_{(2,2)}(F) \cap \iota(\GL_2 \boxtimes \GSp_4)(F) = \iota(P_{(1,0)} \boxtimes P_{(1,2)})(F).$$

If $S\not\subset W_4$ and $\dim S \cap W_2 = 0$, we might assume that $v_1 = f_1 + f_3$. As $S$ is isotropic, $v_2 \not \in W_2$. If $v_2$ is in $W_4$, it can be mapped to $f_2$; the flag $0 \subset \langle f_1+f_3,f_2 \rangle$ represents the orbit of flags $0 \subset S$ such that $\dim S \cap W_2=0$ and $\dim S \cap W_4 =1$. It is equal to $P_{(2,2)}(F) \cdot \tau_2$. 
By an explicit matrix calculation, we obtain the stabilizer given in the statement of the Lemma. Finally, the set of flags $0 \subset S$ such that $\dim S \cap W_2 = \dim S \cap W_4 =0$ forms the remaining orbit. We choose it to be represented by $0 \subset \langle e_2+f_1, e_1+f_2 \rangle$, which corresponds to the element $\tau_{\rm op}$ defined in \eqref{def_open_orbit_forJR}. The calculation of the stabilizer follows immediately. 
\end{proof}

\begin{theorem}\label{Final:Theorem:Central:Value}
    Let $\pi$ be a globally generic cuspidal automorphic representation of either $\GL_4(\A)$ or $\GU_{2,2}(\A)$ and let $\sigma$ be a cuspidal automorphic representation of $\GL_2(\A)$ such that as a Hecke character on $F^\times \backslash \A^\times$ \begin{equation*}
    \omega_\pi \omega_\sigma = 1.
\end{equation*}
We also let $S$ be a finite set of places containing the archimedean places for $F$ and the ramified places for $\pi$, $\sigma$, and $E/F$. Let $\varphi = \otimes_v \varphi_v$ and $\Psi = \otimes_v \Psi_v$ be cusp forms in $\pi$ and $\sigma$, and let  $\phi_{s} = \otimes_v \phi_{v,s} \in I_{P_{(3,0)}}(s,\mathbf{1})$. If 
$$ L^S(1/2, \pi \otimes \sigma, \wedge^2 \otimes  \mathrm{std}_2) \prod_{v \in S} Z\left(1/2,W_{\pi_v}, W_{\sigma_v,f_{\phi_{1/2,v}}^{\Psi_v}} \right) \ne 0,$$ then either 
\begin{align}
\mathcal{P}_{\mathbf{J}}( \varphi \otimes \Psi) : = \int_{Z_{\GSp_6}(\A)  \backslash[ \mathbf{J}]} \Psi(g_1) \varphi\left( \left(\begin{smallmatrix}
    a& &b\\ &g_2& \\ c& &d
\end{smallmatrix}\right) \right )  dg_1 \, dg_2 \ne 0,  
\end{align}
 where we denoted $g_1 = \left(\begin{smallmatrix}a&b\\c&d\end{smallmatrix}\right)$, or 
\begin{align}\label{FormulaWithTheta}
    \int_{Z_{\GSp_6}(\A)\backslash [\GL_2\boxtimes \GSp_4]} \int_{[\mathbf{O}(D)]} \theta(h h_1, \iota( g_1  , g_2) ; f^D) \Psi^w( g_1 ) \varphi(g_2) \, d h \, dg_1\,dg_2 \ne 0,
\end{align}
where $\mathbf{O}(D)$ is the orthogonal group attached to the norm form of a non-split quaternion algebra $D/F$ and $f^D \in \mathcal{S}(D(\A)^3)$ is such that ${\rm pr}_{\Pi(D)} \phi_{1/2}  = \lambda_D(f^D)^{\sim}$. 
\end{theorem}
   
\begin{proof}
By Theorem \ref{zetaintegralfinalthm}, we have $$ Z(1/2,\varphi, f^\Psi_{\phi_{1/2}}) =L^S(1/2, \pi \otimes \sigma, \wedge^2 \otimes  \mathrm{std}_2) \prod_{v \in S} Z\left(1/2,W_{\pi_v}, W_{\sigma_v,f_{\phi_{1/2,v}}^{\Psi_v}} \right).$$ Moreover, the Siegel-Weil formula gives
\begin{align*}
    Z(1/2,\varphi, f^\Psi_{\phi_{1/2}}) &= 2 \,\zeta^S_F(2)^2  \sum_V \int_{Z_{\GSp_6}(\A)\backslash [\GL_2\boxtimes \GSp_4]} I(\iota( g_1  , g_2), f^V) \Psi^w( g_1 ) \varphi(g_2)\, dg_1\,dg_2 \\ &= 2 \,\zeta^S_F(2)^2  \sum_V \int_{Z_{\GSp_6}(\A)\backslash [\GL_2\boxtimes \GSp_4]} \int_{[\mathbf{O}_V]} \theta(h h_1, \iota( g_1  , g_2) ; f^V) \Psi^w( g_1 ) \varphi(g_2) \, d h \, dg_1\,dg_2,
\end{align*}
hence the non-vanishing of $Z(1/2,\varphi, f^\Psi_{\phi_{1/2}})$ implies the non-vanishing of at least one of the terms of the sum. To prove the Theorem, we are thus left to show that, if the term in the sum corresponding to the totally split quadratic space $\mathbb{H}^2$ \begin{align}\label{eqn_7_7_2025}
    \int_{Z_{\GSp_6}(\A)\backslash [\GL_2\boxtimes \GSp_4]} \int_{[\mathbf{O}_{\mathbb{H}^2}]} \theta(h h_1, \iota( g_1  , g_2) ; f^{\mathbb{H}^2}) \Psi^w( g_1 ) \varphi(g_2) \, d h \, dg_1\,dg_2 
\end{align} is non-zero for some choice of data, then the functional $\mathcal{P}_\mathbf{J}|_{\pi\otimes\sigma}$ is not identically zero.
To show that, we use Theorem \ref{SWcentralSplit} (and Theorem \ref{Siegel-Weilnew}(1)) to write \eqref{eqn_7_7_2025} as 
  \begin{align}\label{eqn2_7_7_2025}
    c \cdot \int_{Z_{\GSp_6}(\A)\backslash [\GL_2\boxtimes \GSp_4]} {\rm Res}_{s=1/2} E_{P_{(2,2)}}(\iota(g_1,g_2), s, F_s(\phi)) \Psi^w( g_1 ) \varphi(g_2) \, d h \, dg_1\,dg_2, 
\end{align}  
where $\phi = \lambda_{\mathbb{H}^2}(f^{\mathbb{H}^2})^{\sim}$ and $c$ is a non-zero constant. We can now prove our claim. Assume that \eqref{eqn_7_7_2025}, hence \eqref{eqn2_7_7_2025}, is non-zero for some choice of data. This implies that for ${\rm Re}(s)$ large enough the integral 
$$Z'(\varphi, \Psi, F_s(\phi)): = \int_{Z_{\GSp_6}(\A)\backslash [\GL_2\boxtimes \GSp_4]} E_{P_{(2,2)}}(\iota(g_1,g_2), s, F_s(\phi)) \Psi^w( g_1 ) \varphi(g_2) \, d h \, dg_1\,dg_2$$
is not identically zero. We use Lemma \ref{doublecosetPtilde} to unfold $Z'(\varphi, \Psi, F_s(\phi))$ and write it as
\begin{align*}
    \sum_{\tau}\int_{Z_{\GSp_6}(\A)\backslash [\GL_2\boxtimes \GSp_4]} \sum_{\gamma \in S_\tau(F) \backslash (\GL_2\boxtimes \GSp_4)(F) } F(\phi, \tau \gamma \iota (g_1,g_2) , s) \Psi^w( g_1 ) \varphi(g_2) \, d h \, dg_1\,dg_2,
\end{align*}
where $\tau$ runs through the set $\{I_6, \tau_1, \tau_2, \tau_{\rm op} \}$ and $S_\tau(F)$ denotes the stabilizer of $\tau$ in $\left(\GL_2\boxtimes\GSp_4\right)(F)$. When $\tau = I_6, \tau_2$, we note that, after collapsing the the sum over the unipotent part of $S_\tau(F)$, the corresponding integral vanishes because of cuspidality of $\Psi$. Similarly, the integral corresponding to $\tau = \tau_1$ contains 
$$\int_{\SL_2(F)\backslash \SL_2(\A)} \Psi^w(h g_1)\, dh ,$$
which is again zero for any cusp form $\Psi$. We thus have 

$$Z'(\varphi, \Psi, F_s(\phi)) =  \int_{Z_{\GSp_6}(\A) \mathbf{J}(F)\backslash (\GL_2\boxtimes \GSp_4)(\A)} F(\phi, \tau_{\rm op}\iota (g_1,g_2) , s) \Psi^w( g_1 ) \varphi(g_2) \, d h \, dg_1\,dg_2.$$
Collapse the sum over $[\mathbf{J}]$ to write the integral as
\small
$$Z'(\varphi, \Psi, F_s(\phi)) =  \int_{\mathbf{J}(\A)\backslash (\GL_2\boxtimes \GSp_4)(\A)} \int_{Z_{\GSp_6}(\A) \mathbf{J}(F)\backslash \mathbf{J}(\A)}F(\phi, \tau_{\rm op} \,\, \iota (h_1g_1,j(h_1,h_2)g_2) , s) \Psi^w( h_1 g_1 ) \varphi(j(h_1,h_2)g_2) \, d h \, dg.$$
\normalsize
The modulus character $\delta_{P_{(2,2)}}$ is trivial on $\tau_{\rm op} \mathbf{J}(\A) \tau_{\rm op}^{-1}$, hence 
\small
\begin{align*}
    Z'(\varphi, \Psi, F_s(\phi)) &=  \int_{\mathbf{J}(\A)\backslash (\GL_2\boxtimes \GSp_4)(\A)}F(\phi, \tau_{\rm op} \, \iota (g_1,g_2) , s)  \int_{Z_{\GSp_6}(\A) \mathbf{J}(F)\backslash \mathbf{J}(\A)}\Psi^w( h_1 g_1 ) \varphi(j(h_1,h_2)g_2) \, d h \, dg \\ 
    &=  \int_{\mathbf{J}(\A)\backslash (\GL_2\boxtimes \GSp_4)(\A)}F(\phi, \tau_{\rm op} \, \iota (g_1^w,g_2) , s)  \int_{Z_{\GSp_6}(\A) \mathbf{J}(F)\backslash \mathbf{J}(\A)}  \Psi\left(\left(\begin{smallmatrix}
    a &b\\ c &d
\end{smallmatrix}\right) g_1\right) \varphi\left( \left(\begin{smallmatrix}
    a& &b\\ &h_2& \\ c& &d
\end{smallmatrix}\right) g_2 \right ) \, d h \, dg,
\end{align*}
\normalsize
where we write $h_1 = \left(\begin{smallmatrix}
    a &b\\ c &d
\end{smallmatrix}\right)$ and $g_1^w = w g_1 w$. Thus, $Z'(\varphi, \Psi, F_s(\phi))$ contains the period $\mathcal{P}_{\mathbf{J}}(\varphi \otimes \Psi)$ as an inner integral. Since, by hypothesis,  $ Z'(\varphi, \Psi, F_s(\phi))$ is not identically zero, it follows that $\mathcal{P}_{\mathbf{J}}$ is not identically zero on the space of $\sigma \otimes \pi$. This finishes the proof.  
\end{proof}

In the course of the proof of Corollary \ref{Main:Corollary:Non:Spherical}, we will need the following auxiliary Lemma on the properties of the map $p_\Psi$ introduced in \eqref{PullbackMapRepresentation}.

\begin{lemma}\label{Non:Vanishing:Pullback}
   \leavevmode \begin{enumerate}
       \item Given $s_0 \in \C$, there exist data $\phi_{s_0}\in I_{P_{(3,0)}}(s_0,\mathbf{1})$ and $\Psi$ such that $p_\Psi(\phi_{s_0})$ is not identically zero.
       \item If $\sigma$ is unramified everywhere, one can choose $\phi_{1/2}\in \Pi(\mathbb{H}^2)$ and $\Psi$ so that $p_\Psi(\phi_{1/2})$ is a non-trivial unramified section of $I_{P_{(1,2)}}(\tfrac{1}{2},\sigma)$.
   \end{enumerate} 
\end{lemma}
\begin{proof} \leavevmode
   \begin{enumerate} 
    \item The result follows from showing that, for a given  $s_0\in\C$, one can choose  $\alpha \in \sigma_v$ and $\phi_{v,s_0}$  at each place $v$ of $F$ such that the  local zeta integral $z(s, \phi_{v,s_0}, \alpha)$ of \eqref{Local:Zeta:Pullback:Aux} is non-zero. This is analogous to the existence of $\phi_{v,s_0}$ and $v_1,v_2\in\sigma_v$ such that 
    \begin{align}\label{lemma710_int_step}
        \int_{Z_{\SL_2}(F_v)\setminus \SL_2(F_v)}\phi_{v,s_0}(\xi\iota( g_1,1))\langle \sigma_v(g_1)v_1,v_2\rangle dg_1\neq 0.
    \end{align}
    The non-vanishing of the latter follows essentially from \cite[Theorem 3.2.2 (ii)]{Kudla:Rallis:lost}. We include a proof here for completeness, as we use it in the second part of the Lemma.
    Fix an isomorphism $I_{P_{(3,0)}}(s_0,\mathbf{1})_v\simeq \mathcal{C}^{\infty}(P_{(3,0)}(F_v)\setminus \GSp_6(F_v))$. From Lemma \ref{lemmaondoublecosetsimple}, we deduce that 
    \begin{equation*}P_{(3,0)}(F_v)\setminus \GSp_6(F_v)\simeq \Omega_{\rm op} \sqcup \Omega_{\rm cl},
    \end{equation*}
    where $\Omega_{\rm op} = \SL_2(F_v)\times P^{\circ}_{(1,2)}(F_v)\setminus \Sp_4(F_v)$ is open and dense and $\Omega_{\rm cl} = (P^{\circ}_{(1,0)}(F_v)\times P^{\circ}_{(2,0)}(F_v))\setminus(\SL_2\times \Sp_4)(F_v)$ is closed. The inverse image of $\Omega_{\rm op}$ in $\GSp_6(F_v)$ is $$P_{(3,0)}(F_v)\xi\iota(\SL_2(F_v),\Sp_4(F_v)).$$ This implies that any function $\varphi := \varphi_1\otimes\varphi_2\in \mathcal{C}^{\infty}_c(\SL_2(F_v))\otimes\mathcal{C}^{\infty}_c( \Sp_4(F_v))$ can be uniquely extended to a section $\Phi_{\varphi}\in I_{P_{(3,0)}}(s_0,\mathbf{1})_v$. Hence,  for any $v_1,v_2\in \sigma_v$, we obtain 
    \begin{align*}\int_{Z_{\SL_2}(F_v)\setminus \SL_2(F_v)}\!\!\!\!\!\!\Phi_{\varphi}(\xi\iota( g_1,I_4))\langle \sigma_v(g_1)v_1,v_2\rangle dg_1 &= \varphi_2(I_4)\int_{Z_{\SL_2}(F_v)\setminus \SL_2(F_v)}\!\!\!\!\!\varphi_1(g_1)\langle \sigma_v(g_1)v_1,v_2\rangle dg_1\\ &= \varphi_2(I_4)\langle\sigma_v(\varphi_1)v_1,v_2\rangle.\end{align*}
    Choosing $\varphi_1$ so that $\sigma_v(\varphi_1)v_1\neq 0$, then there exists $v_2$ such that the integral does not vanish.
    
    \item Let $\Psi = \otimes_v \Psi_v$ be a spherical cusp form in $\sigma$. The statement follows from choosing, at each place $v$ of $F$, a section $\phi_{1/2,v} \in \Pi(\mathbb{H}_{F_v}^2)$, with $\mathbb{H}_{F_v} = \mathbb{H} \otimes_F {F_v}$, such that $p_\Psi(\phi_{1/2})_v$ is a non-trivial spherical section. At non-archimedean places, this follows from Theorem \ref{PullbackFormulaNew}(2). We are thus left to show it at archimedean places of $F$. Firstly, suppose that $F_v = \R$. By \cite[Proposition 5.3(ii)]{GanQiuTakeda},  $$I_{P_{(3,0)}}(\tfrac{1}{2},\mathbf{1})_v =\Pi(\mathbb{H}^2_\R) \oplus \Pi(D),$$ where $\mathbb{H}^2_\R$ denotes hyperbolic plane over $\R$ and $D$ denotes the anisotropic quadratic space of dimension $4$ over $\R$. According to \cite[Proposition 2.8, (ii)]{KudlaRallis}, the $\U(3)$-types in the representation $\Pi(\mathbb{H}^2_\R)$ are precisely those with highest weight of the form
    \[(2a,2b,-2c)\;\;\text{and}\;\;(2\alpha,-2\beta,-2\gamma),\]
 where $a\geq b\geq 0$, $c\geq 0$ and $\alpha \geq 0$,  $\gamma \geq \beta \geq 0$. By \emph{loc.cit.}, the $\U(3)$-types in $\Pi(D)$ are instead of the form $$
    (2+2\alpha,2+2\beta, 2+2\gamma),$$ with $\alpha\geq \beta\geq \gamma\geq 0$. Since $\U(n)$ is a compact real form of $\GL_n(\C)$, the branching law for $\GL_3(\C)\downarrow \GL_2(\C)$ given in \cite[Theorem 1]{Branching:Law} induces one for $\mathbf{U}(3)\downarrow \mathbf{U}(2)$. By matching the action of the center, we deduce that a representation $V_{\lambda}$ of $\U(3)$ of highest weight $\lambda = (\lambda_1,\lambda_2,\lambda_3)$, with $\lambda_1\geq \lambda_2\geq \lambda_3$, decomposes into a sum of irreducible $\U(1)\times \U(2)$-representations as follows:
    \begin{align}\label{branching_law}
        V_{\lambda}|_{\U(1)\times \U(2)}\simeq \bigoplus_{\substack{\mu_1,\mu_2\in \Z\\\lambda_1\geq\mu_1\geq \lambda_2\geq \mu_2\geq\lambda_3}}\mathrm{det}^{|\lambda|-\mu_1-\mu_2}\otimes V_{(\mu_1,\mu_2)},
    \end{align}
    where $|\lambda| := \lambda_1+\lambda_2+\lambda_3$ and $V_{(\mu_1,\mu_2)}$ denotes the representation of $\U(2)$ of highest weight $(\mu_1,\mu_2)$.  By \eqref{branching_law}, any $\U(3)$-type of the form $(2a,2b,0)$ or $(0,-2\beta,-2\gamma)$ in $\Pi(\mathbb{H}^2_\R)$ contains a unique sub $\U(1)\times \U(2)$-representation, with $\U(1)$-component being the trivial representation. In contrast, no $\U(3)$-type in $\Pi(D)$ has this property. Hence, if we choose $\varphi_1 \in \mathcal{C}^{\infty}_c(\SL_2(\R))$ to be right $\U(1)$-invariant, we have $\Phi_{\varphi_1 \otimes \varphi_2}\in \Pi(\mathbb{H}^2_\R)$ for any $\varphi_2 \in \mathcal{C}^{\infty}_c(\Sp_4(\R))$. Moreover, if  $\varphi = \varphi_1 \otimes \varphi_2$ is $\U(1)\times \U(2)$-right invariant, the integral $z(1/2, \Phi_{\varphi}, \Psi_v)$ is non-zero because, following the proof of (1), we can choose $v_1,v_2\in \sigma_v$ such that $\langle\sigma_v(\varphi_1)v_1,v_2\rangle\neq 0$. 
    
    When $F_v = \C$, \cite[Theorem 1, Case IId]{Soo:Bo} shows that $$I_{P_{(3,0)}}(\tfrac{1}{2},\mathbf{1})_v =\Pi(V_\C),$$ with $V_\C \simeq (M_2(\C), {\rm det})$; the non-vanishing follows from (1). To conclude the proof, take $\phi_{1/2} = \otimes_v \phi_{1/2,v}\in \Pi(\mathbb{H}^2)$ such that $\phi_{1/2,v}$ is the unramified section at any finite or complex place $v$ and equals $\Phi_{\varphi}$, with $\varphi$ being $\U(1)\times \U(2)$-right invariant, at any real place. Then, the section $p_{\Psi}(\phi_{1/2})$, attached to such $\phi_{1/2}$ and the spherical cusp form $\Psi$, is unramified and non-zero by the discussion above.
    \end{enumerate}
\end{proof}

We now discuss what Theorem \ref{Final:Theorem:Central:Value} says in the case of cuspidal automorphic representations on $\GL_4(\A) \times \GL_2(\A)$ which are unramified everywhere. 

\begin{corollary}\label{Main:Corollary:Non:Spherical}
     Let $\pi$ and $\sigma$ be cuspidal automorphic representations of $\GL_4(\A)$ and $\GL_2(\A)$ which have all local components unramified and such that $\omega_\pi \omega_\sigma = 1.$ Then, 
\[L(\tfrac{1}{2}, \pi \otimes \sigma, \wedge^2 \otimes  \mathrm{std}_2) \neq 0 \Longrightarrow {\mathcal{P}_{ \mathbf{J}}}|_{\pi \otimes \sigma} \not \equiv 0,\]
where we recall that $L(s, \pi \otimes \sigma, \wedge^2 \otimes  \mathrm{std}_2)$ denotes the completed $L$-function.
\end{corollary} 
\begin{proof}
By Lemma \ref{Non:Vanishing:Pullback}, we can choose $\phi_{1/2}\in \Pi(\mathbb{H}^2)$ so that $p_\Psi(\phi_{1/2})$ is a non-trivial unramified section of $I_{P_{(1,2)}}(\tfrac{1}{2},\sigma)$. Thus,  by Theorems \ref{zetaintegralfinalthm} and \ref{Theorem_Arch_computation_split}, we have $$ Z(1/2,\varphi, p_\Psi(\phi_{1/2})) =L(1/2, \pi \otimes \sigma, \wedge^2 \otimes  \mathrm{std}_2).$$
On the other hand, since $\pi$ and $\sigma$ are unramified everywhere, the Siegel Weil formula gives 
$$ Z(1/2,\varphi, p_\Psi(\phi_{1/2})) = 2 \,\zeta_F(2)^2  \int_{Z_{\GSp_6}(\A)\backslash [\GL_2\boxtimes \GSp_4]} \int_{[\mathbf{O}_{\mathbb{H}^2}]} \theta(h h_1, \iota( g_1  , g_2) ; f^{\mathbb{H}^2}) \Psi^w( g_1 ) \varphi(g_2) \, d h \, dg_1\,dg_2,$$
where $\phi_{1/2} = \lambda_{\mathbb{H}^2}(f^{\mathbb{H}^2})^\sim$. Arguing as in Theorem \ref{Final:Theorem:Central:Value}, the result follows.
\end{proof}
\subsection{On the generalized Shalika period}\label{Shalika_period_subsec}
Consider the embedding 
\begin{align*}
    m: \GL_2 &\to \GL_4,\\
    g &\mapsto  \mathrm{diag}(g,g) = \left(\begin{smallmatrix}
        g& \\ &g\end{smallmatrix}\right).
\end{align*}
We also let $N$ denote the unipotent radical of the Siegel parabolic of $\GL_4$, defined by 
$$ N(R) = \left \{n(A) : =  \left(\begin{smallmatrix}
    I_2 & A \\ & I_2
\end{smallmatrix} \right),\, \text{with }A \in M_{2}(R) \right\},$$
for any $F$-algebra $R$. Observe that $\GL_2$ embeds via $m$ into the Levi of the Siegel parabolic, thus acts by conjugation on $N$. The Shalika subgroup of $\GL_4$ is given by $S : = \GL_2 \,  N$. Fix the character $\psi_S: [S] \to \C^\times$, given by \[ m(g)\, n(A) \mapsto \psi\left({\rm Tr}(A)\right).  \]

\begin{definition}\label{definition_shalika_generalized}
   Given cusp forms $\varphi \in \mathcal{A}_{{\rm cusp}}(\GL_4)$, $\Psi \in \mathcal{A}_{{\rm cusp}}(\GL_2)$ with trivial product of central characters, define the generalized Shalika period by
   \[\mathcal{S}_{\psi_S}({\varphi,\Psi}):= \int_{\A^\times \backslash[\GL_2]} \int_{[N]}  \varphi(n\, m(g))\Psi(g) \psi_S^{-1}(n) dn\,dg.\]
\end{definition}

For any character $\chi$ of $[\GL_1]$ and cusp forms $\varphi \in \mathcal{A}_{{\rm cusp}}(\GL_4)$, $\Psi \in \mathcal{A}_{{\rm cusp}}(\GL_2)$ with trivial product of central characters, we also define 
\[\mathcal{P}_{\GL_2\times\GL_2}(\varphi,\Psi,\chi) := \int_{Z_{\GL_4}(\A)\backslash[\GL_2\times \GL_2]}\varphi(\mathrm{diag}(g_1,g_2))\Psi(g_2)\chi(\mathrm{det}(g_1)/\mathrm{det}(g_2))dg_1dg_2,\]
where $\mathrm{diag}(g_1,g_2) = \left(\begin{smallmatrix}
        g_1& \\ &g_2\end{smallmatrix}\right)$. This is the $\chi$-twisted automorphic period associated with the spherical variety   \[(\GL_2\times \GL_2)\setminus (\GL_4\times \GL_2).\]
The period $\mathcal{P}_{\GL_2\times\GL_2}(\varphi,\Psi,\chi)$ is readily seen to contain $\mathcal{P}_{\mathbf{J}}(\varphi\otimes\Psi)$ as an inner integral. Moreover, it also unfolds to the generalized Shalika period, as the following (well-known) Lemma shows.

\begin{lemma}\label{GL2:period:To:Shal}
    Let $\chi$ be a character of $F^{\times}\backslash \A^{\times}$, then $\mathcal{P}_{\GL_2\times\GL_2}(\varphi,\Psi,\chi)$ equals
    \begin{align*}\int_{\Delta(\GL_2(\A))\setminus (\GL_2(\A)\times \GL_2(\A))}\mathcal{S}_{\psi_S}(\pi(\mathrm{diag}(g_1,g_2))\varphi,\sigma(g_2)\Psi)\chi(\mathrm{det}(g_1)/\mathrm{det}(g_2))dg_1dg_2,
    \end{align*}
\end{lemma}
\begin{proof}
    We first Fourier expand the cusp form $\varphi$ along $[N]$, obtaining
    \[\varphi(g) = \sum_{\psi_{T}:[N]\to \C^{\times}}\varphi_{\psi_{T}}(g),\]
where $\psi_T:[N]\to \C^{\times}$ is of the form 
    \[\psi_{T}(n(X)) := \psi(\mathrm{Tr}(TX)),\;T\in M_2(F).\]
    The group $\GL_2(F)\times \GL_2(F)$ acts on the set of these characters with three orbits, indexed by the possible ranks of $T$. Representatives of each orbit  can be chosen as follows: \begin{itemize}
        \item $\psi_{0_2}$ with stabilizer $\GL_2(F)\times\GL_2(F)$,
        \item $\psi_{\left(\begin{smallmatrix}1& \\ & 0\end{smallmatrix}\right)}$ with stabilizer $\overline{B}_{\GL_2}(F)\times_{\rm det} B_{\GL_2}(F)$, where $\overline{B}_{\GL_2}$ denotes the lower triangular Borel subgroup of $\GL_2$. 
        \item $\psi_S = \psi_{I_2}$ with stabilizer $\Delta(\GL_2(F))$, \textit{i.e.} the group $\GL_2(F)$ diagonally embedded into $\GL_2(F)\times \GL_2(F)$. 
    \end{itemize}  By the cuspidality of both $\varphi$ and $\Psi$, one thus shows that the integral unfolds to 
    \begin{align*}
    \int_{\A^\times \backslash[\GL_2\times \GL_2]}&\varphi(\mathrm{diag}(g_1,g_2))\Psi(g_2)\chi(\mathrm{det}(g_1)/\mathrm{det}(g_2))dg_1dg_2 \\&= \int_{\Delta(\GL_2(F))\A^\times \backslash\GL_2(\A)\times \GL_2(\A)}\varphi_{\psi_{I_2}}(\mathrm{diag}(g_1,g_2))\Psi(g_2)\chi(\mathrm{det}(g_1)/\mathrm{det}(g_2))dg_1dg_2.
    \end{align*}
    Finally, collapsing the integral over $[\Delta(\GL_2)]$ yields the desired equality.
\end{proof}
\begin{proposition}\label{J:Period:To:GL2}
    Let $\pi$ and $\sigma$ be cuspidal automorphic representations of $\GL_4(\A)$ and $\GL_2(\A)$ such that $\omega_\pi \omega_\sigma = 1.$ Then $$P_{\mathbf{J}}|_{\pi\otimes\sigma}\not \equiv 0 \Longrightarrow \mathcal{S}_{\psi_S}|_{\pi\otimes\sigma}\not \equiv 0.$$
\end{proposition}
\begin{proof}
    Observe that the map \begin{align}\label{isom:Quotient:Shalika:To:J}
    \mathbf{J}\setminus \GL_2\times\GL_2&\to \GL_1,\\
    (g_1,g_2)&\mapsto \mathrm{det}(g_1)/\mathrm{det}(g_2),\nonumber
    \end{align}
    is an isomorphism. Hence, collapsing the integral over $[\mathbf{J}]$ yields
    \[\mathcal{P}_{\GL_2\times\GL_2}(\varphi,\Psi,\chi) = \int_{[\GL_1]}\mathcal{P}_{\mathbf{J}}(\pi(\mathrm{diag}(h_{1,t},h_{2,t})) \cdot   \varphi \otimes \sigma(h_{2,t}) \cdot \Psi) \chi(t)dt,\]
    where $(h_{1,t},h_{2,t})\in \GL_2(\A)\times\GL_2(\A)$ is any element which maps to $t$ via \eqref{isom:Quotient:Shalika:To:J}. By the injectivity of the Fourier transform  (see \cite[Corollary 4.34]{Folland:Harmonic:Analysis}), if $P_{\mathbf{J}}|_{\pi\otimes\sigma}\not \equiv 0$, there exists $\varphi \in \pi, \Psi \in \sigma,$ and $\chi$ such that $\mathcal{P}_{\GL_2\times\GL_2}(\varphi,\Psi,\chi) \ne 0$. The result then follows from Lemma \ref{GL2:period:To:Shal}.
\end{proof}

Finally, we prove Conjecture \ref{conj:Wan:Zhang:+:Bessel:GGP}(1) for representations which are unramified everywhere.

\begin{theorem}\label{Final:Theorem:Shalika}
    Let $\pi$ and $\sigma$ be cuspidal automorphic representations of $\GL_4(\A)$ and $\GL_2(\A)$ which have all local components unramified and such that $\omega_\pi \omega_\sigma = 1.$ Then, 
    \[L(\tfrac{1}{2}, \pi \otimes \sigma, \wedge^2 \otimes  \mathrm{std}_2) \neq 0 \Longleftrightarrow \mathcal{S}_{\psi_S}|_{\pi \otimes \sigma} \not \equiv 0,\]
    where we recall that $L(s, \pi \otimes \sigma, \wedge^2 \otimes  \mathrm{std}_2)$ denotes the completed $L$-function.
\end{theorem}
\begin{proof}
    The implication $(\Longrightarrow)$ follows from a combination of Proposition \ref{J:Period:To:GL2} and Corollary \ref{Main:Corollary:Non:Spherical}.

    \noindent The implication $(\Longleftarrow)$ follows from \cite[Theorem 1.3]{PanYan} in the case where $m'=7$, $k=2$, and $\ell=0$. Indeed, under the double covering map $\GSpin_6 \to \GL_4$, the Bessel subgroup of $\GSpin_6$, as defined in \cite[\S 2.2]{PanYan}, is isomorphic to the Shalika subgroup of $\GL_4$.   
\end{proof}

\section{The degree $20$ central $L$-value and the linear period}\label{sec:SVvsLcentral}

In what follows, we restrict ourselves to the case of $\GL_4 \times \GL_2$.

\subsection{Global theta liftings for general linear groups}
We recall some properties of the theta lift for the type $\mathrm{II}$ dual reductive pair $(\GL_4 ,\GL_4)$. We invite the reader to consult \cite{Watanabe:Global:Theta} for more details. 

Let $\mathcal{S}(M_4(\A))$  be the space of Schwartz function on $M_4(\A)$, where $M_4(\A)$ denotes the space of $4\times 4$-matrices with entries in $\A$. The Weil representation $\omega_{4,4}$ of  $\GL_4 (\A)\times \GL_4(\A)$ on $\mathcal{S}(M_4(\A))$ is defined by $$(\omega_{4,4}(h,g)f)(x) = | {\rm det}(h)|_\A^{-2} |{\rm det}(g)|_\A^2 f(h^{-1} x g),$$ for every $g,h \in \GL_4(\A)$ and $f \in \mathcal{S}(M_4(\A))$.

Denote by $\mathcal{S}_0(M_4(\A))$ the space of  $K_{\GL_4}\times K_{\GL_4}$-finite vectors of $\mathcal{S}(M_4(\A))$, where $K_{\GL_4}$ is the standard maximal compact subgroup of $\GL_4(\A)$.
Let $\pi$ be a cuspidal representation of $\GL_4(\A)$. For a cusp form $\varphi\in\pi$, a function $f\in \mathcal{S}_0(M_4(\A))$, and a complex number $z\in \C$, define
\begin{align} \varphi^z_f(h) &:= \int_{[\GL_4]}\varphi(g)|\det(g)|^{z}_{\A}\sum_{\substack{x\in M_4(F)\\x\neq 0}}(\omega_{4,4}(h,g)f)(x)dg \nonumber \\
&\,=|\det(h)|^{-2}_{\A}\int_{[\GL_4]}\varphi(g)|\det(g)|^{z+2}_{\A}\sum_{\substack{x\in M_4(F)\\x\neq 0}}f(h^{-1}xg)dg. \label{Theta:Lift:Expression}
\end{align}

\noindent This integral is absolutely convergent for $\mathrm{Re}(z)$ big enough and can be analytically continued to all $\C$ as an entire function, which we still call $\varphi^z_f$, by \cite[Lemma 3]{Watanabe:Global:Theta}. Moreover, $\varphi^z_f$ is a (possibly zero) automorphic form on $\GL_4$ and is a cusp form for $\mathrm{Re}(z)$ big enough (see \cite[p. 707]{Watanabe:Global:Theta}).
For a fixed $z \in \C$, let $\Theta^z(\pi)$ be the space spanned by $\varphi^z_f(h)$ as $\varphi$ varies in the space of $\pi$ and $f$ among the subspace of $\mathcal{S}(M_4(\A))$ of $K_{\GL_4}\times K_{\GL_4}$-finite vectors. If $\omega_\pi$ denotes the central character of $\pi$, then $\Theta^z(\pi)$ has central character $\omega_\pi(\cdot)| \cdot |^{4z}_\A$. 

\begin{proposition}[{\cite[Theorem 1]{Watanabe:Global:Theta}}]\label{Watanabe:Theta:Main}
Let $L(z,\pi,{\rm std}_4)$ be the standard $L$-function of $\pi$. Then, the space $\Theta^0(\pi)$ is non-zero if and only if $L(\tfrac{1}{2},\pi,{\rm std}_4) \ne 0$. In this case, $\Theta^0(\pi)$ coincides with $\pi$.  
\end{proposition}

\noindent The Proposition implies that, if $L(\tfrac{1}{2},\pi,{\rm std}_4)\neq 0$, for any cusp form $\varphi' \in \pi$, there exist a cusp form $\varphi\in \pi$ and a $K_{\GL_4}\times K_{\GL_4}$-finite function $f\in \mathcal{S}(M_4(\A))$ such that $\varphi' = \varphi^{0}_f$. Thanks to the following Lemma, one can choose such $f \in \mathcal{S}(M_4(\A))$ with support away from $0$.

\begin{lemma}\label{Support:Away:From:0}
    For any cuspidal automorphic representation $\pi$ of $\GL_4(\A)$ with $\Theta^0(\pi) = \pi$ and any $\varphi' \in \pi$, there exists $\varphi \in \pi$ and $f\in \mathcal{S}(M_4(\A))$ with $0\not\in\mathrm{supp}(f)$ such that $\varphi' = \varphi^0_f$.
\end{lemma}
\begin{proof}
    Let $\mathcal{H}_{K_{\GL_4}} :=\mathcal{C}^{\infty}_c(\GL_4(\A)//K_{\GL_4})$ denote the Hecke algebra of $\GL_4(\A)$. The Weil representation $\omega_{4,4}$ defines a representation of $\mathcal{H}_{K_{\GL_4}}\times \mathcal{H}_{K_{\GL_4}}$ on the space $\mathcal{S}_0(M_4(\A))$. By \cite[Proof of Theorem 2]{Watanabe:Global:Theta}, the map
    \begin{align*}
        \mathcal{S}_0(M_{4}(\A)) \otimes \pi \otimes \pi^\vee&\to \C,\\        (f,\varphi_1,\varphi_2)&\mapsto \langle\varphi_{1,f}^0,\varphi_2\rangle, 
    \end{align*}
    with $\langle\,,\,\rangle$ denoting the Petersson inner product, is $\mathcal{H}_{K_{\GL_4}}\times \mathcal{H}_{K_{\GL_4}}$-equivariant. By duality, we thus get a $\mathcal{H}_{K_{\GL_4}}\times \mathcal{H}_{K_{\GL_4}}$-equivariant map   \begin{align*}
        \mathcal{S}_0(M_{4}(\A)) \otimes \pi \to \pi,\,\,\,        (f,\varphi)\mapsto \varphi_{f}^0, 
    \end{align*}
    such that $(\beta * f * \alpha , \alpha * \varphi) \mapsto  (\alpha * \varphi)^0_{\beta * f * \alpha} = \beta * (\alpha * \varphi)^0_{f * \alpha} =\beta * \varphi^0_{f} $ for any $\alpha, \beta \in \mathcal{H}_{K_{\GL_4}}$.
    Suppose now that $\varphi' = \varphi^0_{f'}$ for some $\varphi \in \pi$ and $f '\in \mathcal{S}_0(M_{4}(\A))$. For any non-trivial $\alpha\in \mathcal{H}_{K_{\GL_4}}$ such that $\int_{\GL_4(\A)}\alpha(r)dr = 0$, we have that $0\not\in\mathrm{supp}(f'*\alpha)$. By Hecke equivariance, it thus follows that $\varphi' = (\alpha * \varphi)^0_{f'*\alpha}$, as desired. 
\end{proof}
\subsection{The two variable zeta integral}\label{subsec:TwoVariableZeta}

In this section, we use the theta correspondence for $(\GL_4,\GL_4)$ to upgrade the zeta integral of \S \ref{sec:The_zeta_integral} to a two complex variable integral. 

Fix $\pi \subset \mathcal{A}_{\rm cusp}([\GL_4])$ and $\sigma \subset \mathcal{A}_{\rm cusp}([\GL_2])$. Let $s, z \in \C$. Given a cusp form $\varphi = \otimes_v \varphi_v\in \pi$, a section $f_s = \otimes_v f_{s,v} \in I_{P_{(1,2)}}(s,\sigma,\omega_{\pi})$, and a factorizable Schwartz function $\eta$ on $M_4(\A)$, define \[J(s,z,\varphi,f_s, \eta) := \int_{{\GSp_4}(F) Z_{\GSp_4}(\A)\setminus {\GSp_4}(\A)} \!\!\!\!\!\! E^*_{P_{(1,2)}}(g,s, f_s)\varphi_\eta^z(g)| {\rm det}(g)|_\A^{-z} dg.\]

\noindent For ${\rm Re}(z)$ large enough, this is well defined as $\varphi_\eta^z$ is a cusp form of central character $\omega_\pi(\cdot)| \cdot |^{4z}_\A$. Fix a finite set of places $S$ outside of which the data $(\varphi,f_s, \eta)$ is unramified. Then we have the following. 
\begin{lemma}\label{unfolding_for_2_vars}
    For big enough ${\rm Re}(s)$ and ${\rm Re}(z)$, we have \[J(s,z,\varphi,f_s, \eta) = L^S(2s, \sigma, {\rm Sym}^2 \otimes \omega_\pi)\int_{U_{\GSp_4}(\A)Z_{\GSp_4}(\A)\setminus {\GSp_4}(\A)} \!\!\!\!\!\!\!\!\!\!\!\!W_{\sigma,s}\left(g\right)W_{\varphi,\eta}^z(g)| {\rm det}(g)|_\A^{-z}dg,\]
    where $W_{\sigma,s}$ is as in \eqref{WhittakerFoRKES}, while $W_{\varphi,\eta}^z$ denotes the Whittaker model of $\varphi_\eta^z$.
\end{lemma}
\begin{proof}
    The proof is identical to the one of Proposition \ref{unfolding}.
\end{proof}

  By the proof of \cite[Theorem 1]{Watanabe:Global:Theta},  the Whittaker model of $\varphi_\eta^z$ can be expressed as 
$$W_{\varphi,\eta}^z(g) = | {\rm det}(g)|_\A^{z} \int_{\GL_4(\A)} W_\varphi(g h)| {\rm det}(h)|_\A^{z+2} \eta(h) d h. $$
Write $\varphi = \otimes_v ' \varphi_v$ and $\eta = \prod_v \eta_v$, then 
$$W_{\varphi,\eta}^z(g) = \prod_v W_{\varphi_v,\eta_v}^z(g_v),$$
   where
\begin{align}\label{eq_local_whittaker_theta}
W_{\varphi_v,\eta_v}^z(g_v) =  | {\rm det}(g_v)|_v^{z} \int_{\GL_4(F_v)} W_{\varphi_v}(g_v h_v)| {\rm det}(h_v)|_v^{z+2} \eta_v(h_v) d h_v. 
\end{align}
Recall also that $W_{\sigma,s}(g)$ decomposes as $\prod_v W_{\sigma_v,s}(g_v)$, with $W_{\sigma_v,s}(g_v)$ equal to the spherical vector $W_{s,v}^{\rm o}(g_v)$ of \eqref{SphericalWhittakerforKES} at almost all $v$.
\begin{theorem}\label{THM_two_var_integral} For big enough ${\rm Re}(s)$ and ${\rm Re}(z)$, we have the equality 
    \[ J(s,z,\varphi,f_s, \eta) = L^S(z+ \tfrac{1}{2},\pi,{\rm std}_4) L^S(s,\pi \otimes \sigma,\wedge^2 \otimes \mathrm{std}_2) \prod_{v \in S} J(s,z,\varphi_v,f_{s,v},\eta_v),\]
    where \[J(s,z,\varphi_v,f_{s,v},\eta_v) :=  \int_{U_{\GSp_4}(F_v)Z_{\GSp_4}(F_v)\setminus {\GSp_4}(F_v)} \!\!\!\!\!\!\!\!\!\!\!\!W_{\sigma_v,s}\left(g_v\right)W_{\varphi_v,\eta_v}^z(g_v)| {\rm det}(g_v)|_v^{-z}dg_v.\]
\end{theorem}
\begin{proof}

Lemma \ref{unfolding_for_2_vars} implies that, for big enough ${\rm Re}(s)$ and ${\rm Re}(z)$, \[\frac{J(s,z,\varphi,f_s, \eta)}{L^S(2s, \sigma, {\rm Sym}^2 \otimes \omega_\pi)} = \prod_v J(s,z,\varphi_v,f_{s,v},\eta_v).\] 

Let now $v \not \in S$. Then $W_{\varphi_v}$ is the spherical Whittaker function and $\eta_v = {\rm char}(M_4(\mathcal{O}_v))$. By \cite[Lemma 6.10]{GodementJacquet}, together with \cite[(1.5) \& (2.1)]{Watanabe:Global:Theta}, we have 
\begin{align}\label{eq_Godement_Jacquet}
W_{\varphi_v,\eta_v}^z(g_v) = L(z + \tfrac{1}{2}, \pi_v,{\rm std}_4)   | {\rm det}(g_v)|_v^{z}  W_{\varphi_v}(g_v).
\end{align}
We plug \eqref{eq_Godement_Jacquet} into $J(s,z,\varphi_v,f_{s,v},\eta_v)$ to write 
\begin{align*}
    J(s,z,\varphi_v,f_{s,v},\eta_v) &= L(z+ \tfrac{1}{2},\pi_v,{\rm std}_4) \int_{U_{\GSp_4}(F_v)Z_{\GSp_4}(F_v)\setminus {\GSp_4}(F_v)} \!\!\!\!\!\!\!W_{s,v}^{\rm o}\left(g_v\right)W_{\varphi_v}(g_v)dg_v.
    \end{align*}
The integral of the right hand side equals $L(2s, \sigma_v, {\rm Sym}^2\otimes \omega_{\pi_v} )^{-1} \cdot Z(s,W_{\pi_v},W_{s,v}^{\rm o})$, where $Z(s,W_{\pi_v},W_{s,v}^{\rm o})$ is the integral calculated in \S \ref{subsec_localzeta1}. Hence,  Corollaries \ref{ZtoLtensorNonSplit} and \ref{Lrequaltowedge} give 
\begin{align*}
    J(s,z,\varphi_v,f_{s,v},\eta_v) &= \frac{L(z+ \tfrac{1}{2},\pi_v,{\rm std}_4)L(s,\pi_v\otimes \sigma_v,\wedge^2 \otimes \mathrm{std}_2)}{L(2s, \sigma_v, {\rm Sym}^2\otimes \omega_{\pi_v} )}, 
    \end{align*}
    thus 
     \[ J(s,z,\varphi,f_s, \eta) =L^S(z+ \tfrac{1}{2},\pi,{\rm std}_4)L^S(s,\pi \otimes \sigma,\wedge^2 \otimes \mathrm{std}_2) \prod_{v \in S} J(s,z,\varphi_v,f_{s,v},\eta_v),\]
     as desired.
\end{proof}

\begin{lemma}\label{J_ram_lemma} Suppose that $v \in S$, then, as a function of $z$, $$J^\star(s,z,\varphi_v,f_{s,v},\eta_v) :=L(z + \tfrac{1}{2}, \pi_v,{\rm std}_4)^{-1} J(s,z,\varphi_v,f_{s,v},\eta_v)$$ has holomorphic continuation to all $\C$ and there exists a constant $\epsilon > 0$ such that it converges absolutely for ${\rm Re}(s) > 1/2 - \epsilon$. 
\end{lemma}
\begin{proof}
When $v \in S$, let $\Omega$ be an open compact subgroup of $\GL_4(F_v)$ such that $\eta_v(k_v g_v) =\eta_v(g_v)$ for all $k_v \in \Omega$ and fix a basis $\{ W_{v,1}, \dots, W_{v,m}\}$ of the $\Omega$-invariant vectors in the Whittaker model of $\pi_v$. The results \cite[(1.5) \& (2.1)]{Watanabe:Global:Theta} give (see also p. 708 of \emph{loc.cit.}) 
$$W_{\varphi_v,\eta_v}^z(g_v) = L(z + \tfrac{1}{2}, \pi_v,{\rm std}_4)   | {\rm det}(g_v)|_v^{z}\sum_j \Xi_{v,j}(z)W_{v,j}(g_v),$$ where  $ \Xi_{v,j}(z)$ are entire functions on $z$. Hence, $$ J(s,z,\varphi_v,f_{s,v},\eta_v) =  L(z + \tfrac{1}{2}, \pi_v,{\rm std}_4) \sum_j \Xi_{v,j}(z) Z(s, W_{v,j}, W_{\sigma_v}),$$
where $Z(s, W_{v,j}, W_{\sigma_v,s})$ denotes the integral of \S \ref{ramified_integrals}. The result follows from this equality, together with Proposition \ref{non:vanishing:and:abs:conv:zeta:bad}(i). 
\end{proof}

\begin{corollary}\label{Corollary:Multi:Variable:Integral}
Suppose that  $\omega_\pi  \omega_\sigma =1$. Given a cusp form $\varphi = \otimes_v \varphi_v\in \pi$, a section $f =  f_{v} \in I_{P_{(1,2)}}(1/2,\sigma,\omega_{\pi})$, and a Schwartz function $\eta = \otimes_v \eta_v \in \mathcal{S}(M_4(\A))$, we have 
    \[ J(1/2, 0,\varphi, f, \eta) =L( \tfrac{1}{2},\pi,{\rm std}_4)L^S(\tfrac{1}{2},\pi \otimes \sigma,\wedge^2 \otimes \mathrm{std}_2) \prod_{v \in S} J^\star(1/2,0,\varphi_v,f_{v},\eta_v).\]
\end{corollary}
\begin{proof}
    Choose a standard holomorphic section $f_s \in I_{P_{(1,2)}}(s,\sigma, \omega_\pi)$ such that $f_{1/2}=f$. Thanks to Theorem \ref{THM_two_var_integral}, we have that, for large enough ${\rm Re}(s)$ and ${\rm Re}(z)$, 
    \begin{align}\label{eq_goingto01over2}
        J(s,z,\varphi,f_s, \eta) = L(z+ \tfrac{1}{2},\pi,{\rm std}_4) L^S(s,\pi \otimes \sigma,\wedge^2 \otimes \mathrm{std}_2) \prod_{v \in S} J^\star(s,z,\varphi_v,f_{s,v},\eta_v).
    \end{align} 
     The Klingen Eisenstein series $E^*_{P_{(1,2)}}(g,s, f_s)$ is entire when $\omega_{\pi}\omega_{\sigma}= 1$ (see \S\ref{Section:Klingen:Eisenstein}), thus $J(s,z,\varphi,f_s, \eta)$ has holomorphic continuation in the variable $s$ to all $\C$. Moreover, by Lemma \ref{J_ram_lemma} and the holomorphicity of the standard $L$-function of $\pi$, we deduce that both sides have holomorphic continuation in the variable $z$. By Lemma \ref{J_ram_lemma}, we also know that the ramified zeta factor $J^\star(s,z,\varphi_v,f_{s,v},\eta_v)$ converges absolutely whenever ${\rm Re}(s) > 1/2 - \epsilon$, for some $\epsilon >0$. Thus, we obtain the desired formula by evaluating \eqref{eq_goingto01over2} at $s = 1/2$ and $z = 0$.
\end{proof}

\subsection{The interplay between two seesaw identities}\label{Two:SeeSaw:Section}
 
 Let $\pi$, resp. $\sigma$, be cuspidal automorphic representations of 
$\GL_4(\A)$, resp. $\GL_2(\A)$, such that 
$\omega_\pi  \omega_\sigma =1$.
As in \S \ref{subsec_non_vanishing_central_L_value_degree_12}, we use Garrett's pullback formula (Corollary \ref{CorPullbackZeta}) and the Siegel-Weil formula of Theorem \ref{Siegel-Weilnew} to deduce that, for ${\rm Re}(z)>>0$, we have 

\small 
\begin{align}
    J(\tfrac{1}{2},z,\varphi,f^\Psi_{\phi_{1/2}}, \eta) &= \int_{Z_{\GSp_6}(\A)\backslash [\GL_2\boxtimes \GSp_4]} E_{P_{(3,0)}}^*(\iota( g_1  , g_2), \tfrac{1}{2},\phi_{\frac{1}{2}}) \Psi^w( g_1  ) \varphi^z_\eta(g_2) |\nu(g_2)|_\A^{-2z}\, dg_1\,dg_2
    \nonumber\\ &= 2 \,\zeta^S_F(2)^2  \sum_V \int_{Z_{\GSp_6}(\A)\backslash [\GL_2\boxtimes \GSp_4]} \int_{[\mathbf{O}_V]}\!\!\!\!\! \theta(h h_1, \iota( g_1  , g_2) ; f^V) \Psi^w( g_1 ) \varphi^z_\eta(g_2) |\nu(g_2)|_\A^{-2z}  \, d h \, dg_1\,dg_2 \nonumber\\ 
    &= \zeta^S_F(2)^2  \sum_V \int_{Z_{\GSp_6}(\A)\backslash [\GL_2\boxtimes \GSp_4]} \int_{[\SO_V]}\!\!\!\!\! \theta(h h_1, \iota( g_1  , g_2) ; f^V) \Psi^w( g_1 ) \varphi^z_\eta(g_2) |\nu(g_2)|_\A^{-2z}  \, d h \, dg_1\,dg_2, \label{after_GL4theta_SW1}
\end{align}
\normalsize
where the latter equality follows from Lemma \ref{Reduction:To:SO}. Recall also that $V$ ranges through the (finite) set of quaternionic quadratic spaces of dimension $4$ over $F$ for which the image of $\phi_{1/2}$ via the projection ${\rm pr}_{\Pi(V)}: I_{P_{(3,0)}}(1/2, \mathbf{1}) \to  \Pi(V)$ is non-zero, and $f^V \in \mathcal{S}(V(\A)^3)$ is such that ${\rm pr}_{\Pi(V)} \phi_{1/2}  = \lambda_V(f^V)^{\sim}$.

\noindent Now consider the factor of the sum \eqref{after_GL4theta_SW1} corresponding to $f= f^V \in \mathcal{S}(V(\A)^3)$. Recall the seesaw pair \begin{align*}
\xymatrix{
    \mathbf{GSO}_V \boxtimes \mathbf{GSO}_V \ar@{-}[dd] \ar@{-}[ddr]  & \GSp_6 \ar@{-}[ddl] \ar@{-}[dd] \\ \\ 
   \mathbf{GSO}_V  & \GL_2\boxtimes \GSp_4,
} 
\end{align*}
where the vertical left line denotes the diagonal embedding $\mathbf{GSO}_V \hookrightarrow \mathbf{GSO}_V \boxtimes \mathbf{GSO}_V$. 
\noindent To obtain the corresponding seesaw identity, we first record the following Lemma.

\begin{lemma}\label{Factorization:Regularization}
For $f = f_1\otimes f_2\in \mathcal{S}(V(\A))\otimes \mathcal{S}(V(\A)^2)\subset \mathcal{S}(V(\A)^3)$, we have the following equality:
\begin{align}\label{eq_regularized_thetas_seesaw1}
    \int_{[\mathbf{SO}_V]}\theta(hh_1, \iota( g_1  , g_2) ; f)dh = \int_{[\mathbf{SO}_V]}\theta(hh_1 , g_1; f_1)\theta(hh_1, g_2 ; f_2)dh,
\end{align}
where each theta function is regularized as in \cite{IchinoSW} when $V$ is split.
\end{lemma}
\begin{proof}
The result is straightforward when $V$ is non-split; when $V$ is split, to obtain the desired equality one has to check that each side of \eqref{eq_regularized_thetas_seesaw1} can be compatibly regularized. This follows from the results of \cite{IchinoSW}, as we now explain.
    
    Let $V$ be split and let $\omega(z_2)$, $\omega(z_4)$ and $\omega(z_6)$ be the regularization operators defined in \cite[Lemma 1.3]{IchinoSW} for the dual reductive pairs $(\mathbf{O}_V,\Sp_2)$, $(\mathbf{O}_V,\Sp_4)$ and $(\mathbf{O}_V,\Sp_6)$, respectively. We then let $c_2,c_4,c_6$ be the constants of \cite[Lemma 1.6]{IchinoSW} attached to $z_2$, $z_4$ and $z_6$, respectively. We define the following two functionals: 
      \begin{align*}
    I_1 &: \mathcal{S}(V(\A))\otimes \mathcal{S}(V(\A)^2)\to \C,\,\, f\otimes f'\mapsto (c_2c_4)^{-1}\int_{[\mathbf{SO}_V]}\theta(h,I_2,\omega(z_2)f)\theta(h,I_4,\omega(z_4)f')dh,\\
    I_2 &: \mathcal{S}(V(\A)^3)\to \C,\,\, f''  \mapsto c_{6}^{-1}\int_{[\mathbf{SO}_V]}\theta(h,I_6,\omega(z_6)f'')dh.\end{align*} 
    By the proof of \cite[Proposition 1.5]{IchinoSW}, the theta functions \[\theta(h,I_2,\omega(z_2)f),\;\theta(h,I_4,\omega(z_4)f'),\;\theta(h,I_6,\omega(z_6)f'')\] are rapidly decreasing. Thus, the functionals $I_1$ and $I_2$ are well defined for any Schwartz function $f, f'$ and $f''$. The result follows from \cite[Lemma 1.9]{IchinoSW}, once we show that the functional $I_1$ extends to a functional $\tilde{I}_1$ on $\mathcal{S}(V(\A)^3)$ and that $\tilde{I}_1$ equals $I_2$.
 Firstly, using the fact that $\mathcal{S}(V(\A))\otimes \mathcal{S}(V(\A)^2)$ is dense in $\mathcal{S}(V(\A)^3)$, the functional $I_1$ extends to an $\mathbf{SO}_V(\A)$-invariant functional $\tilde{I}_1 : \mathcal{S}(V(\A)^3)\to \C$. Moreover, using again the density of  $\mathcal{S}(V(\A))\otimes \mathcal{S}(V(\A)^2)$ in $\mathcal{S}(V(\A)^3)$,  $$\tilde{I}_1(f) = I_2(f) = \int_{[\mathbf{SO}_V]}\theta(h,I_6,f)dh,$$ for any Schwartz function $f \in \mathcal{S}(V(\A)^3)$ such that the theta integral is absolutely convergent.  This and \cite[Lemma 1.9]{IchinoSW} (which applies to these $\mathbf{SO}_V(\A)$-invariant functionals because of Lemma \ref{Reduction:To:SO}) conclude the proof of the Lemma.
\end{proof}

Suppose that $f = f_1\otimes f_2\in \mathcal{S}(V(\A))\otimes \mathcal{S}(V(\A)^2)\subset \mathcal{S}(V(\A)^3)$. By Lemma \ref{Factorization:Regularization} and following verbatim the argument of \cite[Lemma 5.1]{KudlaHarrisJacquet}, we get the seesaw identity
\begin{align}
    &\int_{Z_{\GSp_6}(\A)\backslash [\GL_2\boxtimes \GSp_4]} \int_{[\SO_V]}\!\!\!\!\! \theta(h h_1, \iota( g_1  , g_2) ; f) \Psi^w( g_1 ) \varphi^z_\eta(g_2) |\nu(g_2)|_\A^{-2z}  \, d h \, dg_1\,dg_2 \nonumber\\ &= \int_{Z_{\GSO_V}(\A)\backslash [\GSO_V]} \int_{[\SL_2 \times \Sp_4]}\!\!\!\!\! \theta(h , g_1 g'_1; f_1)\theta(h , g_2 g'_2; f_2) \Psi^w( g_1g'_1 ) \varphi^z_\eta(g_2 g'_2) |\nu'(h)|_\A^{-2z}   \, dg_1\,dg_2 \, d h\nonumber\\ &= \int_{Z_{\GSO_V}(\A)\backslash [\GSO_V]} \int_{[\SL_2 \times \Sp_4]}\!\!\!\!\! \theta(h , g_1 g'_1; f_1)\theta(h , g_2 g'_2; f_2) \Psi^w( g_1g'_1 ) |\nu'(h)|_\A^{-2z} \cdot \nonumber\\ &\cdot \int_{[\GL_4]}\varphi(r)|\det(r)|^{z}_{\A}\sum_{\substack{x\in M_4(F)\\x\neq 0}}\omega_{4,4}(g_2g_2',r)\eta(x)\,dr  \, dg_1\,dg_2\, d h,\label{intermediateseesawintegral}
\end{align}
where $g'_1\in \GL_2$ and $g_2'\in \GSp_4$ are chosen so that $\nu(g_1') = \nu(g_2') = \nu'(h)$ and where the last equality follows from \eqref{Theta:Lift:Expression}.
 
The following Proposition and Lemma concern a second seesaw pair, which involves the theta correspondence for $(\GL_4,\GL_4)$ and the symplectic-orthogonal theta correspondences. Namely, if we let $V' := V \oplus \mathbb{H}^4$, the discussion below can be summarized by saying that there is a seesaw pair \begin{align*}
\xymatrix{
    \mathbf{GSO}_{V'} \ar@{-}[dd] \ar@{-}[ddr]  & \GSp_4 \times  \GL_4  \ar@{-}[ddl] \ar@{-}[dd] \\ \\ 
   \mathbf{GSO}_V \times \GL_4  &  \GSp_4,
} 
\end{align*}
where the right vertical line corresponds to the diagonal embedding $\GSp_4 \hookrightarrow \GSp_4 \times  \GL_4 $, while the left vertical line is given by 
\begin{align*}
    j:\mathbf{GSO}_V\times \GL_4&\to \mathbf{GSO}_{V'},\\
    (h,r)&\mapsto \left(\begin{smallmatrix}r& &\\ &h& \\ & &\nu'(h)I_4'\;^{t}r^{-1}I_4'\end{smallmatrix}\right).
\end{align*}

\begin{proposition}\label{Product:Theta:Functions}
    Let $\eta \in \mathcal{S}(M_4(\A))$ be a Schwartz function with support outside $0$. There is a $\mathbf{GSO}_V(\A)\times \GL_4(\A)$-equivariant isomorphism $\mathcal{F}:\mathcal{S}(V(\A)^2)\otimes \mathcal{S}(M_4(\A))\to \mathcal{S}(V'(\A)^2)$, such that, for any $g_2 \in \Sp_4(\A)$,   
    \[\theta(h,g_2g_2';f_2)\cdot \!\!\sum_{\substack{x\in M_4(F)\\x\neq 0}}\omega_{4,4}(g_2g_2',r)\eta(x) = \theta(j(h,r), g_2g_2', \mathcal{F}(f_2\otimes \eta)), \]
    where the theta function appearing on the right hand side is the theta function with respect to the dual reductive pair $(\mathbf{O}_{V'},\Sp_4)$ and polarization $\left(V'\otimes W_4^+\right)\oplus\left(V' \otimes W_4^-\right)$. 
\end{proposition}

\begin{proof}
     The factorization \begin{equation}\label{Polarization:RHS:Lemma}\left(V\otimes W_4^{\vee,-}\oplus  (\mathbb{H}^{\vee,+})^4 \otimes W_4^\vee \right)\oplus\left(V\otimes W_4^{\vee,+} \oplus (\mathbb{H}^{\vee,-})^4 \otimes W_4^\vee\right),\end{equation}
    defines a polarization of $(V\oplus \mathbb{H}^{\vee,4})\otimes W_4^\vee$. By \cite[Chapitre 2, I.7]{Moeglin:Vigneras:Waldspurger}, there exists an intertwining operator \[\mathcal{F}: \mathcal{S}(V(\A)^2)\otimes \mathcal{S}(W_4^\vee(\A)^4)\to \mathcal{S}(V'(\A)^2),\]
    given by the partial Fourier transform that relates the polarization \eqref{Polarization:RHS:Lemma} with the polarization $$\left(V'\otimes W_4^+\right)\oplus\left(V' \otimes W_4^-\right).$$ By Poisson summation formula, we then have that, for $f_2 \in \mathcal{S}(V(\A)^2)$ and $\eta \in \mathcal{S}(W_4^\vee(\A)^4)$,    \begin{equation}\label{Mixed:Model:Lemma:Comparison}\theta(j(h,r), g_2, \mathcal{F}(f_2\otimes \eta)) = \sum_{y\in V\otimes W_4^{\vee,-}}\sum_{x\in (\mathbb{H}^{\vee,+})^4 \otimes W_4^\vee}\omega(j(h,r),g_2) (f_2 \otimes \eta)(y,x),\end{equation}
    for any $g_2 \in \GSp_4(\A)$ and $(h,r) \in \GSO_V(\A) \times \GL_4(\A)$ such that $\nu(g_2) = \nu'(h)$. We now compare $\theta(j(h,r), g_2, \mathcal{F}(f_2\otimes \eta)) $ with
    \begin{align}\theta(h,g_2;f_2)\cdot \!\!\sum_{\substack{x\in M_4(F)\\x\neq 0}}\omega_{4,4}(g_2,r)\eta(x)  &= | \nu'(h)|_\A^{-4} |{\rm det}(r)|_\A^2\sum_{y\in V\otimes W_4^+}  \sum_{\substack{x\in M_4(F)\\x\neq 0}}  \omega(h,g_2) f_2(y)  \eta(g_2^{-1}xr).  \label{Lemma:Comparison:2} 
    \end{align} 
 We fix an isomorphism\footnote{Our choice is in line with the one of Kudla for a similar seesaw dual pair described in \cite[(2.22) \& (2.23)]{SeeSawKudla}.} $$M_{4}(F) \simeq  (\mathbb{H}^{\vee,+})^4 \otimes W_4^\vee,$$  such that if $M \mapsto w \otimes v$, with $w \in (\mathbb{H}^{\vee,+})^4$ and $v \in W_4^\vee$, then $$g^{-1} M r \mapsto (r,g) \cdot  (w \, \otimes \, v)  :=(I_4'{}^tr^{-1}I_4')^{-1}  w \, \otimes \, v  g .$$    
Via this identification, we can write \eqref{Lemma:Comparison:2} as 
\[ | \nu'(h)|_\A^{-4} |{\rm det}(r)|_\A^2\sum_{y\in V\otimes W_4^+}  \sum_{\substack{x\in (\mathbb{H}^{+})^4 \otimes W_4^\vee \\x\neq 0}}  \omega(h,g_2) f_2(y) \eta\left((r,g_2)\cdot   x\right). \] 
The symplectic pairing gives a $\Sp_4(F)$-equivariant isomorphism $W_4 \simeq W_4^\vee$ which induces $W_4^+ \simeq W_4^{\vee,-}$. Using this, if we let $\eta \in \mathcal{S}(M_4(\A))$ be a Schwartz function with support outside $0$, the sums in  
    \eqref{Mixed:Model:Lemma:Comparison} and \eqref{Lemma:Comparison:2} agree up to verifying that the actions of $h$, $r$, and $g_2$ coincide. Recall that the action of $\omega(j(h,r),g_2)$ in  
    \eqref{Mixed:Model:Lemma:Comparison} depends on the polarization \eqref{Polarization:RHS:Lemma}. Since the group $\GL_4$ embeds via $j$ into the Levi component of the maximal parabolic of $\mathbf{GSO}_{V'}$ that stabilizes $(\mathbb{H}^{+})^4$, we deduce that it acts trivially on $V$ and
    \[\omega(j(1,r),1) \eta(x) = |\mathrm{det}(r)|_{\A}^{2}\eta((r,1) \cdot x ).\] Moreover, $h \in \mathbf{GSO}_V$ acts by left translation by $h^{-1}$ on $f_2$ and by multiplication by $|\nu'(h)|_{\A}^{-6}$ (because of \eqref{eq_normalization_Weil_similitudes}). Similarly, for the theta function $\theta(h,g_2;f_2)$ for the pair $(\mathbf{GSO}_V,\GSp_4)$, $h \in \mathbf{GSO}_V(\A)$ acts by left translation by $h^{-1}$ together with multiplication by $|\nu'(h)|_{\A}^{-2}$. These two actions coincide up to $|\nu'(h)|_{\A}^{-4}$, which, indeed,  appears as a multiplying factor in \eqref{Lemma:Comparison:2}. The comparison for the action of $\GSp_4(\A)$ follows in a similar way. 
\end{proof}

\begin{lemma}\label{Regularization:SW:Sp}
    When $V\simeq \mathbb{H}^2$, for any $f_2\otimes \eta\in \mathcal{S}(V(\A)^2)\otimes \mathcal{S}(M_4(\A))$, the function $\theta(j(h,r),g_2g_2',\mathcal{F}(f_2\otimes \eta))$ is absolutely convergent in the variable $g_2$. Furthermore,
    \[\int_{[\Sp_4]}\theta(j(h,r),g_2g_2',\mathcal{F}(\omega(z_4)f_2\otimes \eta))dg_2 = \int_{[\Sp_4]}\theta(j(h,r),g_2g_2',\tilde{\omega}(z_4)\mathcal{F}(f_2\otimes \eta))dg_2,\] 
    where $\tilde{\omega}(z_4)$ is the regularized operator for the dual reductive pair $(\mathbf{O}_{V'},\Sp_4)$ defined in \cite{IchinoSW}.
\end{lemma}
\begin{proof}
    We begin by showing that $\theta(j(h,r),g_2g_2',\mathcal{F}(\omega(z_4)f_2\otimes \eta))$ is rapidly decreasing in $g_2\in \Sp_4(\A)$. 
   Firstly, we prove that  $\theta(h,g_2g_2';\omega(z_4)f_2)$ is rapidly decreasing in $g_2$, by mimicking the strategy used in  \cite[Proposition 1.5]{IchinoSW} and \cite[Proposition 5.3.1]{KudlaRallis}. By the proof of \cite[Proposition 1.5]{IchinoSW}, 
\begin{equation}\label{FirstVanishingConvergenceSp}\omega(z_4)\hat{f_2}(w) = 0,\textit{ for all }w\in W_4(\A)^2\textit{ with }\mathrm{rank}(w)<2,\end{equation}
 where $\hat{f_2}\in \mathcal{S}(W_4(\A)^2)$ is the partial Fourier transform of $f_2$ constructed in \cite[(5.2.4)]{KudlaRallis}; as shown in \cite[Proposition 1.5]{IchinoSW}, this fact implies that the theta function $\theta(h,g_2g_2';\omega(z_4)f_2)$ is rapidly decreasing in $h\in \mathbf{O}_{\mathbb{H}^2}(\A)$. 
 Then, identifying 
 $V(\A)^2 \simeq W_4(\A)^2$ as vector spaces,  \eqref{FirstVanishingConvergenceSp} implies that
   \[\omega(z_4)f_2(w') = 0,\textit{ for all }w'\in V(\A)^2\textit{ with }\mathrm{rank}(w')<2,\]
   thus, following again the proof of \cite[Proposition 1.5]{IchinoSW}, one has that $\theta(h,g_2g_2';\omega(z_4)f_2)$ is rapidly decreasing in $g_2$, proving our claim.  Since the theta function \[\sum_{\substack{x\in M_4(F)\\x\neq 0}}\omega_{4,4}(g_2g_2',r)\eta(x)\]
   has moderate growth in $g_2$, Proposition \ref{Product:Theta:Functions} and the rapid decay of $\theta(h,g_2g_2';\omega(z_4)f_2)$ in $g_2$ imply that \[\theta(j(h,r),g_2g_2',\mathcal{F}(\omega(z_4)f_2\otimes \eta))\]
 is rapidly decreasing in $g_2$. 
The rapid decay of these functions implies that the map 
   \begin{align*}
        \mathcal{S}(V(\A)^2)\otimes \mathcal{S}(M_4(\A))&\to \C,\\
       f\otimes f'&\mapsto \int_{[\Sp_4]}\theta(I_{12},g_2,\mathcal{F}(\omega(z_4)f\otimes f'))dg_2,
   \end{align*}
   is well-defined. The desired equality is then proved as in Lemma \ref{Factorization:Regularization}.
 \end{proof}

Proposition \ref{Product:Theta:Functions} and Lemma \ref{Regularization:SW:Sp} show that \eqref{intermediateseesawintegral} equals 
\begin{align*}
    \int_{Z_{\GSO_{V'}}(\A)\backslash [\mathbf{GSO}_V\times \GL_4]}   &|\nu'(h)|_\A^{-2z}|\det(r)|^{z}_{\A} \varphi(r) \cdot \\ &\left(\int_{[\SL_2]} \theta(h, g_1g_1',f_1) \Psi^w( g_1g_1' )dg_1\right) \left(\int_{[\Sp_4]} \theta(j(h,r), g_2g_2', \mathcal{F}(f_2\otimes \eta)) d g_2\right)dh\, d r,
\end{align*}
where we have omitted the regularization operators to simplify notation. 
\begin{remark}
    The inner integral over $[\SL_2]$ is exactly the one studied by Shimizu \cite{Shimizu:JL} and Waldspurger \cite{WaldsToricPeriods}. In particular, when $V$ is given by a non-split quaternion algebra $D$ over $F$, if $\sigma$ does not lie in the image of the Jacquet--Langlands correspondence from $D^\times$, then the integral above is zero.
\end{remark}
\noindent From now on, we restrict our attention to the case of the totally split quadratic space $\mathbb{H}^2$ and assume the following:
\begin{itemize}
    \item[(\mylabel{Spl}{\textbf{Spl}})] $\phi_{1/2} \in  \Pi(\mathbb{H}^2)$, \textit{i.e.} there exists $f^{\mathbb{H}^2} \in \mathcal{S}(\mathbb{H}^2(\A)^3)$ such that $ \lambda_{\mathbb{H}^2}(f^{\mathbb{H}^2})^{\sim} = \phi_{1/2}.$
\end{itemize} 
 
\begin{proposition}\label{FinalPropSct83}
Assume \eqref{Spl} and let $\eta \in \mathcal{S}(M_4(\A))$ be a Schwartz function with support outside $0$. If $ \lambda_{\mathbb{H}^2}(f_1 \otimes f_2)^{\sim} = \phi_{1/2},$ for $ f_1\in \mathcal{S}(\mathbb{H}^2(\A))$ and $f_2 \in \mathcal{S}(\mathbb{H}^2(\A)^2)$, we have, for ${\rm Re}(z)>>0$, the equality
\begin{align*}
    J(\tfrac{1}{2},z,\varphi,f^\Psi_{\phi_{1/2}}, \eta) &= \zeta^S_F(2)^2  \int_{Z_{\GSO_{\mathbb{H}^6}}(\A)\backslash [\mathbf{GSO}_{\mathbb{H}^2}\times \GL_4]}   \!\!\!\!\!\!\!\!|\nu'(h)|_\A^{-2z}|\det(r)|^{z}_{\A} \varphi(r) \cdot \\ & \left(\int_{[\SL_2]} \theta(h, g_1g_1',f_1) \Psi^w( g_1g_1' )dg_1\right) \left(\int_{[\Sp_4]} \theta(j(h,r), g_2g_2', \mathcal{F}(f_2\otimes \eta)) d g_2\right)dh\, d r. 
\end{align*}
 
\end{proposition}
\begin{proof}
   Thanks to \eqref{Spl}, \eqref{after_GL4theta_SW1} reads as \small
\begin{align*}
  J(\tfrac{1}{2},z,\varphi,f^\Psi_{\phi_{1/2}}, \eta)=\zeta^S_F(2)^2  \int_{Z_{\GSp_6}(\A)\backslash [\GL_2\boxtimes \GSp_4]} \int_{[\SO_{\mathbb{H}^2}]}\!\!\!\!\! \theta(h h_1, \iota( g_1  , g_2) ; f^{\mathbb{H}^2}) \Psi^w( g_1 ) \varphi^z_\eta(g_2) |\nu(g_2)|_\A^{-2z}  \, d h \, dg_1\,dg_2,
\end{align*}
\normalsize
with $f^{\mathbb{H}^2} = f_1 \otimes f_2$. The result then follows from applying the two seesaw identities above.
\end{proof}

\begin{remark}\leavevmode
    \begin{enumerate}
    \item If $\sigma$ and $\pi$ are unramified everywhere, then \eqref{Spl} holds.
    \item If $\sigma$ does not lie in the image of the Jacquet--Langlands correspondence from any non-split quaternion algebra, then one can show that $J(\tfrac{1}{2},z,\varphi,f^\Psi_{\phi_{1/2}}, \eta)$ equals $J(\tfrac{1}{2},z,\varphi,f^\Psi_{\phi_{1/2}'}, \eta)$, with $\phi_{1/2}'$ satisfying \eqref{Spl}.
        \end{enumerate}
        \end{remark}
\begin{remark}\label{remark_on_spl} We impose the hypothesis \eqref{Spl} for the following reason. In order to continue with the computation we relate the integral of the theta function $\theta(-, -, \mathcal{F}(f_2\otimes \eta))$ over $\Sp_4$ to an orthogonal Eisenstein series on $\GSO_{\mathbb{H}^6}$ via a Siegel--Weil type formula and proceed as in \S \ref{Section:Period:Central:Value}. 
When $V\not\simeq \mathbb{H}^2$, this Siegel--Weil type formula does not seem to be available in the literature. Despite this, such a formula should not be essential. Instead, one should be able to replace its use by an explicit relation between the integral over $\GL_4$ of the cusp form $\varphi$ with the theta function $\theta(-, -, \mathcal{F}(f_2\otimes \eta))$ for the pair $(\GSp_4,\GSO_{V'})$ and the Jacquet--Langlands lift of $\varphi$ to the inner form of $\GL_4$ associated with $V$. We hope to return to this problem in future work. 
\end{remark}

\subsection{A Siegel-Weil formula for $(\GSp_4, \GSO_{\mathbb{H}^6})$ in the first term range}

We recall a Siegel-Weil formula proved by Yamana in \cite[Theorem 2.2 (i)]{YamanaSingular}. 

Recall that $Q_{(6,0)} = M_{(6,0)}N_{(6,0)}$ denotes the Siegel parabolic subgroup of $\GSO_{\mathbb{H}^6}$, with $ M_{(6,0)} \simeq \GL_6 \times \GL_1$. For any $q = m n \in Q_{(6,0)}$, we have $$\delta_{Q_{(6,0)}}(q)=\frac{|\det(A)|^5}{|\nu'|^{15}},\text{ where $m = (A , \nu') \in \GL_6 \times \GL_1$.}$$  We consider the normalized induced representation $I_{Q_{(6,0)}}(s,\mathbf{1}) := \mathrm{Ind}_{Q_{(6,0)}(\A)}^{\GSO_{\mathbb{H}^6}(\A)}(\delta_{Q_{(6,0)}}^{\frac{s}{5}})$. Given a standard holomorphic section $\Phi_{s}\in I_{Q_{(6,0)}}(s,\mathbf{1}) $, define the Eisenstein series \[E_{Q_{(6,0)}}(\tilde{h},\Phi_{s}) := \sum_{\gamma\in Q_{(6,0)}(F)\setminus \GSO_{\mathbb{H}^6}(F)}\Phi_{s}(\gamma \tilde{h}).\]
It converges absolutely for $\mathrm{Re}(s)>>0$ and admits a meromorphic continuation to the whole complex plane. According to \cite[Theorem 2.1]{YamanaSingular}, it is holomorphic at $s = -\tfrac{1}{2}$. For the following result, denote by $\mathcal{S}(\mathbb{H}^6(\A)^2)\to \mathcal{S}(W_4(\A)^6)$, $f\mapsto \hat{f}$, the partial Fourier transform associated with the change of polarization. 
\begin{theorem}\label{SW:Sp}
    For every $\tilde{h} \in \GSO_{\mathbb{H}^6}(\A)$ and $g_2' \in \GSp_4(\A)$ such that $\nu'(\tilde{h}) = \nu(g_2')$, we have
    \[ \int_{[\Sp_4]}\theta(\tilde{h},g_2g_2',\mathcal{F}(f_2\otimes \eta))dg_2 = \frac{1}{2}\cdot E_{Q_{(6,0)}}\left(\tilde{h},\Phi_{-\frac{1}{2}}\right),\]
    where, if $\tilde{h} = mnk \in L_{(6,0)}(\A)N_{(6,0)}(\A)K_{\GSO_{\mathbb{H}^6}(\A)}$ with $m = (A,\nu')$, we let $$\Phi_{s}(\tilde{h}) := \omega(\tilde{h},g_2')\widehat{ \mathcal{F}}(f_2\otimes \eta)(0)\cdot \left|\tfrac{\mathrm{det}(A)}{\nu'}\right|^{1/2+s},$$
    with $\widehat{ \mathcal{F}}(f_2\otimes \eta)$ denoting the Fourier transform of $\mathcal{F}(f_2\otimes \eta)$.
\end{theorem}
\begin{proof}
This is \cite[Theorem 2.2 (i)]{YamanaSingular} in the case $3$ with $m = 4$, $r = 2$, and $n = 6$ in the notation of \textit{loc.cit.}.
\end{proof}

\subsection{The Shimizu--Jacquet--Langlands correspondence}

Let $D$ be a quaternion algebra defined over $F$, which, equipped with the reduced norm $\rm{Nm}_D$, is a quadratic space of dimension $4$. Recall from \S\ref{Exceptional:Isom:Section} that we have an exceptional isomorphism
\begin{align*}
    \GSO_{D}&\to D^{\times}\times D^{\times}/\{(z,z),\;z\in Z_{D^{\times}}\},\label{ExceptIsom}\\
    h&\mapsto (h_1,h_2),
\end{align*}
This identifies any automorphic representation of $\GSO_{D}$ with $\Pi_1\boxtimes\Pi_2$, where  $\Pi_1$ and $\Pi_2$ are automorphic representations of $D^{\times}$ satisfying $\omega_{\Pi_1} = \omega_{\Pi_2}^{-1}$. Accordingly, if $\Psi_1$ and $\Psi_2$ are forms in the space of $\Pi_1$ and $\Pi_2$, we write \[(\Psi_1\boxtimes \Psi_2)(h) := \Psi_1(h_1)\Psi_2(h_2),\]
where $h$ corresponds to $(h_1,h_2)$ under the above isomorphism.

\begin{theorem}[{\cite[Th\'eor\`eme 1]{WaldsToricPeriods}, \cite[Theorem 2]{Shimizu:JL}}]\label{SJL}
    Let $\sigma$ be any cuspidal automorphic representation of $\GL_2$. For any cusp form $\Psi\in \sigma$ and any $f\in \mathcal{S}(D(\A))$, we have
    \[\int_{[\SL_2]} \theta(h, g_1g_1',f) \Psi( g_1g_1' )dg_1 = (\overline{\Psi}'\boxtimes\Psi'')(h),\]
    where $\Psi',\Psi''$ are forms in $ \sigma^D$, with $\sigma^D$ denoting the inverse of the Jacquet--Langlands lift of $\sigma$. If $\sigma$ is not in the image of the Jacquet--Langlands correspondence, the above integral is zero. 
\end{theorem}
\begin{remark}
   Recall that, if $D \simeq \mathbb{H}^2$, we have $\sigma^D = \sigma$. 
\end{remark}

For the next Proposition, we consider the quotient $\GL_4\times \GL_1$ by $\GL_1$ via the embedding 
\begin{align}
    \label{Embedding:GL1}\GL_1&\to \GL_4\times \GL_1,\\
    a&\mapsto (aI_4,a^2).\nonumber
\end{align}
\begin{proposition}\label{Final:Proposition:SeeSaw}
   Assume \eqref{Spl} with $ \lambda_{\mathbb{H}^2}(f_1 \otimes f_2)^{\sim} = \phi_{1/2},$ for $ f_1\in \mathcal{S}(\mathbb{H}^2(\A))$ and $f_2 \in \mathcal{S}(\mathbb{H}^2(\A)^2)$. Let $\eta \in \mathcal{S}(M_4(\A))$ have support outside $0$. Then, for ${\rm Re}(z)>>0$, we have that $J(\tfrac{1}{2},z,\varphi,f^\Psi_{\phi_{1/2}}, \eta)$ is equal to
\[\frac{\zeta^S_F(2)^2}{4}\int_{[ \GL_1 \backslash (\GL_4 \times \GL_1)]}\int_{[ \mathbf{SO}_{\mathbb{H}^2}]}|a|_\A^{-2z}|\det(r)|^{z}_{\A}\varphi(r)(\overline{\Psi}'\boxtimes\Psi'')(h t_a )\mathrm{E}_{Q_{(6,0)}}(j(h t_a ,r),\Phi_{-\frac{1}{2}})dhdrda,\]
where $t_a := \left(\begin{smallmatrix}I_2& \\ &aI_2\end{smallmatrix}\right)$, and $\overline{\Psi}'\boxtimes\Psi''$ corresponds, via Theorem \ref{SJL}, to the pair $(\Psi^{w},f_1)\in \sigma\times\mathcal{S}(\mathbb{H}^2(\A))$. 
\end{proposition}
\begin{proof}
Using Theorems \ref{SJL} and \ref{SW:Sp}, we have that $J(\tfrac{1}{2},z,\varphi,f^\Psi_{\phi_{1/2}}, \eta)$ is equal to
\[\frac{\zeta^S_F(2)^2}{2} \int_{Z_{\GSO_{\mathbb{H}^6}}(\A)\backslash[\mathbf{GSO}_{\mathbb{H}^2}\times \GL_4]}|\nu'(h)|_\A^{-2z}|\det(r)|^{z}_{\A}\varphi(r)(\overline{\Psi}'\boxtimes\Psi'')(h)\mathrm{E}_{Q_{(6,0)}}(j(h,r),\Phi_{-\frac{1}{2}})dhdr.\]
Since $\GSO_{\mathbb{H}^2}\simeq \mathbf{SO}_{\mathbb{H}^2}\rtimes \GL_1$ and since $\mathrm{vol}([Z_{\mathbf{SO}_{\mathbb{H}^2}}]) = 2$, we perform an outer integration over $\GL_4 \times \GL_1$ to reduce the integral over $Z_{\GSO_{\mathbb{H}^6}}(\A)\backslash[\mathbf{GSO}_{\mathbb{H}^2}\times \GL_4]$ to the one of the statement. 
\end{proof}

\subsection{A pullback formula for orthogonal groups}\label{subsec:Pullback:Orthogonal}

In order to compute the integral in Proposition \ref{Final:Proposition:SeeSaw}, we perform a series of computations similar to the ones that appear in \S \ref{Section:Period:Central:Value}. As a first step, we prove a pullback formula of Eisenstein series for orthogonal groups.

Consider the embedding
\begin{align*}
    \kappa: \GSO_{\mathbb{H}^2}\boxtimes \GSO_{\mathbb{H}^4}&\to \GSO_{\mathbb{H}^6},\\
    h\times x&\mapsto \left(\begin{smallmatrix}x_1& &x_2\\ &h& \\x_3& &x_4\end{smallmatrix}\right),\nonumber
\end{align*}
where $x = \left(\begin{smallmatrix}x_1&x_2\\x_3&x_4\end{smallmatrix}\right)$. We now study the integral
\begin{align}\label{Main:Integral:Pullback:Orthogonal}\int_{[\mathbf{SO}_{\mathbb{H}^2}]}&(\overline{\Psi}'\boxtimes\Psi'')(h t_{\nu'(x)})\mathrm{E}_{Q_{(6,0)}}(\kappa(ht_{\nu'(x)},x),\Phi_{s})dh,\end{align}
and relate it to an Eisenstein series on $\mathbf{GSO}_{\mathbb{H}^4}$. 
\begin{proposition}\label{Double:Quotient:Pullback:Orthogonal}
    We have
    \[Q^\circ_{(6,0)}(F)\setminus \mathbf{SO}_{\mathbb{H}^6}(F)/\kappa(\mathbf{SO}_{\mathbb{H}^2}\times\mathbf{SO}_{\mathbb{H}^4})(F) = \{q_i\}_{i = 1,2,3}.\]
    Furthermore, the stabilizers are
    \begin{enumerate}
        \item $\mathrm{Stab}_{\mathbf{SO}_{\mathbb{H}^2}\times\mathbf{SO}_{\mathbb{H}^4}}(q_1) = Q^\circ_{(2,0)}(F)\times Q^\circ_{(4,0)}(F)$.
        \item $\mathrm{Stab}_{\mathbf{SO}_{\mathbb{H}^2}\times\mathbf{SO}_{\mathbb{H}^4}}(q_2) $ equals $$ \left\{\left(\begin{smallmatrix}m& & \\ &h& \\ & &m^{-1}\end{smallmatrix}\right)\times \left(\begin{smallmatrix} l& & \\ &h& \\ & &I_3'\;^{t}l^{-1}I_3'\end{smallmatrix}\right),\;l\in \GL_3(F),\; m\in F^\times,\; h\in \mathbf{SO}_{\mathbb{H}^2}(F)\right\} \cdot \left( N_{{(1,2)}}(F) \times N_{(3,2)}(F) \right).$$
        \item $\mathrm{Stab}_{\mathbf{SO}_{\mathbb{H}^2}\times\mathbf{SO}_{\mathbb{H}^4}}(q_3) = \left\{h\times\left(\begin{smallmatrix} l& & \\ &h& \\ & &I_2'\;^{t}l^{-1}I_2'\end{smallmatrix}\right),\;l\in \GL_2, h\in \mathbf{SO}_{\mathbb{H}^2}\right\} \cdot (\{I_4 \} \times  N_{(2,4)} )$.
    \end{enumerate}
\end{proposition}
\begin{proof}
    Note that $Q_{(6,0)}^{\circ}(F)\setminus \mathbf{SO}_{\mathbb{H}^6}(F)$ is isomorphic to the set of flags $$ \mathcal{V} := \{S\subset \mathbb{H}^6\,:\,\mathrm{dim}(S) = 6,\;S\;\text{isotropic}\}.$$ Let us factor $\mathbb{H}^6 = V_1\oplus V_2 := \mathbb{H}^2\oplus \mathbb{H}^4$ according to the embedding $\kappa$. Since all the quadratic spaces considered in this Proposition are hyperbolic planes, we can use \cite[\S 5]{Huang:He} to compute the desired double quotient. Given a subspace $S\subset \mathbb{H}^6$, we denote 
    \[\mathrm{rad}(S) := \{s\in S,\;s\in S^{\perp}\},\]
    and by $P_{V_i}(S) = (S\times V_j)\cap V_i$ the projection to the space $V_i$ with respect to the decomposition $\mathbb{H}^6 = V_1\oplus V_2 $. In \cite[\S 5]{Huang:He} it is shown that the $\kappa(\mathbf{SO}_{\mathbb{H}^2}\times \mathbf{SO}_{\mathbb{H}^4})(F)$-orbits on $\mathcal{V}$ are uniquely determined by the following tuples:
    \[\left(\mathrm{dim}(S\cap V_1),\mathrm{dim}(S\cap V_2),\mathrm{dim}\frac{\mathrm{rad}(P_{V_2}(S))}{S\cap V_2},b_{V_2}(S),b_{V_1}(S)\right),\]
    where $b_{V_2}(S)$ is the dimension of the maximal positive definite subspace of $P_{V_2}(S)$ and $b_{V_1}(S)$ is the dimension of the maximal negative definite subspace of $P_{V_2}(S)$. According to \cite[Theorem 5.2]{Huang:He}, the only possible tuples that appear are $(2,4,0,0,0),(1,3,0,1,1),(0,2,0,2,2)$. The calculation of the stabilizer of each orbit is given in \cite[\S 2.3]{Huang:He}. We then choose representatives $q_1$, $q_2$ and $q_3$, respectively, as in \emph{loc.cit.}
\end{proof}
\begin{remark}
   A representative of the open orbit is given by    \[q_3 = \left(\begin{smallmatrix}
I_2 & 0 & 0 \\
0& A & 0 \\  0 & 0 & I_2
\end{smallmatrix}\right), \text{ where } A = \left(\begin{smallmatrix}
 I_2 & 0_2 & -\frac{1}{2} I_2 & 0_2  \\
 I_2  & 0_2 & \frac{1}{2} I_2 & 0_2  \\
 0_2 & I_2  & 0_2 & \frac{1}{2} I_2  \\
 0_2 & -I_2  & 0_2 & \frac{1}{2} I_2
\end{smallmatrix}\right). \]
\end{remark}

Using Proposition \ref{Double:Quotient:Pullback:Orthogonal} and the cuspidality of $\Psi'\boxtimes\Psi''$, we proceed exactly as in Theorem \ref{PullbackFormulaNew}, to reduce \eqref{Main:Integral:Pullback:Orthogonal} to the integral 
\begin{equation}\label{First:Eisenstein:Second:Pullback}\sum_{\gamma\in Q_{(2,4)}(F)\setminus \mathbf{SO}_{\mathbb{H}^4}(F)}\int_{ \mathbf{SO}_{\mathbb{H}^2}(\A)}(\overline{\Psi}'\boxtimes\Psi'')(ht_{\nu'(x)})\Phi_{s}(q_3\kappa(ht_{\nu'(x)},\gamma x))dh.\end{equation}
This integral converges for $\mathrm{Re}(s)>>0$. Consider the normalized induction\footnote{The modulus character of $Q_{(2,4)}$ is exactly  $\delta_{Q_{(2,4)}}(m(l,h)) = |\mathrm{det}(l)|_{\A}^{5}|\nu'(h)|_{\A}^{-5}$ for $l \in \GL_2$, $h \in \GSO_{\mathbb{H}^2}$.} \begin{equation}\label{induction:orthogonal}I_{Q_{(2,4)}}(s,\overline{\sigma}\boxtimes \sigma) := \mathrm{Ind}_{Q_{(2,4)}(\A)}^{\mathbf{GSO}_{\mathbb{H}^4}(\A)} \left(|\mathrm{det}|_{\A}^{s}|\nu'|^{-s}_{\A}(\overline{\sigma}\boxtimes \sigma) \right).\end{equation}
For any $\varrho_{s}\in I_{Q_{(2,4)}}(s,\overline{\sigma}\boxtimes \sigma)$ define the Eisenstein series
\[E_{Q_{(2,4)}}(x,\varrho_{s}) := \sum_{\gamma\in Q_{(2,4)}(F)\setminus \GSO_{\mathbb{H}^4}(F)}\varrho_{s}(\gamma x),\]
which converges for $\mathrm{Re}(s)>>0$ and admits meromorphic continuation to the whole complex plane $\C$.
\begin{proposition}\label{Proposition:Pullback:Orthogonal}
    Let $\sigma \subset \mathcal{A}_{\rm cusp}([\GL_2])$, $\Psi',\Psi''$ be cusp forms in the space of $\sigma$, and let $\Phi_{s} \in  I_{Q_{(6,0)}}(s,\mathbf{1}) $ be a holomorphic section.
    \begin{enumerate}
    \item For $\mathrm{Re}(s)>>0$, we have
    \begin{align*}
    q_{\overline{\Psi}'\boxtimes\Psi''}(\Phi_{s})(x) := \int_{\mathbf{SO}_{\mathbb{H}^2}(\A)}(\overline{\Psi}'\boxtimes\Psi'')(ht_{\nu'(x)})\Phi_{s}(q_3\kappa(ht_{\nu'(x)},x))dh\in I_{Q_{(2,4)}}(s,\overline{\sigma}\boxtimes \sigma).
    \end{align*}
    \item We obtain the pullback formula
    \[\int_{[\mathbf{SO}_{\mathbb{H}^2}]}(\overline{\Psi}'\boxtimes\Psi'')(ht_{\nu'(x)})\mathrm{E}_{Q_{(6,0)}}(\kappa(h t_{\nu'(x)},x),\Phi_{s})dh = \mathrm{E}_{Q_{(2,4)}}(x,q_{\overline{\Psi}'\boxtimes\Psi''}(\Phi_{s})).\]
    \item Let $S$ be a finite set of places containing the archimedean places and the ramified places of $\sigma$ and $\Phi_s$. If $\overline{\Psi}'\boxtimes\Psi'' \in   \sigma\boxtimes \overline{\sigma}$ and $\Phi_{s} = \otimes_v \Phi_{s,v}\in I_{Q_{(6,0)}}(s,\mathbf{1})$ is normalized so that $\Phi_{s,v}(1) = 1$ when $v \not \in S$, we have \[q_{\overline{\Psi}'\boxtimes\Psi''}(\Phi_{s})(I_8) =  \frac{L^S(5/2+s,\sigma\boxtimes\overline{\sigma},\mathrm{std})}{\zeta^S(5+2s)\zeta^S(3+2s)} \int_{\mathbf{SO}_{\mathbb{H}^2}(\A_S)}(\overline{\Psi}'\boxtimes\Psi'')(h)\Phi_{s,S}(q_3\kappa(h,I_8))dh,\] 
    where $\A_S = \prod_{v\in S} F_v$ and $ \Phi_{s,S} = \otimes_{v \in S} \Phi_{s,v}$. In particular $s\mapsto q_{\overline{\Psi}'\boxtimes\Psi''}(\Phi_{s})$ admits meromorphic continuation to all $s\in \C$. Hence,  $q_{\overline{\Psi}'\boxtimes\Psi''}(\Phi_{s})\in I_{Q_{(2,4)}}(s,\overline{\sigma}\boxtimes \sigma)$ for every $s$ at which it is holomorphic.
    \end{enumerate}
\end{proposition}
\begin{proof} \leavevmode
    \begin{enumerate}
        \item The invariance of $q_{\overline{\Psi}'\boxtimes\Psi''}(\Phi_{s})$ by $N_{(2,4)}(\A)$ follows from the fact that  $q_3\kappa(I_4,n)q_3^{-1}\in N_{(6,0)}(\A)$. 
        If $m(l,h')\in M_{(2,4)}(\A)\simeq \GL_2(\A)\times \GSO_{\mathbb{H}^2}(\A)$, then  $q_{\overline{\Psi}'\boxtimes\Psi''}(\Phi_{s})(m(l,h'))$ equals \begin{align} &\int_{\mathbf{SO}_{\mathbb{H}^2}(\A)}(\overline{\Psi}'\boxtimes\Psi'')(ht_{\nu'(h')})\Phi_{s}(q_3\kappa(ht_{\nu'(h')},m(l,h')))dh \nonumber \\
 =&\int_{\mathbf{SO}_{\mathbb{H}^2}(\A)}(\overline{\Psi}'\boxtimes\Psi'')(ht_{\nu'(h')})\Phi_{s}(q_3\kappa(\tilde{h}',m(l,\tilde{h}'))  \kappa(  {\tilde{h}}^{'-1}h t_{\nu'(h')}, m(I_2, t_{\nu'(h')})))dh \nonumber  \\ 
=&|\det(l)|^{5/2+s}_{\A}\int_{\mathbf{SO}_{\mathbb{H}^2}(\A)}(\overline{\Psi}'\boxtimes\Psi'')(ht_{\nu'(h')})\Phi_{s}(q_3  \kappa(  {\tilde{h}}^{'-1}h t_{\nu'(h')}, m(I_2, t_{\nu'(h')})))dh ,\nonumber 
        \end{align}
        where in the first equality we have denoted $\tilde{h}' := t_{\nu'(h')}^{-1}h'$, and for the second we have used that $$q_3\kappa(\tilde{h}',m(l,\tilde{h}'))q_3^{-1}\in Q_{(6,0)}(\A),$$ and $$\delta_{Q_{(6,0)}}^{\frac{s}{5}+\frac{1}{2}}(q_3\kappa(\tilde{h}',m(l,\tilde{h}'))q_3^{-1}) = |\det(l)|^{5/2+s}_{\A}.$$ 
        Similarly, a direct computation shows that $$\delta_{Q_{(6,0)}}^{\frac{s}{5}+\frac{1}{2}}(q_3 \kappa(  t_{\nu'(h')}, m(I_2, t_{\nu'(h')}))q_3^{-1}) = | \nu'(h')|^{-s-5/2}.$$
        Using this, together with the change of variables $\tilde{h}^{'-1}h\mapsto h$, we obtain  
\begin{align}\label{SecondAuxiliarIntegral:PullbackFormula2}
        q_{\overline{\Psi}'\boxtimes\Psi''}(\Phi_{s})(m(l,h')) = |\det(l)|^{s+5/2}_{\A}|\nu(h')|^{-s-5/2}_{\A}\int_{\mathbf{SO}_{\mathbb{H}^2}(\A)}(\overline{\Psi}'\boxtimes\Psi'')(h'h)\Phi_{s}(q_3\kappa(h,I_8))dh.
        \end{align}
        Hence, for $\mathrm{Re}(s)>>0$,  we have $q_{\overline{\Psi}'\boxtimes\Psi''}(\Phi_{s})\in I_{Q_{(2,4)}}(s,\overline{\sigma}\boxtimes \sigma)$. 
        \item Proceeding exactly as in Theorem \ref{PullbackFormulaNew}, tt follows from \eqref{First:Eisenstein:Second:Pullback} and part (1) above. 
        \item Assume $v$ is a finite place at which both $(\sigma\boxtimes\overline{\sigma})_v$ and $\Phi_{s,v}$ are unramified. Let $v_0\in (\sigma\boxtimes\overline{\sigma})_v$ be the spherical vector satisfying $\langle v_0,v_0\rangle = 1$. Since the space of spherical vectors of $(\sigma\boxtimes\overline{\sigma})_v$ is one-dimensional, we have
        \[\int_{\mathbf{SO}_{\mathbb{H}^2}(F_v)}\Phi_{s,v}(q_3\kappa(h,I_8))(\sigma\boxtimes\overline{\sigma})_v(h)v_0dh = c_v(s)v_0,\]
        where
        \[c_v(s) = \int_{\mathbf{SO}_{\mathbb{H}^2}(F_v)}\Phi_{s,v}(q_3\kappa(h,I_8))\langle(\sigma\boxtimes\overline{\sigma})_v(h)v_0,v_0\rangle dh.\]
        Using the Cartan decomposition as in Theorem \ref{PullbackFormulaNew} (2) and \cite[Proposition 6.4 (a) \& Proof of Theorem 6.1]{ExplicitGSR}, one shows that $c_v(s)$ equals $$\frac{L(5/2+s,\sigma_v\boxtimes\overline{\sigma}_v,\mathrm{std})}{\zeta_v(5+2s)\zeta_v(3+2s)}. $$ 
    \end{enumerate}

\end{proof}

Denote
\begin{align}\label{embedding:e} e:\GL_4 \times \GL_1 &\to \mathbf{GSO}_{\mathbb{H}^4},\\ (r,a)&\mapsto e(r,a) := \left(\begin{smallmatrix}r & \\ &aI_4'\;^{t}r^{-1}I_4'\end{smallmatrix}\right).\nonumber\end{align}

\begin{corollary}\label{Corollary:Pullback:Orthogonal}
    Assume \eqref{Spl} with $ \lambda_{\mathbb{H}^2}(f_1 \otimes f_2)^{\sim} = \phi_{\frac{1}{2}},$ for $ f_1\in \mathcal{S}(\mathbb{H}^2(\A))$ and $f_2 \in \mathcal{S}(\mathbb{H}^2(\A)^2)$. Let $\eta \in \mathcal{S}(M_4(\A))$ have support outside $0$. We have
\[J(\tfrac{1}{2},z,\varphi,f^\Psi_{\phi_{1/2}}, \eta) = \tfrac{\zeta^S_F(2)^2}{4}\int_{[ \GL_1 \backslash (\GL_4 \times \GL_1)]}|a|_\A^{-2z}|\det(r)|^{z}_{\A}\varphi(r)\mathrm{E}_{Q_{(2,4)}}(e(r,a),q_{\overline{\Psi}'\boxtimes\Psi''}(\Phi_{-\frac{1}{2}})) drda.\]
\end{corollary}
\begin{proof}
    It follows from Propositions \ref{Final:Proposition:SeeSaw} and \ref{Proposition:Pullback:Orthogonal}, that, for ${\rm Re}(z)>>0$, we have
\[J(\tfrac{1}{2},z,\varphi,f^\Psi_{\phi_{1/2}}, \eta) = \tfrac{\zeta^S_F(2)^2}{4}\int_{[ \GL_1 \backslash (\GL_4 \times \GL_1)]}|a|_\A^{-2z}|\det(r)|^{z}_{\A}\varphi(r)\mathrm{E}_{Q_{(2,4)}}(e(r,a),q_{\overline{\Psi}'\boxtimes\Psi''}(\Phi_{-\frac{1}{2}})) drda.\]
Since the left hand side defines a holomorphic function on $z$ and the right hand side gives a holomorphic function on $z$ because of rapid decay of $\varphi$, we get the desired equality.
\end{proof}

\subsection{The spherical period}\label{subsec:spherical_period}

Consider the embedding of algebraic groups
\begin{align*}
    \GL_2\times \GL_2&\to \GL_4\times \GL_2,\\
    (g_1, g_2) &\mapsto \left( \mathrm{diag}(g_1,g_2),g_2\right),
\end{align*}
where $\mathrm{diag}(g_1,g_2) = \left(\begin{smallmatrix}
        g_1& \\ &g_2\end{smallmatrix}\right)$. The quotient
\[\mathbf{X} := (\GL_2\times \GL_2)\setminus (\GL_4\times \GL_2),\]
defines a strongly tempered spherical variety (see \cite[Table 1]{Wan:Zhang:Spherical:Periods}). For $\varphi\in \mathcal{A}_{\rm cusp}([\GL_4])$ and $\Psi\in \mathcal{A}_{\rm cusp}([\GL_2])$ with trivial product of central characters, define
\[\mathcal{P}_{\GL_2 \times \GL_2}(\varphi,\Psi):= \int_{Z_{\GL_4}(\A)\backslash [\GL_2\times \GL_2]}\varphi(\mathrm{diag}(g_1,g_2))\Psi(g_2)dg_1dg_2.\]
This is the automorphic period attached to the spherical variety $ \mathbf{X}$ (or rather to the BZSV quadruple $(\GL_4 \times \GL_2, \GL_2\times \GL_2, 0, 1 )$ in the notation of \cite{MWZ}). Associated with $\mathbf{X}$, there is also an automorphic $L$-function $L(s,\pi \otimes \sigma, \rho_\mathbf{X})$ on $\GL_4 \times \GL_2$ (see \cite{sakellaridis2013spherical} and \cite{Wan:Zhang:Spherical:Periods}), where $\rho_{\mathbf{X}}$ is the representation of $^L(\GL_4\times \GL_2)$, given by $$\rho_{\mathbf{X}} := \wedge^2 \otimes {\rm std}_2 \oplus {\rm std}_4 \oplus {\rm std}_4^\vee,$$
where recall that ${\rm std}_n$  and ${\rm std}_n^\vee$ denote the standard representation of $\GL_n(\C)$ and its dual, respectively. 

The aim of this section is to set Conjecture \ref{conj:Wan:Zhang:+:Bessel:GGP}(2), when $\pi \subset \mathcal{A}_{\rm cusp}([\GL_4])$ and $\sigma \subset \mathcal{A}_{\rm cusp}([\GL_2])$ are unramified everywhere. In particular, we will show that $$ L(\tfrac{1}{2}, \pi \otimes \sigma, \rho_{\mathbf{X}}) \ne 0 \Longrightarrow \mathcal{P}_{\GL_2 \times \GL_2}|_{\pi \otimes \sigma} \not \equiv 0.$$ 
Note that the direction $\Longleftarrow$ follows by \cite[Corollary 1.3(6)]{Pollack:Wan:Zydor} and \cite[Theorem 1.1]{Watanabe:Global:Theta}, under certain additional hypotheses \footnote{The work \cite{Pollack:Wan:Zydor} assumes that $\pi$ and $\sigma$ have trivial central characters and that $L(\frac{3}{2},\pi \otimes \sigma,\wedge^2\otimes\mathrm{std}_2)\neq 0$, which is always true if $\pi$ and $\sigma$ are tempered. Under these hypotheses, they deduce that $\mathcal{P}_{\GL_2 \times \GL_2}|_{\pi \otimes \sigma} \not \equiv 0 \Longrightarrow L(\frac{1}{2},\pi \otimes \sigma,\wedge^2\otimes\mathrm{std}_2)\neq 0$. To obtain the non-vanishing of $L(\tfrac{1}{2}, \pi \otimes \sigma, \rho_{\mathbf{X}})$, one also needs a comparison between $\mathcal{P}_{\GL_2 \times \GL_2}|_{\pi \otimes \sigma}$ and generalized Shalika periods for $\Theta^0(\pi) \otimes \sigma$, which we discuss in Lemma \ref{Shalika:Equals:Period}}. 

To this end, our aim is to relate the period $\mathcal{P}_{\GL_2 \times \GL_2}|_{\pi \otimes \sigma}$ to
\begin{equation}\label{Integral:Section:Period:Final}\int_{[ \GL_1 \backslash (\GL_4 \times \GL_1)]}\varphi(r)\mathrm{E}_{Q_{(2,4)}}(e(r,a),q_{\overline{\Psi}'\boxtimes\Psi''}(\Phi_{s})) drda,\end{equation}
as it plays a key role in proving that $$J(\tfrac{1}{2},0,\varphi,f^\Psi_{\phi_{1/2}}, \eta) \ne 0  \Longrightarrow {\mathcal{P}_{ \GL_2 \times \GL_2}}|_{\pi \otimes \sigma} \not \equiv 0.$$

\begin{lemma}\label{Final:Double:Quotient}
    Let $Q_{(2,4)}$ be the parabolic subgroup of $\mathbf{GSO}_{\mathbb{H}^4}$ with Levi component equal to $\GL_2\times \mathbf{GSO}_{\mathbb{H}^2}$. We have that 
    \[Q_{(2,4)}(F)\setminus \mathbf{GSO}_{\mathbb{H}^4}(F)/e(\GL_4\times \GL_1)(F) = \{\omega_i\}_{i = 1,\cdots,7}.\]
    For $1\leq i \leq 6$, the stabilizer of $\omega_i$ in $\GL_4$ contains a normal unipotent subgroup. Moreover, the stabilizer of the (open) orbit $ \omega_7$ is isomorphic to   \[\left\{\left(\left(\begin{smallmatrix}g_1& \\ &g_2\end{smallmatrix}\right),\mathrm{det}(g_1)\right),\;g_1,g_2\in \GL_2(F)\right\}\subset \GL_4(F)\times \GL_1(F).\]
\end{lemma}
\begin{proof}
    The quotient $Q_{(2,4)}(F)\setminus \mathbf{GSO}_{\mathbb{H}^4}(F)$ can be identified with the set of flags $$\mathcal{V} := \{S\subset \mathbb{H}^4,\;\mathrm{dim}(S) = 2,\;S\;\text{isotropic}\}.$$ Set $U_1 := \langle e_1,e_2,e_3,e_4\rangle $ and $U_2 := \langle f_4,f_3,f_2,f_1\rangle$. Given $S\in \mathcal{V}$, consider the following invariants for the $\GL_4(F) \times \GL_1(F)$-orbit $\mathcal{S}$ of $S$:
    \[ \mathbf{d}_\mathcal{S} : = (\mathrm{dim}(S\cap U_1), \mathrm{dim}(S\cap U_2)).\]
    When $\mathbf{d}_\mathcal{S} \ne (0,0)$, we get 5 (closed) orbits corresponding to $(2,0),(0,2),(1,1),(1,0),(0,1)$. An elementary computation shows that the stabilizer of each of these orbits contains a normal unipotent subgroup. We omit the details. Suppose that $\mathbf{d}_\mathcal{S} = (0,0)$. For $S \in \mathcal{S}$, write
\begin{equation}\label{Expression:Projections:Final:Double}
    S = \langle x,y\rangle  = \langle p_{U_1}(x)+p_{U_2}(x),p_{U_1}(y)+p_{U_2}(y)\rangle,
    \end{equation}
    where $p_{U_i}(\cdot)$ is the projection map to the vector space $U_i$. Any element $(g,\lambda)\in \GL_4(F)\times \GL_1(F)$ acts on $S$ via
    \[ \langle x,y\rangle \cdot e(g,\lambda) = \langle  p_{U_1}(x) \cdot g +  p_{U_2}(x) \cdot \lambda g', p_{U_1}(y) \cdot  g +  p_{U_2}(y)\cdot \lambda g'\rangle,\]
    where $g' = I_4'\;^{t}g^{-1}I_4'$. After using the $\GL_4(F)$-action, we can suppose that $p_{U_1}(x) = e_1$ and $p_{U_1}(y) = e_2$. We now show there are two $ \GL_4(F) \times \GL_1(F)$-orbit $\mathcal{S}$ of planes satisfying $\mathbf{d}_\mathcal{S} = (0,0)$ with representatives equal to $\langle e_1 + f_3, e_2 + f_4 \rangle$ and $\langle e_1 - f_2, e_2 + f_1 \rangle$. Indeed, if $S =  \langle e_1+p_{U_2}(x), e_2 +p_{U_2}(y)\rangle$, we can assume  (because of isotropy of $S$)
    \begin{align*}
    p_{U_2}(x) &= \alpha_2 f_2 + \alpha_3 f_3 +\alpha_4 f_4   \\ 
    p_{U_2}(y) &= -\alpha_2 f_1 + \beta_3 f_3 +\beta_4 f_4, 
     \end{align*}    
    with $\alpha_i,\beta_i \in F$. \begin{itemize} 
    \item Suppose that  $\alpha_2 = 0$. Then, the pair $\{p_{U_2}(x),    p_{U_2}(y) \}$ defines a basis of $\langle f_3 ,f_4 \rangle$. Using the action of the Levi of the Siegel parabolic $P$ of $\GL_4(F)$, which stabilizes\footnote{As we are letting $\GL_4$ act on the right, this Siegel parabolic is contained in the lower triangular Borel of $\GL_4$} the plane $\langle e_1 , e_2 \rangle$, we can thus send $\langle e_1+p_{U_2}(x), e_2 +p_{U_2}(y)\rangle$ to $\langle e_1+f_3, e_2 +f_4\rangle$.  A direct computation shows that $\mathrm{Stab}_{\GL_4(F)\times \GL_1(F)}(\omega_6)$ contains, as a normal subgroup, the unipotent radical of $P$. 
    \item Suppose that $\alpha_2 \ne 0$. Use the action of the unipotent radical of $P$ to send $\langle e_1+p_{U_2}(x), e_2 +p_{U_2}(y)\rangle$ to $ \langle e_1 +  \alpha_2 f_2, e_2 -\alpha_2 f_1\rangle$. Finally, the action of $e(\{ I_4\} \times F^\times)$ let us send $S$ to $\omega_7 :=\langle e_1 - f_2, e_2 + f_1 \rangle $. By direct computation, its stabilizer is given by \[\left\{\left(\left(\begin{smallmatrix}g_1& \\ &g_2\end{smallmatrix}\right),\mathrm{det}(g_1)\right),\;g_1,g_2\in \GL_2(F)\right\}\subset \GL_4(F)\times \GL_1(F).\]
    \end{itemize}
\end{proof}
\begin{remark}
In matrix form, the representative of the open orbit can be chosen to be
\[\omega_7 =\left(\begin{smallmatrix}
I_2 &  &  -A \\
 & I_4 &  \\
A & &  0_2
\end{smallmatrix}\right), \text{ with } A =\left(\begin{smallmatrix} 1 & 0 \\ 0 & -1 
\end{smallmatrix}\right).\] 
\end{remark}

Note that the section $q_{\overline{\Psi}'\boxtimes\Psi''}(\Phi_{s})$ gives a family of cusp forms on $ \overline{\sigma}\boxtimes \sigma$. To simplify exposition, we give the following.
\begin{notation}
   For any $(r,a) \in \GL_4(\A) \times \GL_1(\A)$ and $s \in \C$, denote by  \[(\Lambda_{r,a,s}'\boxtimes\Lambda''_{r,a,s})(h) := q_{(\overline{\sigma} \boxtimes \sigma)(h^{-1})\cdot\overline{\Psi}'\boxtimes\Psi''}(\Phi_{s})(\omega_7e(r,a)).\] 
   \end{notation}
Concretely, $\Lambda_{r,a,s}'\boxtimes\Lambda''_{r,a,s}$ is a cusp form of $\overline{\sigma} \boxtimes  \sigma$, depending, in particular, on the section $\Phi_s$ and the cusp form $\overline{\Psi}'\boxtimes\Psi''$ of $\overline{\sigma} \boxtimes  \sigma$.

\begin{proposition}\label{Unfolding:Spherical:Period} When $\mathrm{Re}(s)>>0$, the integral \eqref{Integral:Section:Period:Final} equals 
      \begin{align*}
        \int_{ \mathbf{Y}(\A)^\circ } \Lambda_{r,a,s}'(I_2)\int_{Z_{\GL_4}(\A)\backslash [\GL_2\times \GL_2]} \!\!\!\!\varphi({\rm diag}(g_1,g_2) r) \Lambda''_{r,a,s}(g_2) dg_1dg_2drda.  
             \end{align*}
   where  $\mathbf{Y}(\A)^\circ :=  (\GL_2 \times \GL_2)(\A) \backslash (\GL_4 \times \GL_1)(\A)$.
\end{proposition}
\begin{proof}
    We proceed as in Theorem \ref{Final:Theorem:Central:Value} by unfolding the integral \eqref{Integral:Section:Period:Final}. For ${\rm Re}(s)$ big enough, the Eisenstein series unfolds and, by Lemma \ref{Final:Double:Quotient} and the cuspidality of $\varphi$ and $\Psi'\boxtimes\Psi''$, \eqref{Integral:Section:Period:Final} is shown to be equal to    \begin{equation}\label{First:Eq:Prop:Final:Unfolding}\int_{ (\GL_2 \times \GL_2)(F) \GL_1(\A) \backslash (\GL_4 \times \GL_1)(\A)}\varphi(r) q_{\overline{\Psi}'\boxtimes\Psi''}(\Phi_{s})(\omega_7 e(r,a)) drda,\end{equation}  
    where we recall that $\GL_1$ embeds into $\GL_4\times \GL_1$ via \eqref{Embedding:GL1}. Collapsing the integral over $[\GL_2 \times \GL_2]$, we get 
\begin{equation}\label{Second:Eq:Prop:Final:Unfolding}\int_{ \mathbf{Y}(\A)^\circ } \int_{Z_{\GL_4}(\A)\backslash [\GL_2\times \GL_2]} \!\!\!\!\!\!\!\!\!\!\! \varphi({\rm diag}(g_1,g_2)r) q_{\overline{\Psi}'\boxtimes\Psi''}(\Phi_{s})(\omega_7 e({\rm diag}(g_1,g_2) r,\det(g_1) a)) drda.\end{equation} 
    Consider the element
    \[\alpha := \left(\begin{smallmatrix}1& & & & &\\ &\mathrm{det}(g_1)^{-1}\mathrm{det}(g_2)& & & & \\ & & I_2& & &\\  & & &\mathrm{det}(g_1)^{-1}\mathrm{det}(g_2)I_2 & & \\   & & & &1 & \\   & & & & &\mathrm{det}(g_1)^{-1}\mathrm{det}(g_2)\end{smallmatrix}\right)\in  Q_{(2,4)}(\A),\]
which commutes with $\omega_7$. Note
    \[\omega_7 e(\mathrm{diag}(g_1,g_2),\mathrm{det}(g_1)) \alpha \omega_7^{-1} = \left(\begin{smallmatrix}\mathrm{det}(g_2)A'\;^t( g_1 \cdot t_{g_1,g_2} )^{-1}A'^{-1}  & & & \star\\ &g_2& & \\ & &\mathrm{det}(g_2)I_2'\;^tg_2^{-1}I_2'& \\ & & &A g_1 \cdot t_{g_1,g_2} A^{-1}\end{smallmatrix}\right)\in Q_{(2,4)}(\A),\]
    where $A= \left(\begin{smallmatrix}1& \\ &-1\end{smallmatrix}\right)$, $ A':= \left(\begin{smallmatrix} & -1\\1& \end{smallmatrix}\right)$, and  $t_{g_1,g_2} := \left(\begin{smallmatrix}1 & \\ &\mathrm{det}(g_1)^{-1}\mathrm{det}(g_2)\end{smallmatrix}\right)$. Thus, multiplying and dividing by $\alpha$ to the right of $\omega_7e(\mathrm{diag}(g_1,g_2),\mathrm{det}(g_1))\omega_7^{-1}$, by Proposition \ref{Proposition:Pullback:Orthogonal} (1), we have that the integral \eqref{Second:Eq:Prop:Final:Unfolding} equals
        \begin{align*}
     \int_{ \mathbf{Y}(\A)^\circ } \int_{Z_{\GL_4}(\A)\backslash [\GL_2\times \GL_2]} \!\!\!\!\varphi({\rm diag}(g_1,g_2)r) 
       q_{(\overline{\sigma} \boxtimes \sigma)(\tilde{g}_2^{-1}) \cdot \overline{\Psi}'\boxtimes\Psi''}(\Phi_{s})(\omega_7 \alpha^{-1}e(r,a))dg_1dg_2drda,
       \end{align*}
       where we denoted $\tilde{g}_2 := \left(\begin{smallmatrix}g_2& \\ &\mathrm{det}(g_2)I_2'\;^tg_2^{-1}I_2'\end{smallmatrix}\right)$. Since $\alpha^{-1} = e\left(m_{g_1,g_2}, \mathrm{det}(g_1)\mathrm{det}(g_2)^{-1}\right)$, where $$m_{g_1,g_2} := \left(\begin{smallmatrix}1 & & & \\ &\mathrm{det}(g_1)\mathrm{det}(g_2)^{-1}& & \\ & &1 & \\ & & &1\end{smallmatrix}\right)\in \GL_4(\A),$$  we do the change of variables \[\mathrm{det}(g_1)\mathrm{det}(g_2)^{-1}a\mapsto a,\;\;\;\;m_{g_1,g_2}r\mapsto r,\]
       to obtain
       \begin{align*}
       &\int_{ \mathbf{Y}(\A)^\circ } \int_{Z_{\GL_4}(\A)\backslash [\GL_2\times \GL_2]} \!\!\!\!\varphi({\rm diag}(g_1 ,g_2) m_{g_1,g_2}^{-1} r) 
       q_{(\overline{\sigma} \boxtimes \sigma)(\tilde{g}_2^{-1}) \cdot \overline{\Psi}'\boxtimes\Psi''}(\Phi_{s})(\omega_7 e(r,a))dg_1dg_2drda \\ 
        &=\int_{ \mathbf{Y}(\A)^\circ } \int_{Z_{\GL_4}(\A)\backslash [\GL_2\times \GL_2]} \!\!\!\!\varphi({\rm diag}(g_1 t_{g_1,g_2},g_2) r) 
       q_{(\overline{\sigma} \boxtimes \sigma)(\tilde{g}_2^{-1}) \cdot \overline{\Psi}'\boxtimes\Psi''}(\Phi_{s})(\omega_7 e(r,a))dg_1dg_2drda .
        \end{align*}
      Lastly, we do the change of variables $g_1t_{g_1,g_2}\mapsto g_1$, so that the above integral equals 
        \begin{align*}
        \int_{ \mathbf{Y}(\A)^\circ } \int_{Z_{\GL_4}(\A)\backslash [\GL_2\times \GL_2]} \!\!\!\!\varphi({\rm diag}(g_1,g_2) r) 
       q_{(\overline{\sigma} \boxtimes \sigma)(\tilde{g}_2^{-1}) \cdot \overline{\Psi}'\boxtimes\Psi''}(\Phi_{s})(\omega_7 e(r,a))dg_1dg_2drda.  
             \end{align*}
      The result now follows from the fact that, via the exceptional isomorphism, $ \tilde{g}_2$ corresponds to $(I_2, g_2)$ (see Remark \ref{Remark:Siegel:Parabolic:Exceptional}). Then 
    \begin{align*}q_{(\overline{\sigma} \boxtimes \sigma)(\tilde{g}_2^{-1}) \cdot \overline{\Psi}'\boxtimes\Psi''}(\Phi_{s})(\omega_7e(r,a)) = \left( \Lambda_{r,a,s}'\boxtimes\Lambda''_{r,a,s}\right)(\tilde{g_2}) = \Lambda_{r,a,s}'(I_2)\Lambda''_{r,a,s}(g_2),
    \end{align*}
    obtaining the result. 
\end{proof}
\begin{theorem}\label{Main:Theorem:Spherical:Period}
     Let $\pi$ and $\sigma$ be cuspidal automorphic representations  of $\GL_4$ and $\GL_2$ respectively, such that $\omega_\pi  \omega_\sigma = \mathbf{1}.$
     Assume that $\phi_{1/2} = \otimes_v \phi_{1/2,v} \in  I_{P_{(3,0)}}(1/2, \mathbf{1})$ satisfies the hypothesis \eqref{Spl}. Then, for any cusp form $\varphi = \otimes_v \varphi_v\in \pi$ and Schwartz function $\eta = \prod_v \eta_v \in \mathcal{S}(M_4(\A))$ with support outside $0$ and unramified outside $S$, we have that if
     \begin{equation}\label{Lfunctionsandzetaintegrals} L( \tfrac{1}{2},\pi,{\rm std}_4)L^S(\tfrac{1}{2},\pi \otimes \sigma,\wedge^2 \otimes \mathrm{std}_2) \prod_{v \in S} J^\star(1/2,0,\varphi_v,f_{\phi_{1/2,v}}^\Psi,\eta_v)\neq 0,\end{equation}
     there exists a cusp form $\Lambda \in \sigma$, depending on $f$ and $\eta$, such that $\mathcal{P}_{\GL_2 \times \GL_2}(\varphi,\Lambda) \ne 0$. 
\end{theorem}
\begin{proof}
  It follows from Corollary \ref{Corollary:Multi:Variable:Integral}, Corollary \ref{Corollary:Pullback:Orthogonal}, and Proposition \ref{Unfolding:Spherical:Period}.
\end{proof}

\begin{remark}\label{Main:remark:after:MTSP}
    We note that we could formulate Theorem \ref{Main:Theorem:Spherical:Period} in the same style as Theorem \ref{Final:Theorem:Central:Value} without imposing the hypothesis \eqref{Spl}. In this case, the non-vanishing of \eqref{Lfunctionsandzetaintegrals} 
    implies the non-vanishing of either $\mathcal{P}_{\GL_2 \times \GL_2}|_{\pi \otimes \sigma}$ or of the period \eqref{intermediateseesawintegral} for some non-split quaternionic quadratic space $V=D$ of dimension $4$ over $F$. However, we do not currently know how to relate \eqref{intermediateseesawintegral} to the linear period  $\mathcal{P}_{D^\times \times D^\times}$  - see Remark \ref{remark_on_spl}.
\end{remark}

\begin{corollary}\label{Main:Corollary:Spherical:1}
    Let $\pi$ and $\sigma$ be cuspidal automorphic representations of $\GL_4$ and $\GL_2$ respectively, such that $\omega_\pi  \omega_\sigma = \mathbf{1}$. \begin{enumerate}
        \item Assume that there exist $\phi \in \Pi(\mathbb{H}^2)$, $\eta = \otimes_v \eta_{v}\in \mathcal{S}(M_4(\A))$, $\Psi = \otimes_v \Psi_v \in \sigma$, and $\varphi = \otimes_v \varphi_v \in \pi$ such that
    \[\prod_{v \in S} J^\star(1/2,0,\varphi_v,f^{\Psi}_{\phi,v},\eta_v)\neq 0.\]
    Then \[L( \tfrac{1}{2},\pi,\mathrm{std}_4)L( \tfrac{1}{2},\pi,\mathrm{std}_4^{\vee})L^S(\tfrac{1}{2},\pi \otimes \sigma,\wedge^2 \otimes \mathrm{std}_2)\neq 0 \Longrightarrow {\mathcal{P}_{ \GL_2 \times \GL_2}}|_{\pi \otimes \sigma} \not \equiv 0.\]
    \item If $\pi$ and $\sigma$ are unramified everywhere, we have
    \[L( \tfrac{1}{2}, \pi \otimes \sigma, \rho_{\mathbf{X}}) \neq 0 \Longrightarrow {\mathcal{P}_{ \GL_2 \times \GL_2}}|_{\pi \otimes \sigma} \not \equiv 0.\]
    \end{enumerate} 
\end{corollary}

\begin{proof}
    We first show that the non-vanishing of $L( \tfrac{1}{2},\pi,\mathrm{std}_4^{\vee})$ is equivalent to the non-vanishing of $L( \tfrac{1}{2},\pi,{\rm std}_4)$. For any place $v$, we denote by $\phi_{\pi_v}$ the Langland parameter of the representation $\pi_v$ of $\GL_4(F_v)$. By the local Langlands correspondence for $\GL_4$, we have that $\phi_{\pi_v^{\vee}} = \phi_{\pi_v}^{\vee}$ (see \cite[(1.4)]{Adams:Vogan}), which implies that \[L( \tfrac{1}{2},\pi^{\vee},{\rm std}_4) = L( \tfrac{1}{2},\pi,{\rm std}_4^{\vee}).\]
    Moreover, the functional equation gives
    \[L( s,\pi,{\rm std}_4) = \varepsilon(s,\pi,{\rm std}_4)L( 1-s,\pi^{\vee},{\rm std}_4).\]
    Consequently, $L( \tfrac{1}{2},\pi,{\rm std}_4)\neq 0$ if and only if $L( \tfrac{1}{2},\pi,{\rm std}_4^{\vee})\neq 0$. Then,  (1) follows from Theorem \ref{Main:Theorem:Spherical:Period}. For (2), 
recall that, thanks to the archimedean computation of Theorem \ref{Theorem_Arch_computation_split}, Corollary \ref{Corollary:Multi:Variable:Integral} shows that there exist data $(\varphi, f , \eta)$ such that $J(\varphi, f , \eta, s ,z)$ calculates the product of completed $L$-functions $$L( \tfrac{1}{2},\pi,\mathrm{std}_4) L(\tfrac{1}{2},\pi \otimes \sigma,\wedge^2 \otimes \mathrm{std}_2).$$ Since $S = \emptyset$, Lemma \ref{Non:Vanishing:Pullback} (ii) allows us to choose a non-trivial section $f = f^\Psi_\phi$, which is unramified at every place. Then, the implication follows from Theorem \ref{Main:Theorem:Spherical:Period}.
\end{proof}

\begin{remark}
    We expect the hypothesis on the non-vanishing of the local zeta integrals at any place $v$ in $S$ to be related to the non-vanishing of the local relative characters for the strongly tempered pair $(\GL_4(F_v), \GL_2(F_v) \times \GL_2(F_v))$. For example, an instance of this phenomenon is given in \cite{IchinoTriple} (see the proof of Theorem 1.1 and equation (5.1) of \emph{loc.cit.}). 
\end{remark}

\subsection{The relation between generalized Shalika periods and linear periods}

In what follows, we conclude the proof of Theorem \ref{Thmb}. Recall that, in \cite{Pollack:Wan:Zydor}, under certain hypotheses on $\pi \otimes \sigma$, it is shown that the non-vanishing of the period over $\GL_2 \times \GL_2$ implies that the central $L$-value $L(\tfrac{1}{2},\pi \otimes \sigma,\wedge^2 \otimes \mathrm{std}_2)\neq 0$. To further show that this implies the non-vanishing of $L( \tfrac{1}{2}, \pi \otimes \sigma, \rho_{\mathbf{X}})$, we use the fact that the non-vanishing of ${\mathcal{P}_{ \GL_2 \times \GL_2}}|_{\pi \otimes \sigma}$ implies that $\pi$ is in the image of the theta correspondence from $\GL_4$. The latter essentially follows from the computations of \cite[\S 3]{GanPeriodsandTheta} and from the proof of \cite[Proposition 7.1]{Chen-Gan}.

Recall that, in Definition \ref{definition_shalika_generalized}, given cusp forms $\varphi \in \mathcal{A}_{{\rm cusp}}(\GL_4)$, $\Psi \in \mathcal{A}_{{\rm cusp}}(\GL_2)$ with trivial product of central characters, we have defined the generalized Shalika period by
   \[\mathcal{S}_{\psi_S}({\varphi,\Psi}):= \int_{\A^\times \backslash[\GL_2]} \int_{[N]}  \varphi(n\, m(g))\Psi(g) \psi_S^{-1}(n) dn\,dg.\]

The authors thank Wee Teck Gan for sharing the following result with us. 

\begin{lemma}\label{Shalika:Equals:Period}
Given a cuspidal automorphic representation $\pi \otimes \sigma$ of $\GL_4 \times \GL_2$, we have 
$$\mathcal{S}_{\psi_S}|_{\Theta^0(\pi) \otimes \sigma}  \not \equiv 0 \Longleftrightarrow {\mathcal{P}_{ \GL_2 \times \GL_2}}|_{\pi \otimes \sigma} \not \equiv 0.$$
In particular, if ${\mathcal{P}_{ \GL_2 \times \GL_2}}|_{\pi \otimes \sigma}  \not \equiv 0$, then $\pi$ is in the image of the theta correspondence from $\GL_4$.
\end{lemma}
\begin{proof}
    The proof follows closely that of \cite[Proposition 7.1]{Chen-Gan}. For the reader's convenience, we sketch it.
    
      Let $\eta \in \mathcal{S}_0(M_4(\A))$ be a Schwartz function whose support does not contain zero. For $\mathrm{Re}(z)>>0$, consider
    \[\mathcal{S}_{\psi_S}({\varphi_\eta^z,\Psi})= \int_{\A^\times \backslash[S]}\Psi(g)\psi_S^{-1}(n)\int_{[\GL_4]}|\det(h)|^z\varphi(h)\sum_{\substack{x\in M_{4}(F)}}(\omega_{4,4}(n\, m(g),h)\eta)(x)dh\,dn \,dg.\]
    Choose two $4$-dimensional $F$ vector spaces $L_1$, $L_2$ with $M_{4}(F)\simeq L_1^\vee \otimes L_2$ and fix a decomposition $L_2 = \mathcal{L}_1\oplus \mathcal{L}_2$, where $\mathrm{dim}_F\;\mathcal{L}_i = 2$. Let
    \[\mathcal{F}:\mathcal{S}(M_{4}(\A))\to \mathcal{S}(L_1(\A)\otimes \mathcal{L}_1^{\vee}(\A))\otimes \mathcal{S}(L_1^{\vee}(\A)\otimes \mathcal{L}_2(\A))\]
    be the partial Fourier transform of \cite[\S 3.2]{GanPeriodsandTheta}. By Poisson summation
    \[\sum_{\substack{x\in M_{4}(F)}}(\omega_{4,4}(h_1,h_2)\eta)(x) = \sum_{(y,t)\in (L_1\otimes \mathcal{L}_1^{\vee})\oplus (L_1^{\vee}\otimes \mathcal{L}_2)}(\omega(h_1,h_2)\mathcal{F}(\eta))(y,t),\]
    with $\omega$ denoting the action of the Weil representation on this different model. Using the explicit action of $N$ on this model, as given by the formulas in \textit{loc.cit.}, the inner integral over $[N]$ of $\mathcal{S}_{\psi_S}({\varphi_\eta^z,\Psi})$ is non-zero for $(y,t)$ satisfying $ty = I_4$, while it vanishes otherwise for orthogonality of characters. Hence,
    \[\mathcal{S}_{\psi_S}({\varphi_\eta^z,\Psi}) = \int_{\A^{\times}\setminus[\GL_2]}\Psi(g)\int_{[\GL_4]}|\det(h)|^z\varphi(h)\sum_{\substack{(y,t)\in \mathcal{O}}}(\omega(m(g),h)\mathcal{F}(\eta))(y,t)dhdg,\]
    where $\mathcal{O} := \{(y,t)\in (L_1\otimes \mathcal{L}_1^{\vee})\oplus (L_1^{\vee}\otimes \mathcal{L}_2),\,\text{with }\,ty = I_4\}$. By \cite[\S 3.3]{GanPeriodsandTheta}, the group $\GL_2(F)\times \GL_4(F)$ acts transitively on $\mathcal{O}$. Furthermore, as explained in \emph{loc.cit.}, one can choose $(y_\mathcal{O},t_\mathcal{O})\in \mathcal{O}$ such that its stabilizer is $$H(F) = \left\{ (g,  \mathrm{diag}(g,h))\,:\, g,h \in \GL_2(F) \right\} \hookrightarrow \GL_2(F)\times \GL_4(F).$$
    Unfolding the sum yields that $\mathcal{S}_{\psi_S}({\varphi_\eta^z,\Psi})$ equals 
    \begin{align}\label{eq_final_Shal}
        \int_{H(\A)\setminus (\GL_2(\A)\times \GL_4(\A)}(\omega(m(g),h)\mathcal{F}(\eta))(y_\mathcal{O},t_\mathcal{O})\mathcal{P}_{\GL_2\times\GL_2}(\pi(h)(\varphi \otimes|\det(-)|^z),\sigma(g)\Psi)dhdg.
    \end{align}
    Note that one can choose $\eta$ so that the integral \eqref{eq_final_Shal} converges absolutely when $z = 0$, which shows that  
    $$\mathcal{S}_{\psi_S}|_{\Theta^0(\pi) \otimes \sigma}  \not \equiv 0 \Longrightarrow {\mathcal{P}_{ \GL_2 \times \GL_2}}|_{\pi \otimes \sigma} \not \equiv 0.$$
    Finally, arguing as in \cite[Proposition 7.1]{Chen-Gan}, one can further choose $\eta$ so that  $\mathcal{P}_{\GL_2\times\GL_2}(\varphi,\Psi)\neq 0$ implies that the integral \eqref{eq_final_Shal} is non-zero at $z=0$. Hence, $\mathcal{S}_{\psi_S}|_{\Theta^0(\pi) \otimes \sigma}  \not \equiv 0$ as desired. 
\end{proof}

Thanks to Lemma \ref{Shalika:Equals:Period}, using a combination of the above-mentioned result of \cite{Pollack:Wan:Zydor} and Corollary \ref{Main:Corollary:Spherical:1}, we can prove \cite[Conjecture 1.9]{Wan:Zhang:Spherical:Periods} for Model 1 of \emph{loc.cit.} in the case where $\pi$ and $\sigma$ are unramified everywhere.

\begin{corollary}\label{Main:Corollary:Spherical:2}
     Let $\pi$ and $\sigma$ be cuspidal automorphic representations of $\mathbf{PGL}_4$ and $\mathbf{PGL}_2$, which are unramified everywhere. Suppose that $L(\frac{3}{2},\pi \otimes \sigma,\wedge^2\otimes\mathrm{std}_2)\neq 0$. Then,
\[L( \tfrac{1}{2}, \pi \otimes \sigma, \rho_{\mathbf{X}}) \neq 0 \Longleftrightarrow {\mathcal{P}_{ \GL_2 \times \GL_2}}|_{\pi \otimes \sigma} \not \equiv 0.\]
\end{corollary} 
     
\begin{proof}
The implication $(\Longrightarrow)$ follows from Corollary \ref{Main:Corollary:Spherical:1}(2). 

Now, assume first that there exist cusp forms $\varphi,\Psi$ such that $\mathcal{P}_{ \GL_2 \times \GL_2}(\varphi,\Psi)\neq 0$. By \cite[Corollary 1.3(6)]{Pollack:Wan:Zydor}, we have  \[L(\tfrac{1}{2},\pi \otimes \sigma,\wedge^2 \otimes \mathrm{std}_2)\neq 0.\] 
Moreover, by Lemma \ref{Shalika:Equals:Period}, we have that $\Theta^0(\pi)\neq 0$. Thus, Proposition \ref{Watanabe:Theta:Main}, together with the discussion in Corollary \ref{Main:Corollary:Spherical:1}, implies that \[L( \tfrac{1}{2},\pi,\mathrm{std}_4)L( \tfrac{1}{2},\pi,\mathrm{std}_4^{\vee})\neq 0,\] which proves the implication $(\Longleftarrow)$. 
\end{proof}

\section{The pole of the degree 12 $L$-function}

Let $\pi$ be a globally generic cuspidal automorphic representation of either $\GL_4(\A)$ or $\GU_{2,2}(\A)$ and let $\sigma$ be a cuspidal automorphic representation of $\GL_2(\A)$. We suppose that both representations are unramified outside a finite set $S$ of places of $F$. In this section, we study the analytic properties of the partial $L$-function $L^S(s,\pi \otimes \sigma, \wedge^2 \otimes  \mathrm{std}_2)$. We start with the following.

\begin{proposition}\label{Proposition:Poles:L}
    The $L$-function $L^S(s,\pi \otimes \sigma, \wedge^2 \otimes  \mathrm{std}_2)$ is entire unless $\omega_{\pi}\omega_{\sigma}\neq 1$ and $(\omega_{\pi}\omega_{\sigma})^2 = 1$. In this case, it can have at most a simple pole at $s = 1$. 
\end{proposition}

\begin{proof}
The proof relies on the analysis of the poles of the zeta integral $Z(s,\varphi,f_s)$. As we discussed in \S\ref{Section:Klingen:Eisenstein}, the Klingen Eisenstein series $E^*_{P_{(1,2)}}(g,s, f_s)$ is entire unless $\omega_{\pi}\omega_{\sigma}\neq 1$ and $(\omega_{\pi}\omega_{\sigma})^2 = 1$, in which case it can have a simple pole at $s=1$. Therefore,  $Z(s,\varphi,f_s)$ can only have a pole at $s = 1$ provided that $\omega_{\pi}\omega_{\sigma}\neq 1$ and $(\omega_{\pi}\omega_{\sigma})^2 = 1$. The result then follows from Theorem \ref{zetaintegralfinalthm} and Proposition \ref{non:vanishing:and:abs:conv:zeta:bad}.
\end{proof}

\begin{remark}
   We note that Proposition \ref{Proposition:Poles:L} can be equivalently proved by using the pullback formula (see Corollary \ref{CorPullbackZeta}) and analyzing the zeta integral at points where the normalized Siegel Eisenstein series $E_{P_{(3,0)}}^*(g, s,\phi_s)$ can have poles. As a sanity check, we provide a sketch of the proof using this alternative method. It is well known that  $E_{P_{(3,0)}}^*(g, s,\phi_s)$ is entire if $(\omega_{\pi}\omega_{\sigma})^2 \ne 1$ and it has simple poles at $s=1$, resp. $s=1$ and $2$, when  $\omega_{\pi}\omega_{\sigma} \ne 1$ and $(\omega_{\pi}\omega_{\sigma})^2=1$, resp. $\omega_{\pi}\omega_{\sigma} = 1$. We now explain how to deduce in the latter case that the zeta integral is entire everywhere using Corollary \ref{CorPullbackZeta}. Assume $\omega_{\pi}\omega_{\sigma} = 1$; the residue at $s=2$ of $E_{P_{(3,0)}}^*(g, s,\phi_s)$ is constant, hence the residue at $s=2$ of the right hand side of \eqref{eq_Cor_Pullback}, hence of $Z(s,\varphi,f_s)$, is zero because the period of the cusp form $\varphi$ over $\Sp_4$ is zero. At $s=1$, the residue of $E_{P_{(3,0)}}^*(g, s,\phi_s)$ is proportional to a value of a degenerate Klingen Eisenstein series $E_{P_{(1,4)}}(g, s_0,\phi_{s_0}')$  (see \cite[Proposition 1.10]{IkedaTriple}), where $P_{(1,4)}$ is the standard parabolic of $\GSp_6$ with Levi $\GL_1 \times \GSp_4$, for some $s_0 \in \C$ and section $\phi_{s_0}'$. Then the residue of \eqref{eq_Cor_Pullback} at $s=1$ is zero because a simple unfolding argument shows that 
   $$\int_{Z_{\GSp_6}(\A)\backslash [\GL_2\boxtimes \GSp_4]} E_{P_{(1,4)}}(\iota( g_1  , g_2), s,\phi_{s}') \Psi^w( g_1  ) \varphi(g_2)\, dg_1\,dg_2$$
   is identically zero. 
\end{remark}

In what follows, we assume that  $\omega_{\pi}\omega_{\sigma}\neq 1$ and $(\omega_{\pi}\omega_{\sigma})^2 = 1$ and give a period condition for the vanishing of the residue at $s = 1$ of the partial $L$-function $L^S(s,\pi \otimes \sigma, \wedge^2 \otimes  \mathrm{std}_2)$. We follow a similar strategy to that of \S\ref{Section:Period:Central:Value}: we first use Corollary \ref{CorPullbackZeta} and then apply the relevant Siegel--Weil formulae.

\subsection{An application of the Siegel-Weil formula}\label{SW:FirstTerm:Pole}

For simplicity, denote by $\chi$ the quadratic Hecke character $ \omega_\pi\omega_\sigma$ on $F^\times \backslash \A^\times$. Associated to $\chi$, there is a unique quadratic field extension
$F_{\chi}/F$. When endowed with the quadratic form  $N_{F_\chi/F}$, it  defines a quadratic space over $F$ of dimension $2$. 

Let $\tilde{I}_{P_{(3,0)}}(s,\chi)$ be the restriction to $\Sp_6(\A)$ of the induction representation $I_{P_{(3,0)}}(s,\chi)$. Let $V_\chi$ be any quadratic space over  $F$ of dimension $2$ with character $\chi$. By \cite{KudlaRallis}, we have a map 
\[\lambda_{V_\chi} :\mathcal{S}(V_\chi(\A)^3)\to  \tilde{I}_{P_{(3,0)}}(0,\chi).\] Its image is denoted by $\tilde{\Pi}(V_\chi)$.

For the next Theorem, let $M^*(1,\chi): \tilde{I}_{P_{(3,0)}}(1,\chi)\to \tilde{I}_{P_{(3,0)}}(0,\chi)$ be the (normalized) intertwining operator defined as in \cite[p. 4]{KudlaRallis}.
\begin{theorem}\label{Siegel-Weil}
     Let $\phi_s\in I_{P_{(3,0)}}\left(s,\chi\right)$ be a section such that $M^*(1,\chi){\phi_1}|_{\Sp_6(\A)}\in\tilde{\Pi}(V_\chi)$. For any $g\in\Sp_6(\A)$, we have
     \[\mathrm{Res}_{s = 1}E_{6}\left(g,s,\phi_s\right) = c_0\int_{[\mathbf{O}_{{V_\chi}}]}\theta(g,h,\varphi_{V_\chi})dh,\]
     where $\varphi_{V_\chi}\in \mathcal{S}(V_\chi(\A)^3)$ is so that $\lambda_{V_\chi}(\varphi_{V_\chi}) = M^*(1,\chi){\phi_1}|_{\Sp_6(\A)}$ and $c_0\in \C^{\times}$.
\end{theorem}
\begin{proof}
    This is \cite[Corollary 6.3]{KudlaRallis} for the quadratic space  $V_\chi \oplus \mathbb{H}^2$, which has complementary space $V_\chi$ in the notation of \textit{loc.cit.}.
\end{proof}

\begin{remark}
    In contrast to \S\ref{Section:Period:Central:Value}, we cannot extend the formula of Theorem \ref{Siegel-Weil} to the case of similitude groups. The reason being that the translates of $\lambda_{V_\chi}(\varphi_{V_\chi})$ by elements in $\GSp_6^+(\A)$ might not come from any $\mathcal{S}(V(\A)^3)$, where $V$ is an anisotropic quadratic space of dimension $2$ over $F$ of character $\chi$. 
\end{remark}

\begin{proposition}\label{First:Term:Application:pole}
If ${\rm Res}_{s=1}Z(s,\varphi,f^\Psi_{\phi_s}) \ne 0$, then there exists $\varphi_{V_\chi}\in \mathcal{S}(V_\chi(\A)^3)$ such that 
 \[ \int_{[\mathbf{O}_{V_\chi}]}\int_{[\mathbf{PSL}_2\times \mathbf{PSp}_4]}\theta(\iota(g_1,g_2),h,\varphi_{V_\chi})\Psi^{w}(g_1)\varphi(g_2)dg_1dg_2 \ne 0.\]
\end{proposition}
\begin{proof}
   If ${\rm Res}_{s=1}Z(s,\varphi,f^\Psi_{\phi_s}) \ne 0$, then 
$${\rm Res}_{s=1}  \int_{ [\mathbf{PSp}_4]} E^*_{P_{(1,2)}}(g,s, f^\Psi_{\phi_s})\varphi(g)dg \ne 0.$$
Then, by Garrett's pullback formula and the Siegel-Weil formula of Theorem \ref{Siegel-Weil}, the non-vanishing of  ${\rm Res}_{s=1}Z(s,\varphi,f^\Psi_{\phi_s}) \ne 0$ implies that 
$$\sum_{V_\chi} \int_{[\mathbf{O}_{V_\chi}]}\int_{[\mathbf{PSL}_2\times \mathbf{PSp}_4]}\theta(\iota(g_1,g_2),h,\varphi_{V_\chi})\Psi^{w}(g_1)\varphi(g_2)dg_1dg_2 \ne 0,$$
where, we have used that, by \cite[Theorem 4.9]{KudlaRallis},  $$M^*(1,\chi)\phi_{1}\mapsto \mathrm{Res}_{s = 1}E^*_{P_{(3,0)}}(\iota(g_1,g_2),s,M^*(s,\chi)\phi_s)$$ factors through $\bigoplus_{V_\chi} \tilde{\Pi}(V_\chi)$, and we have denoted $\varphi_{V_\chi} = \mathrm{pr}_{\tilde{\Pi}(V)}(M^*(1,\chi)\phi_{1})$.
\end{proof}

Applying Theorem \ref{zetaintegralfinalthm} and Proposition \ref{First:Term:Application:pole}, we get the following.

\begin{corollary}\label{cor_pole_SW1} 
Let $\phi_s\in I_{P_{(3,0)}}\left(s,\chi\right)$ be a section such that $M^*(1,\chi){\phi_1}|_{\Sp_6(\A)}\in\tilde{\Pi}(V_\chi)$. Suppose that there exist data $W_{\pi_v}$ and $W_{\sigma_v,f_{\phi_{1/2,v}}^{\Psi_v}}$ such that $\prod_{v \in S} Z\left(1,W_{\pi_v}, W_{\sigma_v,f_{\phi_{1/2,v}}^{\Psi_v}} \right) \ne 0,$ then \[{\rm Res}_{s=1}L^S(s, \pi \otimes \sigma, \wedge^2 \otimes {\rm std}_2) \ne 0 \Longrightarrow  \int_{[\mathbf{O}_{V_\chi}]}\int_{[\mathbf{PSL}_2\times \mathbf{PSp}_4]}\theta(\iota(g_1,g_2),h,\varphi_{V_\chi})\Psi^{w}(g_1)\varphi(g_2)dg_1dg_2 \ne 0.\]
\end{corollary}

\bibliographystyle{acm}

\bibliography{V2}

@article{CauchiGutiMVI,
  title={A two variable {R}ankin{--}{Selberg} integral for {$\mathrm{GU}_{2, 2}$} and the degree 5 {$L$}-function of {$\mathrm{GSp}_4$}},
  author={Cauchi, Antonio and Gutierrez Terradillos, Armando},
  journal={Mathematische {Z}eitschrift},
  volume={308},
  pages={29},
  year={2024}
}

@article{IshiiMultivariate,
  title={The computation of multivariate archimedean zeta integrals on {${\mathrm{GL}_2}\times \mathrm{GSp}_4$} and {$\mathrm{GL}_4$}},
  author={Ishii, Taku},
  journal={Proceedings of the American Mathematical Society},
  volume={147},
  number={1},
  pages={103--114},
  year={2019}
}

@article{Stade,
  title={Archimedean {L}-factors on {$\mathrm{GL} (n)\times \mathrm{GL} (n)$} and generalized {B}arnes integrals},
  author={Stade, Eric},
  journal={Israel Journal of Mathematics},
  volume={127},
  pages={201--219},
  year={2002},
  publisher={Springer}
}

@book{HiranoIshiiMiyazakiGL3,
  title={Archimedean zeta integrals for {${\mathrm{GL}_3}\times \mathrm{GL}_2$}},
  author={Hirano, Miki and Ishii, Taku and Miyazaki, Tadashi},
  volume={278},
  number={1366},
  year={2022},
  publisher={American Mathematical Society}
}

@article{HiranoIshiiMiyazakiGL4,
  title={Whittaker functions on {$\mathrm{GL} (4, \mathbb{R})$} and archimedean {Bump--Friedberg} integrals},
  author={Hirano, Miki and Ishii, Taku and Miyazaki, Tadashi},
  journal={arXiv preprint arXiv:2409.00401},
  year={2024}
}

@inproceedings{Novodvorsky,
  title={Automorphic {L}-functions for symplectic group {$\mathrm{GSp}(4)$}},
  author={Novodvorsky, Mark E},
  booktitle={Automorphic forms, representations and L-functions (Proc. Sympos. Pure Math., Oregon State Univ., Corvallis, Ore., 1977), Part},
  volume={2},
  pages={87--95},
  year={1979}
}

@article {RobertsTheta,
    AUTHOR = {Roberts, Brooks},
     TITLE = {The theta correspondence for similitudes},
   JOURNAL = {Israel J. Math.},
  FJOURNAL = {Israel Journal of Mathematics},
    VOLUME = {94},
      YEAR = {1996},
     PAGES = {285--317},
       DOI = {10.1007/BF02762709},
}

@incollection {SeeSawKudla,
    AUTHOR = {Kudla, Stephen S.},
     TITLE = {Seesaw dual reductive pairs},
 BOOKTITLE = {Automorphic forms of several variables ({K}atata, 1983)},
    SERIES = {Progr. Math.},
    VOLUME = {46},
     PAGES = {244--268},
 PUBLISHER = {Birkh\"auser Boston, Boston, MA},
      YEAR = {1984},
      ISBN = {0-8176-3172-0},
}

@book{GodementJacquet,
 author = {Godement, Roger and Jacquet, Herv{\'e}},
 title = {Zeta functions of simple algebras},
 fseries = {Lecture Notes in Mathematics},
 series = {Lect. Notes Math.},
 issn = {0075-8434},
 volume = {260},
 year = {1972},
 publisher = {Springer, Cham},
 language = {English},
 doi = {10.1007/BFb0070263},
 zbMATH = {3385758},
 Zbl = {0244.12011}
}

@Article{KimK,
 Author = {Kim, Henry H. and Krishnamurthy, Muthukrishnan},
 Title = {Twisted exterior square lift from {{\(\text{GU}(2,2)_{E/F}\)}} to {{\(\text{GL}_6/F\)}}},
 FJournal = {Journal of the Ramanujan Mathematical Society},
 Journal = {J. Ramanujan Math. Soc.},
 ISSN = {0970-1249},
 Volume = {23},
 Number = {4},
 Pages = {381--412},
 Year = {2008},
 Language = {English},
 URL = {www.math.uiowa.edu/~mkrishna/Research/ki-mu-final-rev.pdf},
 zbMATH = {5532120},
 Zbl = {1220.11070}
}

@Article{FurusawaMorimoto,
 Author = {Furusawa, Masaaki and Morimoto, Kazuki},
 Title = {Shalika periods on {{\(\mathrm{GU}(2,2)\)}}},
 FJournal = {Proceedings of the American Mathematical Society},
 Journal = {Proc. Am. Math. Soc.},
 ISSN = {0002-9939},
 Volume = {141},
 Number = {12},
 Pages = {4125--4137},
 Year = {2013},
 Language = {English},
 DOI = {10.1090/S0002-9939-2013-11690-4},
 zbMATH = {6218142},
 Zbl = {1354.11038}
}

@inproceedings{KnappLLCA,
  title={Local {L}anglands correspondence: the {A}rchimedean case},
  author={Knapp, Anthony W},
  booktitle={Proc. Sympos. Pure Math},
  volume={55},
  number={2},
  year={1994}
}

@article{Chen-Gan,
  title={Unitary {F}riedberg-{J}acquet periods},
  author={Chen, Rui and Gan, Wee Teck},
  journal={arXiv preprint arXiv:2108.04064},
  year={2021}
}

@article {IchinoTriple,
    AUTHOR = {Ichino, Atsushi},
     TITLE = {Trilinear forms and the central values of triple product
              {$L$}-functions},
   JOURNAL = {Duke Math. J.},
  FJOURNAL = {Duke Mathematical Journal},
    VOLUME = {145},
      YEAR = {2008},
    NUMBER = {2},
     PAGES = {281--307},
      ISSN = {0012-7094,1547-7398},
   MRCLASS = {11F67 (11F70)},
}

@article{IshiiStade,
  title={New formulas for {W}hittaker functions on {$\mathrm{GL} (n, \mathbf{R})$}},
  author={Ishii, Taku and Stade, Eric},
  journal={Journal of functional Analysis},
  volume={244},
  number={1},
  pages={289--314},
  year={2007},
  publisher={Elsevier}
}

@article{GinzburgRallis,
  title={The exterior cube {$L$}-function for {$\mathrm{GL}(6)$}},
  author={Ginzburg, David and Rallis, Stephen},
  journal={Compositio Mathematica},
  volume={123},
  number={3},
  pages={243--272},
  year={2000},
  publisher={London Mathematical Society}
}

@Article{GelbartJacquet,
 Author = {Gelbart, Stephen and Jacquet, Herv{\'e}},
 Title = {A relation between automorphic representations of {{\(\text{GL}(2)\)}} and {{\(\text{GL}(3)\)}}},
 FJournal = {Annales Scientifiques de l'{\'E}cole Normale Sup{\'e}rieure. Quatri{\`e}me S{\'e}rie},
 Journal = {Ann. Sci. {\'E}c. Norm. Sup{\'e}r. (4)},
 ISSN = {0012-9593},
 Volume = {11},
 Number = {4},
 Pages = {471--542},
 Year = {1978},
 Language = {English},
 DOI = {10.24033/asens.1355},
 zbMATH = {3630846},
 Zbl = {0406.10022}
}

@misc{CauchiGutiArchGU,
      title={Archimedean integrals for {$\mathrm{GU}_{2,2} \times \mathrm{GL}_2$}}, 
      author={Cauchi, Antonio and Gutierrez Terradillos, Armando},
     note={In preparation}
}

@article {YamanaSingular,
    AUTHOR = {Yamana, Shunsuke},
     TITLE = {On the {S}iegel-{W}eil formula: the case of singular forms},
   JOURNAL = {Compos. Math.},
  FJOURNAL = {Compositio Mathematica},
    VOLUME = {147},
      YEAR = {2011},
    NUMBER = {4},
     PAGES = {1003--1021},
      ISSN = {0010-437X,1570-5846},
   MRCLASS = {11F27 (11F70 11M36 22E55)},
       DOI = {10.1112/S0010437X11005379},
       URL = {https://doi.org/10.1112/S0010437X11005379},
}

@Article{WaldsToricPeriods,
 Author = {Waldspurger, J.-L.},
 Title = {On the values of certain automorphic {{\(L\)}}-functions at their center of symmetry},
 FJournal = {Compositio Mathematica},
 Journal = {Compos. Math.},
 ISSN = {0010-437X},
 Volume = {54},
 Pages = {173--242},
 Year = {1985},
 Language = {French},
 zbMATH = {3904677},
 Zbl = {0567.10021}
}

@Book{BumpBook,
 Author = {Bump, Daniel},
 Title = {Automorphic forms and representations},
 FSeries = {Cambridge Studies in Advanced Mathematics},
 Series = {Camb. Stud. Adv. Math.},
 Volume = {55},
 ISBN = {0-521-55098-X},
 Year = {1997},
 Publisher = {Cambridge: Cambridge University Press},
 Language = {English},
 zbMATH = {1016459},
 Zbl = {0868.11022}
}

@Book{ExplicitGSR,
 Author = {Gelbart, Stephen and Piatetski-Shapiro, Ilya I. and Rallis, Stephen},
 Title = {Explicit constructions of automorphic {{\(L\)}}-functions},
 FSeries = {Lecture Notes in Mathematics},
 Series = {Lect. Notes Math.},
 ISSN = {0075-8434},
 Volume = {1254},
 ISBN = {3-540-17848-1},
 Year = {1987},
 Publisher = {Springer, Cham},
 Language = {English},
 zbMATH = {3989450},
 Zbl = {0612.10022}
}

@Article{Ton-That,
 Author = {Tuong Ton-That},
 Title = {Lie group representations and harmonic polynomials of a matrix variable},
 FJournal = {Transactions of the American Mathematical Society},
 Journal = {Trans. Am. Math. Soc.},
 ISSN = {0002-9947},
 Volume = {216},
 Pages = {1--46},
 Year = {1976},
 Language = {English},
 DOI = {10.2307/1997683},
 zbMATH = {3450509},
 Zbl = {0287.22014}
}

@InCollection{BlasiusRogawski,
 Author = {Blasius, Don and Rogawski, Jonathan D.},
 Title = {Zeta functions of {Shimura} varieties},
 BookTitle = {Motives. Proceedings of the summer research conference on motives, held at the University of Washington, Seattle, WA, USA, July 20-August 2, 1991},
 ISBN = {0-8218-1637-3; 0-8218-1635-7},
 Pages = {525--571},
 Year = {1994},
 Publisher = {Providence, RI: American Mathematical Society},
 Language = {English},
 zbMATH = {596379},
 Zbl = {0827.11033}
}

@Article{CasselmanShalikaII,
 Author = {Casselman, William and Shalika, J.},
 Title = {The unramified principal series of p-adic groups. {II}: {The} {Whittaker} function},
 FJournal = {Compositio Mathematica},
 Journal = {Compos. Math.},
 ISSN = {0010-437X},
 Volume = {41},
 Pages = {207--231},
 Year = {1980},
 Language = {English},
 zbMATH = {3739835},
 Zbl = {0472.22005}
}

@Article{GanHundley,
 Author = {Gan, Wee Teck and Hundley, Joseph},
 Title = {The spin {{\(L\)}}-function of quasi-split {{\(D_4\)}}},
 FJournal = {IMRP. International Mathematics Research Papers},
 Journal = {IMRP, Int. Math. Res. Pap.},
 ISSN = {1687-3017},
 Volume = {2006},
 Number = {14},
 Pages = {74},
 Note = {Id/No 68213},
 Year = {2006},
 Language = {English},
 DOI = {10.1155/IMRP/2006/68213},
 zbMATH = {5136937},
 Zbl = {1201.11052}
}

@InCollection{GarrettIntegralRepEisenstein,
 Author = {Garrett, Paul B.},
 Title = {Integral representations of {Eisenstein} series and {L}-functions},
 Year = {1989},
 BookTitle = {Number theory, trace formulas and discrete groups, {Symp}. in {Honor} of {Atle} {Selberg}, {Oslo}/{Norway} 1987},
Pages = {241--264}
}

@Article{KudlaRallis,
 Author = {Kudla, Stephen S. and Rallis, Stephen},
 Title = {A regularized {Siegel}-{Weil} formula: {The} first term identity},
 FJournal = {Annals of Mathematics. Second Series},
 Journal = {Ann. Math. (2)},
 ISSN = {0003-486X},
 Volume = {140},
 Number = {1},
 Pages = {1--80},
 Year = {1994},
 Language = {English},
 DOI = {10.2307/2118540},
 zbMATH = {709808},
 Zbl = {0818.11024}
}

@InCollection{KudlaHarrisJacquet,
 Author = {Harris, Michael and Kudla, Stephen S.},
 Title = {On a conjecture of {Jacquet}},
 BookTitle = {Contributions to automorphic forms, geometry, and number theory. Papers from the conference in honor of Joseph Shalika on the occasion of his 60th birthday, Johns Hopkins University, Baltimore, MD, USA, May 14--17, 2002},
 ISBN = {0-8018-7860-8},
 Pages = {355--371},
 Year = {2004},
 Publisher = {Baltimore, MD: Johns Hopkins University Press},
 Language = {English},
 URL = {muse.jhu.edu/journals/american_journal_of_mathematics/info/docs/hida_pdfs/hida14.pdf},
 zbMATH = {2149785},
 Zbl = {1080.11039}
}

@Article{GanSavinSW,
 Author = {Gan, Wee Teck and Savin, Gordan},
 Title = {An exceptional {Siegel}-{Weil} formula and poles of the spin {{\(L\)}}-function of {{\(\mathrm{PGSp}_6 \)}}},
 FJournal = {Compositio Mathematica},
 Journal = {Compos. Math.},
 ISSN = {0010-437X},
 Volume = {156},
 Number = {6},
 Pages = {1231--1261},
 Year = {2020},
 Language = {English},
 DOI = {10.1112/S0010437X20007186},
 zbMATH = {7321635},
 Zbl = {1476.11076}
}

@Article{GanGurevichCAP,
 Author = {Gan, Wee Teck and Gurevich, Nadya},
 Title = {{CAP} representations of {{\(G_{2}\)}} and the spin {{\(L\)}}-function of {{\(\mathrm{PGSp}_6\)}}},
 FJournal = {Israel Journal of Mathematics},
 Journal = {Isr. J. Math.},
 ISSN = {0021-2172},
 Volume = {170},
 Pages = {1--52},
 Year = {2009},
 Language = {English},
 DOI = {10.1007/s11856-009-0018-9},
 zbMATH = {5601705},
 Zbl = {1177.22010}
}

@Misc{JacquetShalikaExterior,
 Author = {Jacquet, Herv{\'e} and Shalika, Joseph.},
 Title = {Exterior square {{\(L\)}}-functions},
 Year = {1990},
 Language = {English},
 HowPublished = {Automorphic forms, {Shimura} varieties, and {L}-functions. {Vol}. {II}, {Proc}. {Conf}., {Ann} {Arbor}/{MI} ({USA}) 1988, {Perspect}. {Math}. 11, 143-226 (1990).},
 zbMATH = {4137873},
 Zbl = {0695.10025}
}

@Article{PitaleSchmidt,
 Author = {Pitale, Ameya and Saha, Abhishek and Schmidt, Ralf},
 Title = {On the standard {{\(L\)}}-function for {{\(\mathrm{GSp}_{2n}\times\mathrm{GL}_1\)}} and algebraicity of symmetric fourth {{\(L\)}}-values for {{\(\mathrm{GL}_2\)}}},
 FJournal = {Annales Math{\'e}matiques du Qu{\'e}bec},
 Journal = {Ann. Math. Qu{\'e}.},
 ISSN = {2195-4755},
 Volume = {45},
 Number = {1},
 Pages = {113--159},
 Year = {2021},
 Language = {English},
 DOI = {10.1007/s40316-020-00134-6},
 zbMATH = {7341608}
}

@Article{IchinoSW,
 Author = {Ichino, Atsushi},
 Title = {On the regularized {Siegel}-{Weil} formula},
 FJournal = {Journal f{\"u}r die Reine und Angewandte Mathematik},
 Journal = {J. Reine Angew. Math.},
 ISSN = {0075-4102},
 Volume = {539},
 Pages = {201--234},
 Year = {2001},
 Language = {English},
 DOI = {10.1515/crll.2001.076},
 zbMATH = {1653375},
 Zbl = {1022.11020}
}

@Article{IkedaTriple,
 Author = {Ikeda, Tamotsu},
 Title = {On the location of poles of the triple {L}-functions},
 FJournal = {Compositio Mathematica},
 Journal = {Compos. Math.},
 ISSN = {0010-437X},
 Volume = {83},
 Number = {2},
 Pages = {187--237},
 Year = {1992},
 Language = {English},
 zbMATH = {64245},
 Zbl = {0773.11035}
}

@Article{GanQiuTakeda,
 Author = {Gan, Wee Teck and Qiu, Yannan and Takeda, Shuichiro},
 Title = {The regularized {Siegel}-{Weil} formula (the second term identity) and the {Rallis} inner product formula},
 FJournal = {Inventiones Mathematicae},
 Journal = {Invent. Math.},
 ISSN = {0020-9910},
 Volume = {198},
 Number = {3},
 Pages = {739--831},
 Year = {2014},
 Language = {English},
 DOI = {10.1007/s00222-014-0509-0},
 zbMATH = {6379806},
 Zbl = {1320.11037}
}

@article{GrossPrasadFirst,
 author = {Gross, Benedict H. and Prasad, Dipendra},
 title = {On the decomposition of a representation of {{\(SO_ n\)}} when restricted to {{\(SO_{n-1}\)}}},
 fjournal = {Canadian Journal of Mathematics},
 journal = {Can. J. Math.},
 issn = {0008-414X},
 volume = {44},
 number = {5},
 pages = {974--1002},
 year = {1992},
 language = {English},
 doi = {10.4153/CJM-1992-060-8},
 zbMATH = {85728},
 Zbl = {0787.22018}
}

@incollection{GGPOriginal,
 author = {Gan, Wee Teck and Gross, Benedict H. and Prasad, Dipendra},
 title = {Symplectic local root numbers, central critical {{\(L\)}}-values, and restriction problems in the representation theory of classical groups},
 booktitle = {Sur les conjectures de Gross et Prasad. I},
 isbn = {978-2-85629-348-5},
 pages = {1--109},
 year = {2012},
 publisher = {Paris: Soci{\'e}t{\'e} Math{\'e}matique de France (SMF)},
 language = {English},
 zbMATH = {6154194},
 Zbl = {1280.22019}
}

@book{SakellaridisVenkatesh,
 author = {Sakellaridis, Yiannis and Venkatesh, Akshay},
 title = {Periods and harmonic analysis on spherical varieties},
 fseries = {Ast{\'e}risque},
 series = {Ast{\'e}risque},
 issn = {0303-1179},
 volume = {396},
 isbn = {978-2-85629-871-8},
 year = {2017},
 publisher = {Paris: Soci{\'e}t{\'e} Math{\'e}matique de France (SMF)},
 language = {English},
 zbMATH = {6847674},
 Zbl = {1479.22016}
}

@article{sakellaridis2013spherical,
  title={Spherical functions on spherical varieties},
  author={Sakellaridis, Yiannis},
  journal={American Journal of Mathematics},
  volume={135},
  number={5},
  pages={1291--1381},
  year={2013},
  publisher={Johns Hopkins University Press}
}

@article{MWZ,
  title={Relative {L}anglands duality for some strongly tempered spherical varieties},
  author={Mao, Zhengyu and Wan, Chen and Zhang, Lei},
  journal={Inventiones mathematicae},
  year={2025}
}

@article{BZSV,
  title={Relative {L}anglands duality},
  author={Ben-Zvi, David and Sakellaridis, Yiannis and Venkatesh, Akshay},
  journal={arXiv preprint arXiv:2409.04677},
  year={2024}
}

@article{CauchiGuti,
  title={Spherical {S}halika models on {$\mathrm{PGU}_{2,2}$} and the theta correspondence for {$(\mathrm{PGSp}_{4}, \mathrm{PGU}_{2,2})$}},
  author={Cauchi, Antonio and Gutierrez Terradillos, Armando},
  journal={Forum Mathematicum},
  year={2025}
}

@article{Shimizu:JL,
 author = {Shimizu, Hideo},
 title = {Theta series and automorphic forms on {{\(\text{GL}_2\)}}},
 fjournal = {Journal of the Mathematical Society of Japan},
 journal = {J. Math. Soc. Japan},
 issn = {0025-5645},
 volume = {24},
 pages = {638--683},
 year = {1972},
 language = {English},
 doi = {10.2969/jmsj/02440638},
 zbMATH = {3380700},
 Zbl = {0241.10016}
}

@article{Watanabe:Global:Theta,
  AUTHOR = {Watanabe, Takao},
     TITLE = {Global theta liftings of general linear groups},
   JOURNAL = {J. Math. Sci. Univ. Tokyo},
  FJOURNAL = {The University of Tokyo. Journal of Mathematical Sciences},
    VOLUME = {3},
      YEAR = {1996},
    NUMBER = {3},
     PAGES = {699--711},
      ISSN = {1340-5705},
   MRCLASS = {11F70 (11F27 22E55)},
}

@article{Conservation:Relation:Theta,
 author = {Sun, Binyong and Zhu, Chen-Bo},
 title = {Conservation relations for local theta correspondence},
 fjournal = {Journal of the American Mathematical Society},
 journal = {J. Am. Math. Soc.},
 issn = {0894-0347},
 volume = {28},
 number = {4},
 pages = {939--983},
 year = {2015},
 language = {English},
 doi = {10.1090/S0894-0347-2014-00817-1},
 zbMATH = {6460476},
 Zbl = {1321.22017}
}

@article{Kudla:Integral:Borcherds,
 author = {Kudla, Stephen S.},
 title = {Integrals of {Borcherds} forms},
 fjournal = {Compositio Mathematica},
 journal = {Compos. Math.},
 issn = {0010-437X},
 volume = {137},
 number = {3},
 pages = {293--349},
 year = {2003},
 language = {English},
 doi = {10.1023/A:1024127100993},
 zbMATH = {1972162},
 Zbl = {1046.11027}
}

@book{Moeglin:Vigneras:Waldspurger,
 author = {Moeglin, Colette and Vign{\'e}ras, Marie-France and Waldspurger, Jean-Loup},
 title = {Correspondances de {Howe} sur un corps {{\(p\)}}-adique},
 fseries = {Lecture Notes in Mathematics},
 series = {Lect. Notes Math.},
 issn = {0075-8434},
 volume = {1291},
 isbn = {3-540-18699-9},
 year = {1987},
 publisher = {Springer, Cham},
 language = {French},
 doi = {10.1007/BFb0082712},
 zbMATH = {4046035},
 Zbl = {0642.22002}
}

@incollection{Prasad:Theta:Survey,
 author = {Prasad, Dipendra},
 title = {Weil representation, {Howe} duality, and the theta correspondence},
 booktitle = {Theta functions: from the classical to the modern},
 isbn = {0-8218-6997-3},
 pages = {105--127},
 year = {1993},
 publisher = {Providence, RI: American Mathematical Society},
 language = {English},
 zbMATH = {431570},
 Zbl = {0820.11027}
}

@article{Huang:He,
 author = {Huang, Huajun and He, Hongyu},
 title = {Symmetric subgroup actions on isotropic {Grassmannians}},
 fjournal = {Journal of Algebra},
 journal = {J. Algebra},
 issn = {0021-8693},
 volume = {337},
 number = {1},
 pages = {141--168},
 year = {2011},
 language = {English},
 doi = {10.1016/j.jalgebra.2011.04.022},
 url = {digitalcommons.lsu.edu/cgi/viewcontent.cgi?article=1474&context=mathematics_pubs},
 zbMATH = {5996435},
 Zbl = {1235.14042}
}

@book{Wan:Zhang:Spherical:Periods,
 author = {Wan, Chen and Zhang, Lei},
 title = {Periods of automorphic forms associated to strongly tempered spherical varieties},
 fseries = {Memoirs of the American Mathematical Society},
 series = {Mem. Am. Math. Soc.},
 issn = {0065-9266},
 volume = {1590},
 isbn = {978-1-4704-7656-4; 978-1-4704-8457-6},
 year = {2025},
 publisher = {Providence, RI: American Mathematical Society (AMS)},
 language = {English},
 doi = {10.1090/memo/1590},
 zbMATH = {8096342}
}

@article{Pollack:Wan:Zydor,
 author = {Pollack, Aaron and Wan, Chen and Zydor, Micha{\l}},
 title = {On the residue method for period integrals},
 fjournal = {Duke Mathematical Journal},
 journal = {Duke Math. J.},
 issn = {0012-7094},
 volume = {170},
 number = {7},
 pages = {1457--1515},
 year = {2021},
 language = {English},
 doi = {10.1215/00127094-2020-0078},
 zbMATH = {7369856},
 Zbl = {1473.11104}
}

@article{Adams:Vogan,
 author = {Adams, Jeffrey and Vogan, David A. jun.},
 title = {Contragredient representations and characterizing the local {Langlands} correspondence},
 fjournal = {American Journal of Mathematics},
 journal = {Am. J. Math.},
 issn = {0002-9327},
 volume = {138},
 number = {3},
 pages = {657--682},
 year = {2016},
 language = {English},
 doi = {10.1353/ajm.2016.0024},
 url = {hdl.handle.net/1721.1/116073},
 zbMATH = {6603603},
 Zbl = {1354.11039}
}

@misc{Branching:Law,
 author = {Rajan, C. S. and Shrivastava, Sagar},
 title = {Relative {Weyl} {Character} formula, {Relative} {Pieri} formulas and {Branching} rules for {Classical} groups},
 year = {2023},
 howpublished = {Preprint, {arXiv}:2310.00323 [math.{RT}] (2023)},
 url = {https://arxiv.org/abs/2310.00323},
 arXiv = {arXiv:2310.00323}
}

@article{Soo:Bo,
 author = {Lee, Soo Teck and Zhu, Chen-Bo},
 title = {Degenerate principal series and local theta correspondence. {III}: {The} case of complex groups},
 fjournal = {Journal of Algebra},
 journal = {J. Algebra},
 issn = {0021-8693},
 volume = {319},
 number = {1},
 pages = {336--359},
 year = {2008},
 language = {English},
 doi = {10.1016/j.jalgebra.2007.06.018},
 zbMATH = {5235129},
 Zbl = {1139.22004}
}

@misc{Kudla:Rallis:lost,
 author = {Kudla, Stephen S. and Rallis, Stephen},
 title = {Poles of {Eisenstein} series and {L}-functions},
 year = {1990},
 language = {English},
 howpublished = {Festschrift in honor of {I}.{I}. {Piatetski}-{Shapiro} on the occasion of his sixtieth birthday. {Pt}. {II}: {Papers} in analysis, number theory and automorphic {L}-functions, {Pap}. {Workshop} {L}-{Funct}., {Number} {Theory}, {Harmonic} {Anal}., {Tel}-{Aviv}/{Isr}. 1989, {Isr}. {Math}. {Conf}. {Proc}. 3, 81-110 (1990).},
 zbMATH = {4171016},
 Zbl = {0712.11029}
}

@incollection{GanPeriodsandTheta,
 author = {Gan, Wee Teck},
 title = {Periods and theta correspondence},
 booktitle = {Representations of reductive groups. Conference in honor of Joseph Bernstein. Representation theory and algebraic geometry, June 11--16, 2017. Weizmann Institute of Science and The Hebrew University of Jerusalem, Israel},
 isbn = {978-1-4704-4284-2; 978-1-4704-5157-8},
 pages = {113--132},
 year = {2019},
 publisher = {Providence, RI: American Mathematical Society (AMS)},
 language = {English},
 zbMATH = {7278441},
 Zbl = {1481.11046}
}

@incollection {BumpFriedberg,
    AUTHOR = {Bump, Daniel and Friedberg, Solomon},
     TITLE = {The exterior square automorphic {$L$}-functions on {${\rm
              GL}(n)$}},
 BOOKTITLE = {Festschrift in honor of {I}. {I}. {P}iatetski-{S}hapiro on the
              occasion of his sixtieth birthday, {P}art {II} ({R}amat
              {A}viv, 1989)},
    SERIES = {Israel Math. Conf. Proc.},
    VOLUME = {3},
     PAGES = {47--65},
 PUBLISHER = {Weizmann, Jerusalem},
      YEAR = {1990},
   MRCLASS = {11F55 (11F70 22E55)},
}

@article{PanYan,
  title={Rankin-{S}elberg integrals for {$\mathrm{GSpin}$} groups with application to the global {G}an-{G}ross-{P}rasad conjecture},
  author={Yan, Pan},
  journal={arXiv preprint arXiv:2508.09066},
  year={2025}
}

@book{Folland:Harmonic:Analysis,
 author = {Folland, Gerald B.},
 title = {A course in abstract harmonic analysis},
 edition = {2nd updated edition},
 fseries = {Textbooks in Mathematics},
 series = {Textb. Math.},
 isbn = {978-1-4987-2713-6; 978-1-032-92221-8; 978-1-4987-2715-0},
 year = {2016},
 publisher = {Boca Raton, FL: CRC Press},
 language = {English},
 doi = {10.1201/b19172},
 zbMATH = {6473590},
 Zbl = {1342.43001}
}

\end{document}